\documentclass{amsart}
\usepackage{amssymb,amsmath,amsthm,mathtools, geometry, amssymb, verbatim, esint, soul}
\usepackage{a4wide}
\usepackage{float}
\usepackage{graphicx}
\usepackage[utf8]{inputenc} 
\usepackage{a4wide}
\usepackage{tikz}
\usetikzlibrary{calc}
\usepackage{ulem}
\usepackage{stackengine,scalerel,graphicx}
\stackMath
\usepackage{amsfonts} 
\usepackage{latexsym}
\usepackage[font=small,format=hang,labelfont={sf,bf}]{caption}
\usepackage{epsfig}
\usepackage{subfig}
\usepackage{url}
\usepackage{varioref}
\usepackage{bm}
\usepackage{color}
\usepackage{mathrsfs}
\usepackage{latexsym}
\usepackage{bbm}
\usepackage{faktor}
\usepackage{yhmath}
\usepackage{empheq}
\usepackage{mathtools}
\usepackage{color}
\usepackage{marginnote}
\usepackage{comment}
\usepackage{cancel}
\usetikzlibrary{arrows.meta}

\allowdisplaybreaks[1]

\usepackage{amssymb,amsmath,mathrsfs, mathtools}
\usepackage[colorlinks=false]{hyperref}
\usepackage{comment}
\usepackage{diagbox}

\usepackage[colorinlistoftodos]{todonotes}
\setuptodonotes{color=green!40, fancyline, size=small}

\def\e{\varepsilon}
\def\rr{{\mathbb R}}
\def\dx{\,dx}
\def\NN{{\mathbb N}}
\def\ZZ{{\mathbb Z}}

\def\A{{\mathcal A}}
\newcommand{\I}{I}
\newcommand{\dis}{\mathcal{D}}
\newcommand{\Z}{Z}
\newcommand{\DD}{\mathscr{D}}
\newcommand{\PP}{\mathscr{P}}
\newcommand{\tempo}{\zeta}
\newcommand{\one}{\mathbb{I}}
\newcommand{\zero}{\mathbb{O}}
\newcommand{\enF}{\mathcal{F}^{\tau, \gamma}_\e}
\newcommand{\E}{\mathcal{E}}
\newcommand{\hh}{\mathcal{H}}
\newcommand{\nn}{\mathcal{N}}
\newcommand{\amin}{\alpha^{min}}
\newcommand{\amax}{\alpha^{max}}
\newcommand{\bmax}{\beta^{max}}
\newcommand{\bmin}{\beta^{min}}

\usepackage{stmaryrd}
\newcommand{\dseg}[2]{\llbracket #1, #2 \rrbracket}

\DeclareMathOperator*{\R}{\mathbb{R}}

\DeclareMathOperator*{\argmin}{arg\,min}
\DeclareMathOperator*{\diam}{diam}
\DeclareMathOperator*{\som}{\sum_{i=1}^{4}}

\def\XXint#1#2#3{{\setbox0=\hbox{$#1{#2#3}{\int}$} 
  \vcenter{\hbox{$#2#3$}}\kern-.5\wd0}}

\numberwithin{equation}{section}

\newtheorem{theorem}{Theorem}[section]
\newtheorem*{theorem*}{Theorem}
\newtheorem{lemma}[theorem]{Lemma}
\newtheorem{proposition}[theorem]{Proposition}

\theoremstyle{definition}
\newtheorem{definition}[theorem]{Definition}
\newtheorem{remark}[theorem]{Remark}
\newtheorem{example}[theorem]{Example}

\title[Crystalline Motion for the Blume-Emery-Griffiths Model: partial wetting]{Crystalline Motion of discrete interfaces in the Blume-Emery-Griffiths Model: partial wetting}

\author[M. Cicalese]{Marco Cicalese}
\address[Marco Cicalese]{Technische Universit\"at M\"unchen, Boltzmannstrasse 3, 85748 Garching, Germany	}
\email[]{marco.cicalese@tum.de}

\author[G. Fusco]{Giuliana Fusco}
\address[Giuliana Fusco]{Scuola Superiore Meridionale, via Mezzocannone 4, 80134 Napoli, Italy}
\email[]{g.fusco@ssmeridionale.it}

\author[G. Savarè]{Giovanni Savar\'e}
\address[Giovanni Savar\'e]{Technische Universit\"at M\"unchen, Boltzmannstrasse 3, 85748 Garching, Germany	}
\email[]{giovanni.savare@tum.de}

\date{\today}  

\begin{document}

\begin{abstract}
We continue the variational study of the discrete-to-continuum evolution  of lattice systems of Blume–Emery–Griffith type which model two immiscible phases in the presence of a surfactant. In our previous work \cite{CFS}, we analyzed the case of a completely wetted crystal and described how the interplay between surfactant evaporation and mass conservation leads to a transition between crystalline mean curvature flow and pinned evolutions. In the present paper, we extend the analysis to the regime of partial wetting, where the surfactant occupies only a portion of the interface. Within the minimizing-movements scheme, we rigorously derive the continuum evolution and show how partial wetting introduces a complex coupling between interfacial motion and redistribution of surfactant. The resulting evolution exhibits new features absent in the fully wetted case, including the coexistence of moving and pinned facets or the emergence and long-lived metastable states. This provides, to our knowledge, the first discrete-to-continuum variational description of partially wetted crystalline interfaces, bridging the gap between microscopic lattice models and experimentally observed surfactant-induced pinning phenomena in immiscible systems.
\end{abstract}

\maketitle
	

\tableofcontents

\section {Introduction}

Lattice systems provide a simple framework reach enough to investigate some basic mechanisms of energy-driven pattern formation and phase separation in statistical mechanics and materials science. The rigorous discrete-to-continuum limit of such microscopic systems can be obtained via $\Gamma$-convergence and often leads to coarse grained systems driven at the macroscopic scale by an anisotropic interfacial energy functional. Such functionals then govern at the macroscopic scale interface geometry and, in dynamical contexts, the effective laws of motion. In the static framework, the passage from lattice energies to crystalline perimeter functionals has been thoroughly studied (see \cite{ABCS} for an introduction to the subject) and includes classical ferromagnetic Ising-type systems firstly investigated in this setting in \cite{ABC}. In the dynamical setting, discrete minimizing-movements schemes à la Almgren-Taylor-Wang (see \cite{Almgren-Taylor, Almgren-Taylor-Wang, LS}) have been successfully combined with coarse-graining techniques to derive continuum geometric evolutions. They have revealed novel effects originating from the interaction of lattice and time scales which have eventually fostered additional studies (see for instance \cite{BCY, BGN, BMN, BSc, BST, CDGM, MN, MNPS, S} and the book \cite{BS}). Within this line of research we study the discrete Blume-Emery-Griffiths (BEG) surfactant model introduced in \cite{BEG} (see also \cite{EOT}). The BEG model is a three-state lattice spin system model in which the values $+1$ and $-1$ represent two immiscible bulk phases (e.g. water and oil) and the state $0$ models a surfactant phase which, when adsorbed on the interface, reduces the effective surface tension. On the static side, a rigorous variational coarse-graining for a broad class of surfactant lattice energies - including the BEG energy - was obtained in \cite{ACS} (see also \cite{LGGZ1, LGGZ2} for related studies on ground-state behavior), with earlier continuum treatments of surfactant effects appearing first in \cite{FMS} and then in \cite{AB} (see also \cite{CH1, CH2} for a vectorial and a non-local version of the problem). Building on the results obtained in \cite{ACS} and working within the minimizing-movements framework developed in \cite{BGN, BS} for lattice systems, in \cite{CFS} we initiated the analysis of the evolution of phases driven by the BEG model. We described the continuum limits of the associated discrete flows for several regimes related to a new dissipation mechanism described below and modelling both the case of surfactant evaporation and that of non evaporation and complete wetting. The present paper continues that program by (mostly) focusing on the regime of non evaporation of surfactant in the case of partial wetting described below.  Before introducing in more details the discrete evolution scheme, we shortly remind the variational discrete-to-continuum analysis of the static problem. Given a bounded open set $\Omega \subset \mathbb{R}^2$ with Lipschitz boundary, we define 
$\Omega_\e=\e\ZZ^2\cap \Omega$ as the set of lattice sites in $\Omega$. On $\Omega_\e$ we consider configurations $u\in {\mathcal{A}}_\e=\{u:\Omega_\e\rightarrow \{\pm1,0\}\}$. The values $u(p)=\pm 1$ encode the presence at the node $p$ of particles belonging to the two immiscible fluids phases while $u(p)=0$ represents a surfactant particle at $p$. The energy associated to a configuration $u$ in the BEG model is
\begin{equation}
\label{BEG_energy}\E^{latt}_\e(u)=\sum_{n.n.}\e^2(-u(p)u(q)+k(u(p)u(q))^2)
\end{equation}
where the sum is performed over nearest neighboring (n.n.) points of the lattice and $k>0$ is a constant. In \cite{ACS} it was shown that, as $\e\rightarrow 0$, upon identifying $u$ with its piecewise-constant interpolation on the cell of the lattice, the $\Gamma$-limit (with respect to the $L^{1}$-convergence) of $\mathcal{E}^{latt}_\e$ is finite on $L^{1}(\Omega;[-1,1])$ and it is given by the constant value $2|\Omega|(-1+k)\wedge 0$. Such a value is achieved by one of the three uniform states, namely $u=1$, $u=-1$ and $u=0$. As a consequence, on choosing $k<1$ one can select the uniform states $u=\pm1$ to be ground states and, in the spirit of development by $\Gamma$-convergence (see e.g. \cite{BT}) one can refer the energy to its minimum $m_\e:=\sum_{n.n.}\e^2(k-1)$ and further scale the functionals by $\e$ to obtain the new family of functionals
\begin{equation}
    \label{BEG_riscalata}
    \E_\e(u):=\frac{\E^{latt}_\e(u)-m_\e}{\e}=\sum_{n.n.}\e(1-u(p)u(q)-k(1-(u(p)u(q))^2)).
\end{equation}
On choosing $\frac{1}{3}<k<1$ in \eqref{BEG_riscalata}, interposing a line of phase 0 particles between two lines of particles belonging to opposite phases $-1,+1$. For such a choice of $k$ one is allowed to name the phase $0$ surfactant phase. This educated guess was confirmed in  \cite{ACS} where the $\Gamma$-limit as $\e\rightarrow 0$ (with respect to the $L^1$ convergence) of the energies $\E_\e$ was proved to be
\begin{equation}
    \label{BEG_riscalato_gamma_limite}
    \E(u)=\int_{S(u)}\psi(\nu_u)\, d\mathcal{H}^1
\end{equation}
where $u \in BV(\Omega;\{\pm1\})$, $S(u)$ denotes its jump set, namely the interface between the phases $\{u=1\}$ and $\{u=-1\}$, and $\nu_u$ is the (measure theoretic) normal 
to $S(u)$. The surface tension of the limit functional, denoted by $\psi(\nu)$, is given by $\psi(\nu)=(1-k)(3|\nu_1|\vee|\nu_2|+|\nu_1|\wedge|\nu_2|)$ and its anisotropy reflects that of the lattice geometry.\\

We now turn our attention to the discrete-to-continuum evolution of the BEG system which is defined via the iterative minimization scheme explained below. Given an initial configuration $u^\e_{0}$, one defines 
\begin{equation}\label{intro:schema}
u^\e_{j+1}\in\text{ argmin}\Bigl\{\E_\e(u^\e_{j+1})+\frac{1}{\tau}\dis_{\e, \gamma}(u^\e_{j+1}, u^\e_j)\Bigl\},
\end{equation}
where $\E_\e$ is the energy in \eqref{BEG_riscalata} and $\dis_{\e, \gamma}$ is a dissipation potential. We assume the latter to be the sum of two terms $\dis_{\e, \gamma}\coloneqq \dis_\e^1+\e^\gamma \dis^0_\e$. The term $\dis^1_\e$ is the usual Almgren-Taylor-Wang type dissipation (see \eqref{eq:dis_1}) penalizing changes of the geometry of the $j+1$ phase and already exploited in the context of discrete-to-continuum geometric evolutions in many papers after \cite{BGN}. The $\e^\gamma$-weighted term $\dis^0_\e$ penalizes variations of the surfactant mass. Therefore, the parameter $\gamma>0$ encodes the relative ease with which surfactant mass can change along the discrete evolution and, physically, it models the surfactant tendency to evaporate, i.e.,  to lose mass along the flow. As in \cite{CFS} we associate to each piecewise constant extension of the configuration $u^\e_j$ the set $A^\e_j\subset\mathbb R^2$ given by the union of the $\e$-cells where $u^\e_j=1$. We then study the piecewise-constant in time family $A^\e(t):=A^\e_{\lfloor t/\tau\rfloor}$ in the scaling regime $\tau(\e)=\zeta\e$ for some $\zeta>0$. Our main result is that $A^\e(t)$ converges uniformly locally in time, as $\e$ vanishes, to a limiting flow $A(t)$ whose qualitative behaviour for $\gamma<2$ in the non complete wetting regime is markedly different and richer than in the complete wetting case studied in \cite{CFS}. Note that, in order to reduce the complexity of the geometric description of the motion, we proved our results for initial sets that approximate those octagons which are Wulff shapes of the functional \eqref{BEG_riscalato_gamma_limite}. More precisely, given $A\subset \mathbb{R}^2$ an octagon in the sense of Definition~\ref{def:wulff_shape}, and $(A_\e)_\e$ a sequence of octagons approximating $A$ in the Hausdorff sense as $\e\to 0$, we consider the minimizing movement scheme \eqref{intro:schema} with initial configuration $u^\e_0$ such that $\{p\in\e\ZZ^2:u^\e_0(p)=1\}=A_\e\cap \e\ZZ^2$. Using the notation $I_j:=\{p\in\e\ZZ^2:u^\e_j(p)=1\}$ our aim is to describe the sequence $I_j$ starting from $I_0=A_\e\cap \e\ZZ^2$. In \cite[Lemma 5.1]{CFS} we proved that in case $\gamma<2$ the amount of surfactant (i.e. $\#\{u^\e_j=0\}$) remains constant at each time step $j$. We assume that the surfactant particles are present in the system and they do not completely wet the initial shape.
The first condition translates as $C_\e:=\#\{u^\e_0 = 0\}>0$. The second condition, also known as partial wetting condition, requires additional notation. Denoting by $P_i,\,D_i$ the lengths of the sides of $A$ parallel to the axes or diagonal, respectively (see Figure \ref{Intro0}),
 the partial wetting condition can be stated as the existence of $\lambda\in[0, +\infty)$ such that

\begin{equation}
\label{intro:partial_wetting}
    \lim_\e\e C_\e=\lambda<\sum_{i=1}^4 P_i+\frac{D_i}{\sqrt{2}}.
\end{equation}

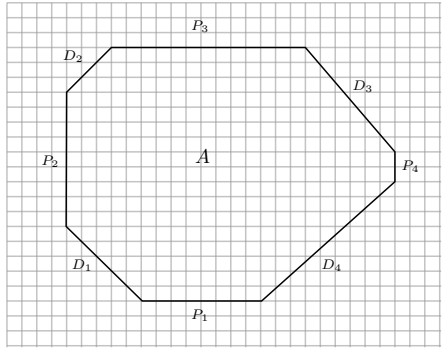
\begin{figure}
    \centering
    \resizebox{0.40\textwidth}{!}{\tikzset{every picture/.style={line width=0.75pt}} 

\begin{tikzpicture}[x=0.75pt,y=0.75pt,yscale=-1,xscale=1]

\draw  [draw opacity=0] (160.44,30.33) -- (451.44,30.33) -- (451.44,261.33) -- (160.44,261.33) -- cycle ; \draw  [color={rgb, 255:red, 155; green, 155; blue, 155 }  ,draw opacity=0.6 ] (160.44,30.33) -- (160.44,261.33)(170.44,30.33) -- (170.44,261.33)(180.44,30.33) -- (180.44,261.33)(190.44,30.33) -- (190.44,261.33)(200.44,30.33) -- (200.44,261.33)(210.44,30.33) -- (210.44,261.33)(220.44,30.33) -- (220.44,261.33)(230.44,30.33) -- (230.44,261.33)(240.44,30.33) -- (240.44,261.33)(250.44,30.33) -- (250.44,261.33)(260.44,30.33) -- (260.44,261.33)(270.44,30.33) -- (270.44,261.33)(280.44,30.33) -- (280.44,261.33)(290.44,30.33) -- (290.44,261.33)(300.44,30.33) -- (300.44,261.33)(310.44,30.33) -- (310.44,261.33)(320.44,30.33) -- (320.44,261.33)(330.44,30.33) -- (330.44,261.33)(340.44,30.33) -- (340.44,261.33)(350.44,30.33) -- (350.44,261.33)(360.44,30.33) -- (360.44,261.33)(370.44,30.33) -- (370.44,261.33)(380.44,30.33) -- (380.44,261.33)(390.44,30.33) -- (390.44,261.33)(400.44,30.33) -- (400.44,261.33)(410.44,30.33) -- (410.44,261.33)(420.44,30.33) -- (420.44,261.33)(430.44,30.33) -- (430.44,261.33)(440.44,30.33) -- (440.44,261.33)(450.44,30.33) -- (450.44,261.33) ; \draw  [color={rgb, 255:red, 155; green, 155; blue, 155 }  ,draw opacity=0.6 ] (160.44,30.33) -- (451.44,30.33)(160.44,40.33) -- (451.44,40.33)(160.44,50.33) -- (451.44,50.33)(160.44,60.33) -- (451.44,60.33)(160.44,70.33) -- (451.44,70.33)(160.44,80.33) -- (451.44,80.33)(160.44,90.33) -- (451.44,90.33)(160.44,100.33) -- (451.44,100.33)(160.44,110.33) -- (451.44,110.33)(160.44,120.33) -- (451.44,120.33)(160.44,130.33) -- (451.44,130.33)(160.44,140.33) -- (451.44,140.33)(160.44,150.33) -- (451.44,150.33)(160.44,160.33) -- (451.44,160.33)(160.44,170.33) -- (451.44,170.33)(160.44,180.33) -- (451.44,180.33)(160.44,190.33) -- (451.44,190.33)(160.44,200.33) -- (451.44,200.33)(160.44,210.33) -- (451.44,210.33)(160.44,220.33) -- (451.44,220.33)(160.44,230.33) -- (451.44,230.33)(160.44,240.33) -- (451.44,240.33)(160.44,250.33) -- (451.44,250.33)(160.44,260.33) -- (451.44,260.33) ; \draw  [color={rgb, 255:red, 155; green, 155; blue, 155 }  ,draw opacity=0.6 ]  ;
\draw    (230.44,60.33) -- (360.44,60.33) -- (420.44,130.33) -- (420.44,150.33) -- (331.11,230.33) -- (251.11,230.33) -- (200.11,180.33) -- (200.44,90.33) -- cycle ;

\draw (282.44,234.73) node [anchor=north west][inner sep=0.75pt]  [font=\scriptsize]  {$P_{1}$};
\draw (182,131.4) node [anchor=north west][inner sep=0.75pt]  [font=\scriptsize]  {$P_{2}$};
\draw (282.44,40.73) node [anchor=north west][inner sep=0.75pt]  [font=\scriptsize]  {$P_{3}$};
\draw (423.44,134.73) node [anchor=north west][inner sep=0.75pt]  [font=\scriptsize]  {$P_{4}$};
\draw (202.44,200.73) node [anchor=north west][inner sep=0.75pt]  [font=\scriptsize]  {$D_{1}$};
\draw (196.44,60.73) node [anchor=north west][inner sep=0.75pt]  [font=\scriptsize]  {$D_{2}$};
\draw (390.44,80.73) node [anchor=north west][inner sep=0.75pt]  [font=\scriptsize]  {$D_{3}$};
\draw (369.78,200.73) node [anchor=north west][inner sep=0.75pt]  [font=\scriptsize]  {$D_{4}$};
\draw (285,127.4) node [anchor=north west][inner sep=0.75pt]  [font=\normalsize]  {$A$};

\end{tikzpicture}}
    \caption{The octagon $A \subset \mathbb{R}^2$. The lengths of the sides parallel to the coordinate axes are denoted by $P_i$, while the lengths of the diagonal sides are denoted by $D_i$, with indices ordered clockwise, for $i=1,\dots, 4$.} 
    \label{Intro0}
\end{figure}

Under this assumption we show in Lemma~\ref{lemma:5_optimal_shape_not_surrounded} that, as long as the cardinality of surfactant particles $C_\e$ is sufficiently small, and in particular such that the surfactant particles at time step $j$ cannot surround $I_{j}$ (see \eqref{eq:boundary_not_surrounded}), then $I_{j+1}$ is a quasi-octagon (roughly speaking an octagon with some defects as additional cells on the sides parallel to the axes) according to Definition~\eqref{def:quasi_octagon}. The description of the evolution, which depends on $\lambda$ is qualitatively different in the case $\lambda\neq 0$ or $\lambda=0$. We describe here the more interesting case $\lambda\neq 0$, leaving the discussion in the case of negligible surfactant $\lambda=0$ to Subsection \eqref{subsec:negligible}. We qualitatively expect that the presence of surfactant particles, decreasing in an anisotropic way the surface tension, leads to an anisotropic evolution of crystalline mean curvature type. As the area of the evolving quasi-octagon decreases the conserved surfactant particles can progressively cover larger parts of the boundary affecting the motion by changing in particular its anisotropy as well as the side velocities. 
We show that, before $I_j$ is completely wetted by surfactant (after which the evolution follows an easier scheme detailed in \cite{CFS}), there exist three qualitatively different types of motion.


\begin{figure}[H]
    \centering
    \resizebox{0.50\textwidth}{!}{\input{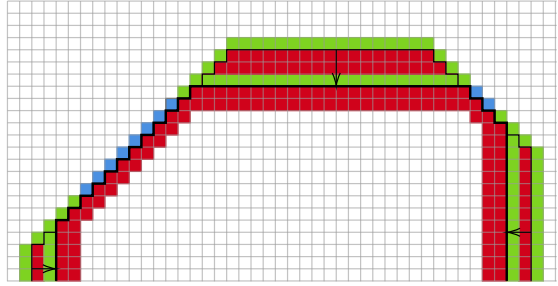}}
    \caption{Stage one of the evolution at $\varepsilon $ scale in the case $\gamma<2$ in the non complete wetting regime. We represent with a thicker black line the minimizer at time step $j+1$.}
    \label{Intro1}
\end{figure}
\noindent{\bf Stage one: ``pinning of the diagonals''.} In Proposition~\ref{prop:5.3_beginning_movement} we describe the motion of the quasi-octagon in case the amount of surfactant is not sufficient to wet its diagonals. Roughly speaking, this condition, expressed in formula \eqref{eq:6.23}, concerns the case the diagonal sides of $I_j$ are long enough to ensure that at the following time step the cells of surfactant will not be sufficient to fully wet the diagonals of $I_{j+1}$. As a result, the minimizing movement scheme is optimized by a configuration $u_{j+1}$ such that $I_{j+1}$ is a (possibly degenerate) discrete octagon whose diagonals $\DD_{j+1, i}$ for $i = 1, \dots, 4$ are pinned, i.e. they do not move with respect to the diagonals $\DD_{j, i}$ of $I_j$, so that it holds $\DD_{j+1, i}\subset \DD_{j, i}$ for all $i$ (see Figure \ref{Intro1}). In particular, only the parallel sides of $I_j$ move inwards with velocity given in \eqref{eq:movement_parallel_sides_not_covered}. \\
\begin{figure}[H]
    \centering
    \resizebox{0.50\textwidth}{!}{\input{Figure2_intro2}}
    \caption{Stage twp of the evolution at $\varepsilon $ scale in the case $\gamma<2$ in the non complete wetting regime. We represent with a thicker black line the minimizer at time step $j+1$.}
    \label{Intro2}
\end{figure}

\noindent{\bf Stage two: ``surfactant dependent side velocities''} The second stage is presented in Proposition~\ref{prop:5.21}, and describes the motion in case the amount of surfactant is at the same time not sufficient to completely wet $I_j$ (condition \eqref{eq:boundary_not_surrounded}), and too much to wet only its diagonals (condition \eqref{eq:5.32}). In this case, as long as the diagonal sides of $I_j$ are sufficiently long, the minimizing movement scheme is optimized by a configuration $u_{j+1}$ such that $I_{j+1}$ is a quasi-octangon (see Figure \ref{Intro2}) and the sides velocities are influenced by the presence of the surfactant, and are given by \eqref{eq:sides_movement_quasi_octagon} and \eqref{eq:sides_movement_quasi_octagon_2}. 

\begin{figure}[H]
    \centering
    \resizebox{0.50\textwidth}{!}{\input{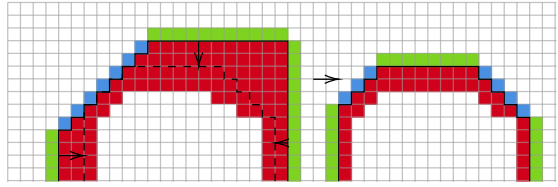}}
    \caption{Example~\eqref{ex:degenerate_octagon_evolution}: stage three of the evolution at $\varepsilon$ scale in the case $\gamma<2$ in the non complete wetting regime when $I_j$ is a  degenerate octagon. On the left we illustrate the displacement of a subset of the sides of $I_j$; the dashed line denotes the corresponding subset of the sides of $I_{j+1}$. On the right side we represent $I_{j+1}$ and the particles that are on its boundary. }
    \label{Intro3}
\end{figure}
\noindent{\bf Stage three: ``nonlocal averaging of velocities by surfactant redistribution''} The third stage is presented in Proposition~\ref{prop:sides_movement_when_set_remains_octagon}. It is an intermediate stage which occurs if both the two previous scenarios do not apply, i.e., if \eqref{eq:6.23} and \eqref{eq:5.32} are not satisfied. In this stage, roughly speaking, the surfactant cells wet the diagonal sides (up to some negligible amount of cells), i.e., denoting by $D_{j, i}$ the length of the $i$-th diagonal side $\DD_{j, i}$ of the quasi-octagon at time step $j$, the condition $\som D_{j, i}/\sqrt{2}\approx\lambda$ holds true. We show in \eqref{eq:sides_movement_when_set_remains_octagon} and \eqref{eq:sides_movement_when_set_remains_octagon_2}  that in this case, at every time step, surfactant cells wetting the longer diagonal sides at step $j$ move to wet the shorter diagonal sides at the next time step $j+1$, favouring their growth. Since in this intermediate stage the sum of the lengths of the diagonals is (approximately) constant and equal to $\lambda\sqrt{2}$ at every time step, the equation governing the sides displacements is of nonlocal type, i.e. the displacement of one side depends on the lengths of all the other sides. The overall effect is an averaging of the velocities of the different sides. 
\begin{figure}[H]
    \centering
    \resizebox{0.50\textwidth}{!}{\input{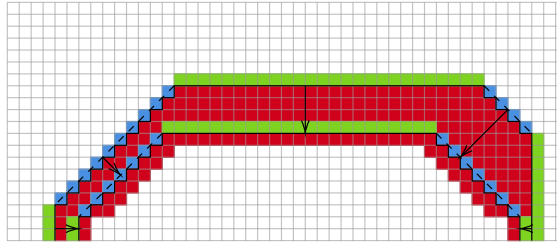}}
    \caption{Example~\ref{ex:surfactant_on_diagonals}: stage three of the evolution at $\varepsilon $ scale in the case $\gamma<2$ in the non complete wetting regime when $I_j$ is a (non degenerate) octagon with diagonals completely  wetted by surfactant particles.}
    \label{Intro4}
\end{figure}
\noindent 
In Example~\ref{ex:degenerate_octagon_evolution} we point out the interesting consequences of this stage of motion is the case of a degenerate octagon, having one of its diagonal sides of vanishing length, and the surfactant precisely covering the other diagonals. In this situation the degenerate vanishing diagonal side gradually increases its length (see Figure \ref{Intro3}) and moves with a velocity depending on the length of the neighbouring parallel sides.
Moreover, in a further example (Example~\ref{ex:surfactant_on_diagonals}) we point out that this stage is not a transient regime, but it can actually last for a macroscopic time interval. More precisely in this example we assume either that $I_{j}$ is a non degenerate octagon whose diagonals are completely wetted by surfactant (see Figure \ref{Intro4}) or it is a quai-octagon such that such that there exists only one surfactant particle that does not fit on the diagonal sides (see Figure \ref{Intro5}).

\begin{figure}[H]
    \centering
    \resizebox{0.50\textwidth}{!}{\input{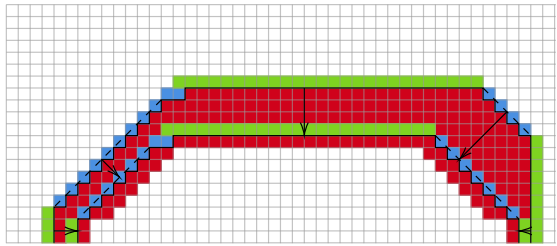}}
    \caption{Example~\ref{ex:surfactant_on_diagonals}: stage three of the evolution at $\varepsilon $ scale in the case $\gamma<2$ in the non complete wetting regime when $I_j$ is a quasi-octagon with diagonals completely  wetted by surfactant particles and there exists only one surfactant particle that does not fit on the diagonals.}
    \label{Intro5}
\end{figure}

\noindent These three stages capture the rich interplay between geometry and limited surfactant resources: depending on the initial geometry and on $\lambda$, the discrete flow may evolve from one regime to another (see Example~\ref{ex:dependence_initial_data}), and transitions can lead either to a stage of complete wetting (thus falling back into the analysis of \cite{CFS}), or to recurrent intermediate configurations with persistent partial wetting.\\

Finally, in Section~\ref{sec:gamma=2}, we analyze the case $\gamma=2$, where at each time step the dissipation due to the surfactant is of the same order as both the phase-one dissipation and the variation of the energy $\E_\e$. We show that the evolution depends on the relation between the parameters $\tempo$ and $k$. Our analysis shows that in most cases the description of the motion can be reduced to one of the previously discussed evolutions for $\gamma<2$. \\

\noindent The paper develops a rigorous discrete-to-continuum analysis of a lattice-driven geometric evolution in the presence of a finite and non-saturating amount of surfactant. From a mathematical viewpoint, the main novelties are threefold. First, we introduce and handle a partial wetting constraint at the discrete level, which leads to a nontrivial combinatorial-geometric structure of admissible minimizers and requires a refined control of the interface geometry. Second, we construct discrete barrier arguments adapted to crystalline geometries, allowing us to rigorously identify pinning phenomena and regime transitions within the minimizing-movements scheme. Third, we derive explicit, anisotropic and, as far as we know, for the first time genuinely nonlocal laws of motion for the side displacements of evolving quasi-octagons, where the velocity of each side depends on the global distribution of surfactant along the interface. These results reveal new mechanisms in lattice evolutions, showing how nonlocal effects and resource constraints can qualitatively alter crystalline curvature flows before the onset of complete wetting. It is worth noticing that our rigorous characterization of discrete pinning phenomena induced by an insufficient amount of surfactant are a scenario directly connected to experimental observations of contact-line pinning and partial wetting in surfactant-laden systems (see \cite{D1997, K2019, S2013, S2017, Z2017}).\\

The paper is organized as follows. In Section~\ref{sec:formulation} we set the notation and recall preliminary results regarding the complete-wetting case from \cite{CFS}. Subsection \ref{subsec:references} revisits the complete-wetting case for $\gamma>2$ (collected from \cite{CFS}), while Section~\ref{sec:gamma<2} contains the main analysis of the non complete wetting regime for $\gamma<2$, where we identify the three stages described above and establish the corresponding motion laws. Finally, Section~\ref{sec:gamma=2} is devoted to the critical case $\gamma=2$ and to further discussions on transitions among regimes.

\section{Formulation of the problem and preliminary results}\label{sec:formulation}
\subsection{Notation}
We denote by $\NN$ the set of natural numbers starting from $1$. We denote by $\e\ZZ^2$ the square lattice with lattice spacing $\e>0$. We will call the points of $\e\ZZ^2$ with the letters $p, q$. When dealing with more than two points we name them as $p^i$, and we denote by $p^i = (p^i_1, p^i_2)$ their components. We write $p\preccurlyeq q$ to say that $p_1\le q_1$ and $p_2\le q_2$. The symmetric difference of two sets $A, B\subset \mathbb{R}^2$ is denoted by $A \triangle B$, their Hausdorff distance by $d_{\mathcal{H}}(A,B)$. For $i=1, 2$ we denote by $e_i$  the standard basis of $\rr^2$. Given $p\in \e\ZZ^2$ we define $\mathcal{N}(p):=\{p\pm\e e_i:\:i=1, 2\}$ so that, in particular, $\#\mathcal{N}(p) = 4$ for every $p\in\e\ZZ^2$. For \( x \in \mathbb{R} \), we denote by \( \lfloor x \rfloor \) and \( \lceil x \rceil \) the floor and ceiling of \( x \), respectively. When no ambiguity arises, we write \( [x] \) to denote the closest integer to \( x \). In the proofs we will often use the symbol $c$ to denote a generic constant, whose value may change form line to line. 

\subsection{Setting of the problem} On the scaled square lattice $\e\ZZ^2$ we consider the set \[\A_\e\coloneqq\Bigl\{u\colon\e\ZZ^2\to\{\pm1,0\}\text{ such that }\#\{p\in\e\ZZ^2:u(p) = 0\}<+\infty\Bigl\}\]
and the family of energies $\E_\e:\A_\e\mapsto\rr$ defined as follows
\begin{equation*}
\E_\e(u) = \sum_{n.n.}\e(1-u(p)u(q)-k(1-(u(p)u(q))^2)),    
\end{equation*}
where $n.n.$ means that the sum is taken over those $p, q\in \e\ZZ^2$ such that $|p-q| = \e$ and $k\in (1/3, 1)$.
In the following, for $I\subset\e\ZZ^2$ we set \begin{equation}
A_I:=\bigcup_{p\in I}Q_\e(p), \quad Q_\e(p):=p+\e[-1/2, 1/2]^2.
\end{equation} 
For $u^\e\in\A_\e,$ and for any sequence $(u^\e_j)_j\subset\A_\e$, we set
\begin{equation*}
          I_{u^\e}:=\{p\in\e\ZZ^2:u^\e(p)=1\},\;\;Z_{u^\e}:= \{p\in\e\ZZ^2:u^\e(p)=0\},\;\;A_{u^\e}:=A_{I_{u^\e}}
\end{equation*}
and 
\begin{equation*}
    I_j^\e := I_{u_{j}^\e},\,Z^\e_j  := \Z_{u^\e_j},\, A^\e_j := A_{u_j^\e}.
\end{equation*}
Given $I\subset\e\ZZ^2$ and $s=1, +\infty$ we define the discrete $L^s$-distance of a point $p\in\e\ZZ^2$ from $I$ as
\begin{equation*}
d^\e_s(p, I) = \inf\{||p-q||_s:q\in I\},
\end{equation*}
where $||p||_1 = |p_1|+|p_2|$ and $||p||_\infty = \max\{|p_1|, |p_2|\}$. We also introduce  the discrete $L^s$-distance of $p$ from the discrete boundary of $I$
\begin{equation*}
d^\e_s(p, \partial I) := \begin{cases}
    \inf\{||p-q||_s:q\in I\}&\text{ if }p\not\in I,\\
    \inf\{||p-q||_s:q\in\e\ZZ^2\setminus I\}&\text{ if }p\in I.
\end{cases}    
\end{equation*}
We observe that, up to a set of negligible measure, we can extend this distance to all $x\in\rr^2$ by setting 
\[d^\e_s(x, \partial I) := d^\e_s(p, \partial I)  \text{ if }x\in Q_\e(p).\]
The discrete diameter of a set $I\subset\e\ZZ^2$ is given by $\diam(I):=\sup\{d^\e_\infty(p, q):p,\,q\in I\}$, while its exterior and interior boundary are
\[\partial^+ I\coloneqq\{p\in\e\ZZ^2\setminus I:d^\e_1(p, \partial I) = \e\},\quad\partial^-I\coloneqq\{p\in I:d^\e_1(p, \partial I) = \e\},\] respectively.

For $\gamma>0$ we introduce the dissipation functional
$\dis_{\e, \gamma}:\:\A_\e\times\A_\e\to [0,+\infty]$ defined as \[\dis_{\e, \gamma}(u, w):=\dis^1_\e(\I_{u}, \I_w)+\e^\gamma\dis^0_\e(Z_{u}, Z_w),\] where $\dis^1_\e$ and $\dis^0_\e$ are defined as
\begin{equation}\label{eq:dis_1}
\dis_\e^1(\I_{u},\I_{w}):=\sum_{p\in\I_{u}\triangle \I_{w}}\e^2d^\e_1(p, \partial\I_{w}) = \int_{A_{\I_{u}}\triangle A_{\I_w}}d^\e_1(x, \partial A_{\I_w})\dx    
\end{equation}
 and 
\begin{equation*}
\dis^0_\e(\Z_{u}, \Z_{w}):=|\#\Z_{u}-\#\Z_{w}|.
\end{equation*}
For $\tau,\gamma>0$ we  also introduce the total energy $\enF:\A_\e\times\A_\e\to\rr$  given by
\begin{equation*}
\enF(u, w)\coloneqq \E_\e(u)+\frac{1}{\tau}\dis_{\e, \gamma}(u, w).
\end{equation*}
In the following we always consider $\tau = \zeta\e$, for a given parameter $\zeta>0$.

\begin{definition}\label{def:minimizing_movement}
    Given $u^\e_0\in\A_\e$, we say that a sequence $(u^\e_j)_j\subset\A_\e$ defined for every $j\in\NN\cup\{0\}$ is a minimizing movement with starting datum $u^\e_0$ if for every $j$ it holds 
\begin{equation}\label{eq:definition_u_j}    u^\e_{j+1}\in\text{argmin}\Bigl(u\mapsto \enF(u, u_j)\Bigl).
\end{equation}
In the following, we will always suppose that the set $I_0^\e$ associated to the initial datum $u^\e_0$ is the discretization of an octagon, in the sense of Definition~\ref{def:discrete_octagon}.
\end{definition}
In the following proofs in order to deduce some conditions on the geometry of the sets $A_j$, it will often be useful to compare a minimizer $u_j$ with a new function $\tilde{u}_j$ obtained from $u_j$ changing its values only at a few points. For $\I\subset\e\ZZ^2$ we use the notation $u_j(\I)\mapsto\xi$ to indicate that for every point $p\in\I$ we replace the value of $u_j$ at $p$ with the new value $\xi\in\{\pm1, 0\}$, and that we do not modify $u_j$ anywhere else.

In all the pictures that follow, $Z_u$ is colored in blue, $\I_u$ is colored in red and $\{p\in \varepsilon \mathbb{Z}^2 :u(p)=-1\}$ is colored in green.

\subsection{Preliminary results}\label{subsec:references}
Here we recall some useful results from \cite{CFS}. To this end we begin introducing some additional definition.
\subsubsection{Discrete sets}

\begin{definition}\label{def:discrete_path}

    We say that a finite ordered set $\pi = \{p^1, \dots, p^n\}\subset\e\ZZ^2$ is a {\it discrete path} connecting $p^1$ and $p^n$ if $d^\e_1(p^i, p^{i+1}) = 1$ for every $i = 1, \dots, n-1.$ We also denote by $\Gamma_{p, q}(\I)$ the space of all paths $\pi$ connecting $p$ and $q$ such that $\pi\subset\I$.\\
    We say that a set $\I\subset\e\ZZ^2$ is {\it connected} if for every pair $p, q\in\I$ there exists $\pi\in\Gamma_{p, q}(I)$.
\end{definition}

\begin{definition}\label{def:discrete_segment}
	Let $p, q\in\e\ZZ^2$ and $n\in\NN\cup\{0\}$ be such that 
	\[q = p+n\e e_i\text{ for }i=1\text{ or } i=2,\quad\text{ or }\;q = p+n\e(e_1\pm e_2).\]
	We define the (discrete) segment connecting $p$ and $q$ as $\dseg{p}{q}:=\{p+s\e e_i:s = 0, \dots, n\}\subset\e\ZZ^2$ in the first case, and as $\dseg{p}{q}:=\{p+s\e(e_1\pm e_2):s = 0, \dots, n\}\subset\e\ZZ^2$ in the second case.
\end{definition}

\begin{definition}\label{def:convexity}
    We say that a set $\I\subset\e\ZZ^2$ is {\it horizontally} (resp. {\it vertically}) {\it convex} if for every $p\in\I,$ and every $n\in\NN$ such that the point $q := p+n\e e_1  \text{ (resp.}\ q:= p+n\e e_2)$ belongs to $\I$, then the entire discrete segment $\dseg{p}{q}$ is contained in $\I$.
\end{definition}

\begin{definition}\label{def:staircase_set}
     We say that $\I\subset\e\ZZ^2$ is a {\it staircase set} if it is connected, and horizontally and vertically convex. Moreover, we define $R_\I$ as the smallest discrete rectangle which contains $\I$, and we say that $\I$ is a {\it non degenerate staircase set} if the set $R_\I\setminus\I=:T$ has four (not empty) connected components $T_i, i=1, \dots, 4.$ We label the components $T_i$ clockwise setting $T_1$ to be the lowest left component (see Figure \ref{fig:staircasedef}).
\begin{figure}[H]
    \centering
    \resizebox{0.40\textwidth}{!}{\tikzset{every picture/.style={line width=0.75pt}} 

\begin{tikzpicture}[x=0.75pt,y=0.75pt,yscale=-1,xscale=1]

\draw  [draw opacity=0] (250.33,50) -- (480.44,50) -- (480.44,280.33) -- (250.33,280.33) -- cycle ; \draw  [color={rgb, 255:red, 155; green, 155; blue, 155 }  ,draw opacity=0.6 ] (250.33,50) -- (250.33,280.33)(260.33,50) -- (260.33,280.33)(270.33,50) -- (270.33,280.33)(280.33,50) -- (280.33,280.33)(290.33,50) -- (290.33,280.33)(300.33,50) -- (300.33,280.33)(310.33,50) -- (310.33,280.33)(320.33,50) -- (320.33,280.33)(330.33,50) -- (330.33,280.33)(340.33,50) -- (340.33,280.33)(350.33,50) -- (350.33,280.33)(360.33,50) -- (360.33,280.33)(370.33,50) -- (370.33,280.33)(380.33,50) -- (380.33,280.33)(390.33,50) -- (390.33,280.33)(400.33,50) -- (400.33,280.33)(410.33,50) -- (410.33,280.33)(420.33,50) -- (420.33,280.33)(430.33,50) -- (430.33,280.33)(440.33,50) -- (440.33,280.33)(450.33,50) -- (450.33,280.33)(460.33,50) -- (460.33,280.33)(470.33,50) -- (470.33,280.33)(480.33,50) -- (480.33,280.33) ; \draw  [color={rgb, 255:red, 155; green, 155; blue, 155 }  ,draw opacity=0.6 ] (250.33,50) -- (480.44,50)(250.33,60) -- (480.44,60)(250.33,70) -- (480.44,70)(250.33,80) -- (480.44,80)(250.33,90) -- (480.44,90)(250.33,100) -- (480.44,100)(250.33,110) -- (480.44,110)(250.33,120) -- (480.44,120)(250.33,130) -- (480.44,130)(250.33,140) -- (480.44,140)(250.33,150) -- (480.44,150)(250.33,160) -- (480.44,160)(250.33,170) -- (480.44,170)(250.33,180) -- (480.44,180)(250.33,190) -- (480.44,190)(250.33,200) -- (480.44,200)(250.33,210) -- (480.44,210)(250.33,220) -- (480.44,220)(250.33,230) -- (480.44,230)(250.33,240) -- (480.44,240)(250.33,250) -- (480.44,250)(250.33,260) -- (480.44,260)(250.33,270) -- (480.44,270)(250.33,280) -- (480.44,280) ; \draw  [color={rgb, 255:red, 155; green, 155; blue, 155 }  ,draw opacity=0.6 ]  ;
\draw   (280.33,80) -- (450.33,80) -- (450.33,250) -- (280.33,250) -- cycle ;
\draw [color={rgb, 255:red, 208; green, 2; blue, 27 }  ,draw opacity=1 ][line width=0.75]    (320.33,80) -- (390.33,80) -- (390.33,90) -- (410.33,90) -- (410.33,110) -- (420.33,110) -- (420.33,120) -- (450.33,120) -- (450.33,160) -- (440.33,160) -- (440.33,190) -- (430.33,190) -- (430.33,200) -- (420.33,200) -- (420.33,220) -- (400.33,220) -- (400.33,230) -- (400.33,250) -- (330.33,250) -- (330.33,230) -- (320.33,230) -- (320.33,220) -- (310.33,220) -- (310.33,210) -- (300.33,210) -- (300.33,200) -- (290.33,200) -- (290.33,190) -- (280.33,190) -- (280.33,140) -- (290.33,140) -- (290.33,120) -- (310.33,120) -- (310.33,90) -- (320.33,90) -- cycle ;

\draw (362.33,153.4) node [anchor=north west][inner sep=0.75pt]  [font=\small]  {$I$};
\draw (282.33,213.4) node [anchor=north west][inner sep=0.75pt]  [font=\footnotesize]  {$T_{1}$};
\draw (292.33,93.4) node [anchor=north west][inner sep=0.75pt]  [font=\footnotesize]  {$T_{2}$};
\draw (422.33,93.4) node [anchor=north west][inner sep=0.75pt]  [font=\footnotesize]  {$T_{3}$};
\draw (432.33,213.4) node [anchor=north west][inner sep=0.75pt]  [font=\footnotesize]  {$T_{4}$};
\draw (442.33,53.4) node [anchor=north west][inner sep=0.75pt]  [font=\small]  {$R_{I}$};

\end{tikzpicture}}
    \caption{In red an example of non degenerate staircase set $I$; in black the smallest discrete rectangle $R_I$ which contains $I$.}
    \label{fig:staircasedef}
\end{figure}
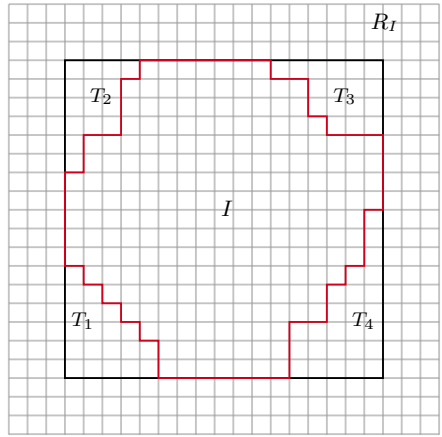
    Moreover, we denote by
    \begin{equation*}
    \label{horizontal_slices}
        H^i = \dseg{p^{h, i}}{q^{h, i}},\,\quad \text{for}\ i=1 \dots, n_h,
    \end{equation*}
    \begin{equation*}
    \label{vertical_slices}
        V^i = \dseg{p^{v, i}}{q^{v, i}}\,\quad \text{for}\ i=1, \dots, n_v,
    \end{equation*} 
    for some $n_h, n_v \in \NN$,
    the horizontal and vertical slices of $\I$, respectively. We label $H^i$ in ascending order starting from the lowest slice, i.e. writing $p^{h,i} = (p^{h,i}_1, p^{h,i}_2)$, $q^{h,i} = (q^{h,i}_1, q^{h,i}_2)$, it holds $p^{h, i}_2 = q^{h, i}_2$ for every $i$, and $p^{h,i}_2<p^{h, j}_2$ if $i<j$. We label the slices $V^i$ from left to right, so that $p^{v, i}_1<p^{v, j}_1$ if $i<j$. We also set \begin{equation}
    \label{eq:parallel_sides_and_slices}
        \PP_1 := H^{1},\,\PP_2 := V^{1},\,\PP_3 := H^{n_h},\,\PP_4 := V^{n_v},\,
    \end{equation}
    to denote the most ``external'' slices. 
We observe that the perimeter of $A_\I$ is \begin{equation*}
\label{eq:length_perimeter}
        Per(A_\I) = 2\e(n_h+n_v),
\end{equation*}
while it holds
\begin{equation*}
     \#\partial^+\I = \frac{Per(A_\I)}{\e}-\#\{p\in \e\ZZ^2\setminus\I:\#(\mathcal{N}(p)\cap\I) = 2\}.
\end{equation*}
With a slight abuse of notation, for $I\subset \e\ZZ^2$ we write $Per(I)$ instead of $Per(A_I).$
\end{definition}

\begin{definition}\label{def:wulff_shape}
    [Wulff-type shape] We denote by $\mathcal{W}$ the set of all the convex octagons contained in $\rr^2$ whose sides are parallel to $e_1, e_2, e_1+e_2$ or $e_1-e_2$. We say that an element of $\mathcal{W}$ is {\it non degenerate} if it has exactly eight sides of positive length. For simplicity, we refer to these Wulff-type shapes as {\it octagons}. For $i=1, \dots, 4,$ we denote by $P_i$ (resp. $D_i$) the length of the sides of the octagon which are parallel either to $e_1$ or $e_2$ (resp. to $e_1 + e_2$ or $e_1-e_2$). We set $P_1$ (resp. $D_1$) to be the length of the lowest side parallel to $e_1$ (resp. to $e_1 - e_2$) and we label all lengths $P_i$ and $D_i$ in clockwise order.
\end{definition}

\begin{definition}\label{def:discrete_octagon}
    [Discrete Wulff-type shape] We say that $I\subset \e\ZZ^2$ is a {\it discrete octagon} if there exists $A\in\mathcal{W}$ such that $I = A\cap\e\ZZ^2$
    (see Figure \ref{fig:disc_oct}).
    \begin{figure}[H]
    \centering
    \resizebox{0.40\textwidth}{!}{\tikzset{every picture/.style={line width=0.75pt}} 

\begin{tikzpicture}[x=0.75pt,y=0.75pt,yscale=-1,xscale=1]

\draw  [draw opacity=0] (260.44,90.33) -- (441.44,90.33) -- (441.44,251.33) -- (260.44,251.33) -- cycle ; \draw  [color={rgb, 255:red, 155; green, 155; blue, 155 }  ,draw opacity=0.6 ] (260.44,90.33) -- (260.44,251.33)(270.44,90.33) -- (270.44,251.33)(280.44,90.33) -- (280.44,251.33)(290.44,90.33) -- (290.44,251.33)(300.44,90.33) -- (300.44,251.33)(310.44,90.33) -- (310.44,251.33)(320.44,90.33) -- (320.44,251.33)(330.44,90.33) -- (330.44,251.33)(340.44,90.33) -- (340.44,251.33)(350.44,90.33) -- (350.44,251.33)(360.44,90.33) -- (360.44,251.33)(370.44,90.33) -- (370.44,251.33)(380.44,90.33) -- (380.44,251.33)(390.44,90.33) -- (390.44,251.33)(400.44,90.33) -- (400.44,251.33)(410.44,90.33) -- (410.44,251.33)(420.44,90.33) -- (420.44,251.33)(430.44,90.33) -- (430.44,251.33)(440.44,90.33) -- (440.44,251.33) ; \draw  [color={rgb, 255:red, 155; green, 155; blue, 155 }  ,draw opacity=0.6 ] (260.44,90.33) -- (441.44,90.33)(260.44,100.33) -- (441.44,100.33)(260.44,110.33) -- (441.44,110.33)(260.44,120.33) -- (441.44,120.33)(260.44,130.33) -- (441.44,130.33)(260.44,140.33) -- (441.44,140.33)(260.44,150.33) -- (441.44,150.33)(260.44,160.33) -- (441.44,160.33)(260.44,170.33) -- (441.44,170.33)(260.44,180.33) -- (441.44,180.33)(260.44,190.33) -- (441.44,190.33)(260.44,200.33) -- (441.44,200.33)(260.44,210.33) -- (441.44,210.33)(260.44,220.33) -- (441.44,220.33)(260.44,230.33) -- (441.44,230.33)(260.44,240.33) -- (441.44,240.33)(260.44,250.33) -- (441.44,250.33) ; \draw  [color={rgb, 255:red, 155; green, 155; blue, 155 }  ,draw opacity=0.6 ]  ;
\draw [color={rgb, 255:red, 208; green, 2; blue, 27 }  ,draw opacity=1 ][line width=0.75]    (270.33,130) -- (270.33,200) -- (280.33,200) -- (280.33,210) -- (290.33,210) -- (290.33,220) -- (300.33,220) -- (300.33,230) -- (310.33,230) -- (310.33,240) -- (400.33,240) -- (400.33,230) -- (410.33,230) -- (410.33,220) -- (420.33,220) -- (420.33,210) -- (430.33,210) -- (430.33,140) -- (420.33,140) -- (420.33,130) -- (410.33,130) -- (410.33,120) -- (400.33,120) -- (400.33,110) -- (390.33,110) -- (390.33,100) -- (300.33,100) -- (300.33,110) -- (290.33,110) -- (290.33,120) -- (280.33,120) -- (280.33,130) -- cycle ;
\draw  [dash pattern={on 4.5pt off 4.5pt}]  (270.33,200) -- (270.44,130.33) -- (300.33,100) -- (390.33,100) ;
\draw  [dash pattern={on 4.5pt off 4.5pt}]  (400.33,240) -- (430.33,210) -- (430.44,140.33) ;
\draw  [dash pattern={on 4.5pt off 4.5pt}]  (400.33,240) -- (310.44,240.33) -- (270.33,200) ;
\draw  [dash pattern={on 4.5pt off 4.5pt}]  (430.33,140) -- (390.33,100) ;

\draw (352.33,173.4) node [anchor=north west][inner sep=0.75pt]  [font=\normalsize,color={rgb, 255:red, 208; green, 2; blue, 27 }  ,opacity=1 ]  {$I$};
\draw (410,98.4) node [anchor=north west][inner sep=0.75pt]    {$A$};
\draw (389.33,228.4) node [anchor=north west][inner sep=0.75pt]  [font=\tiny]  {$q^{1}$};
\draw (270.33,189.4) node [anchor=north west][inner sep=0.75pt]  [font=\tiny]  {$p^{2}$};
\draw (300.33,100.4) node [anchor=north west][inner sep=0.75pt]  [font=\tiny]  {$p^{3}$};
\draw (419.33,140.4) node [anchor=north west][inner sep=0.75pt]  [font=\tiny]  {$q^{4}$};
\draw (310.33,229.4) node [anchor=north west][inner sep=0.75pt]  [font=\tiny]  {$p^{1}$};
\draw (270.33,130.4) node [anchor=north west][inner sep=0.75pt]  [font=\tiny]  {$q^{2}$};
\draw (379.33,100.4) node [anchor=north west][inner sep=0.75pt]  [font=\tiny]  {$q^{3}$};
\draw (419.33,199.4) node [anchor=north west][inner sep=0.75pt]  [font=\tiny]  {$p^{4}$};

\end{tikzpicture}}
    \caption{In black we represent an example of $A \in \mathcal{W}$ and in red the discrete octagon $I$ obtained as $A \cap \varepsilon \mathbb{Z}^2$.}
    \label{fig:disc_oct}
\end{figure}
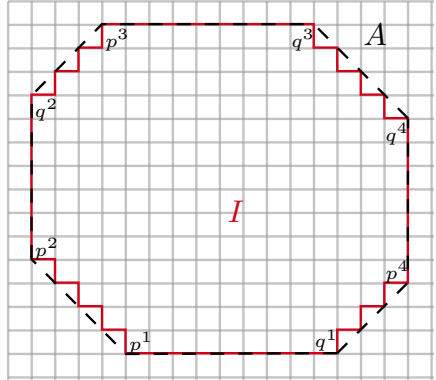

    We say that $I$ is a {\it degenerate} discrete octagon if $A$ is a degenerate octagon. Since a discrete octagon is a staircase set, we adopt notation \eqref{eq:parallel_sides_and_slices} to refer to the sides of $I$ which are parallel to $e_1$ and $e_2$. If $\PP_i = \dseg{p^i}{q^i}$ for $i=1, \dots, 4$, then we set $\DD_1:=\dseg{p^2}{p^1}$ to be the ``diagonal side'' of $I$ which is parallel to $e_1-e_2$, and we define $\DD_i$, for $i=2, 3, 4$, analogously to be the other ``diagonal sides'' of $I,$ labeled clockwise. If $P_i$ and $D_i$ are the side lengths of the smallest octagon $A\in\mathcal{W}$ which contains $A_I$, in the following we say that $P_i$ (resp. $D_i$) is the length of $\PP_i$ (resp. $\DD_i$). Moreover, we define the upper diagonal of $\DD_1$ by \begin{equation}\label{def:upper_diagonal}\DD_1^+:=\dseg{p^2-\e e_2}{p^1-\e e_1},
    \end{equation} and similarly we define $\DD^+_i$ for $i=2, 3, 4.$
    
    Finally, we say that $\I\subset\e\ZZ^2$ is a {\it rectangle} if $A_\I\subset\rr^2$ is a rectangle.
\end{definition}

\begin{definition}\label{def:quasi_octagon}
    [Quasi-octagonal shape] 
    A staircase set $I\subset\e\ZZ^2$ is a {\it quasi-octagon} if either it is an octagon, or if there exists an octagon $\tilde{I}$ such that $I$ can be written as \begin{equation}\label{eq:quasi_octagon_decomposition}
    I := \displaystyle\tilde{\I}\cup\bigl(\cup_{i=1}^4\PP_i\bigl),
    \end{equation}
    where $\PP_i = \dseg{p^i}{q^i}$ are the external slices of $I$ as in Definition~\eqref{eq:parallel_sides_and_slices}. Denoting by $\tilde{\PP}_i =\dseg{\tilde{p}^i}{\tilde{q}^i}$ and $\tilde{\DD}_i$ the sides of $\tilde{I}$, we require moreover that the following properties hold true:    
        
    \begin{equation}\label{eq:quasi_octagon_parallel_sides}
    \begin{aligned} \PP_1\subset\dseg{\tilde{p}^1+\e e_1-\e e_2}{\tilde{q}^1-\e e_1-\e e_2},\\
        \PP_2\subset\dseg{\tilde{p}^2-\e e_1+\e e_2}{\tilde{q}^2-\e e_1-\e e_2},\\
        \PP_3\subset\dseg{\tilde{p}^3+\e e_1+\e e_2}{\tilde{q}^3-\e e_1+\e e_2},\\
        \PP_4\subset\dseg{\tilde{p}^4+\e e_1+\e e_2}{\tilde{q}^4+\e e_1-\e e_2},
    \end{aligned}
    \end{equation}
    see Figure \ref{fig:quasi_oct}.

    Finally, in case $\tilde{p}^1_1<p^1_1$ and $q^1_1<\tilde{q}^1_1$, we set respectively 
\begin{equation}
\label{eq: quasi_oct sides}
    \begin{split}
        \PP_1^{-}:=\dseg{\tilde{p}^1-\e e_2}{p^1-\e e_1},\;\PP_1^{
        +}:=\dseg{q^1+\e e_1}{\tilde{q}^1-\e e_2}
    \end{split}
    \end{equation}
  and we define $\PP_i^\pm$ in the same way for $i=2, 3, 4$. When we need to specify the time step $j$, we write $\PP^{\pm}_{j,i}$.
In the following we denote by  $\tilde{\DD}_i^+$ the upper diagonals of $\tilde{I}$ as in equation \eqref{def:upper_diagonal}.
\begin{figure}[H]
    \centering
    \resizebox{0.40\textwidth}{!}{\input{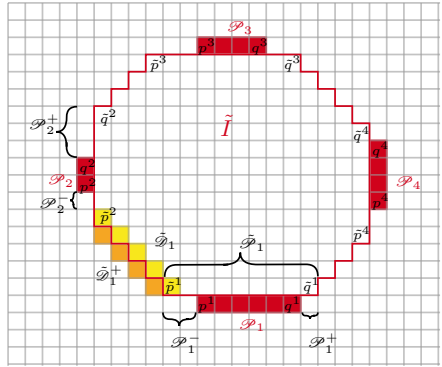}}
    \caption{An example of quasi-octagon $I=\tilde{I}\cup(\cup_{i=1}^{4}\mathscr{P}_i)$; we represent in yellow  the diagonal $\tilde{\DD}_1$, in orange the upper diagonal $\tilde{\DD}_1^+$ and with colored red squares the segments $\PP_i$ for $i=1,\dots, 4$.  }
    \label{fig:quasi_oct}
\end{figure}

\end{definition}
    
    \begin{remark}
        We observe that, for instance, a rectangle is a (degenerate) octagon, i.e. in particular it is a quasi-octagon that cannot be written as in \eqref{eq:quasi_octagon_decomposition}. However, for simplicity, when referring to a quasi-octagon we will always use the decomposition \eqref{eq:quasi_octagon_decomposition} with the convention that $\tilde{I}$ is an octagon, but the external slices $\tilde{\PP}_i$ do not necessarily satisfy \eqref{eq:quasi_octagon_parallel_sides}.
    \end{remark}

\begin{remark}\label{rem:geometry_quasi_octagon}
    Let $\I\subset\e\ZZ^2$ be a staircase set, define $\PP_i$ as in \eqref{eq:parallel_sides_and_slices}, and set
    \[\tilde{\PP}_1 := H^{2},\tilde{\PP}_2 := V^{2},\tilde{\PP}_3 := H^{n_h-1},\tilde{\PP}_4 := V^{n_v-1}.\]
    Then, $\I$ is an octagon if and only if for every pair $p,\,q\in\partial^-\I$ such that $q=p+\e e_i$ for $i=1$ or $i=2$, there exists $j\in\{1, \dots, 4\}$ such that $\{p, q\}\subset\PP_j$. Instead, $\I$ is a quasi-octagon if and only if for every pair $\{p, q\}$ as above there exists $j\in\{1, \dots, 4\}$ such that either $\{p, q\}\subset\PP_j$ or $\{p, q\}\subset\tilde{\PP}_j.$
\end{remark}

\subsubsection{Connection and inclusion of the minimizer} \label{subsub_H} 
We consider the minimizing movement as in  \eqref{def:minimizing_movement}. From now on, in the rest of this article, we will systematically omit the dependence on $\e$ when no confusion occurs; for instance we write $I_j$ in place of $I^\e_j$, and $(u_j)_j\subset\A_\e$ in place of $(u^\e_j)_j\subset\A_\e$. We work with sets $I_j$ having some additional geometric properties that plays the same role of a ``regularity'' assumption in the continuum setting. Namely, we say that the family $(I^\e)_\e$ satisfies assumption \eqref{H} if
\begin{equation}\tag{H}\label{H}
	\begin{minipage}{0.9\textwidth}
		there exist constants $\overline{R}$ and $\overline{c}$ such that $(I^\varepsilon)_\varepsilon$ are staircase sets, all contained in a ball of radius $\overline{R}$, and such that
		\[
		\min\{\#\PP_i^\e : i=1,\dots,4\} \ge \frac{\overline{c}}{\e}.
		\]
	\end{minipage}
\end{equation}
In \eqref{H} we indicated by $\PP_i^\e$ the external slices of $I^\e$, as in \eqref{eq:parallel_sides_and_slices}.

In what follows we say that $u_j$ satisfies \eqref{H} if $(I^\e_j)_\e$ satisfies \eqref{H}. Under this assumption, in \cite[Proposition 3.12]{CFS} we prove that $\I^\e_{j+1}$ is a connected staircase set contained in $\I^\e_j$ for every $\e$ sufficiently small. For reader's convenience, we provide below the statement of the proposition.

\begin{proposition}[\cite{CFS}, Proposition 3.12]

\label{prop:connectedness}
	Let $\gamma>0$ and let $\mu\in(0, 1/4)$. Let $u_0,u_1\in\A_\e$ be such that $u_0$ verifies \eqref{H} and $u_1$ is a minimizer of $\enF(\cdot, u_0)$. Then there exists $\bar{\e}$ such that the set $\I_1$ is a connected staircase set contained in $\I_0$ for every $\e\le\bar{\e}$, and moreover it holds $d_\hh(\partial A_1, \partial A_0)\le\e^\mu$.
\end{proposition}

To conclude the paragraph, we state below a definition and two technical lemmas which will be used several times in the following.

\begin{definition}
	\label{def: corner unit}
	Let $u\in\A_\e$. We say that $p\in \varepsilon\ZZ^2$ is a {\it surfactant in the corner} if, up to rotation, it holds 
	\[u(p) = u(p+\e e_2) = u(p +\e e_1) = 0,\,\text{ and }u(p-\e e_1) = u(p-\e e_2)\neq 0.\]
	
	\noindent Moreover, given $u_j\in\A_\e,$ we define 
	\[\one_j:=\{p\in\partial^-\I_{j}:\#(\mathcal{N}(p)\cap\I_j) = 2\},\;\;\zero_j:=\{p\in\partial^+\I_j:\#(\mathcal{N}(p)\cap\I_j) = 2\}.\]
\end{definition}

\begin{remark}\label{rem:rectangle}
	Let $I\subset\e\ZZ^2$ be a connected set. There exists $p\in\e\ZZ^2\setminus I$ such that $\#\bigl(\mathcal{N}(p)\cap\I\bigl)\ge 2$ if and only $I$ is not a rectangle.
\end{remark}

\begin{lemma}\label{lemma:shape_optimality}
	{\rm[Shape optimality]} Let $\gamma>0$. Let $u_0,\,u_1\in\A_\e$ be such that $I_0$ and $I_1$ are staircase set, and $u_1$ is a minimizer of $\enF(\cdot, u_0)$. It holds that if $p\in\zero_1\cap I_0$ then $u_1(p) = 0$.
	
	Moreover, denote by $H^{i} = \dseg{p^{i}}{ q^{i}},\,i=1, \dots, n_h$ the horizontal slices of $\I_1$. Assume that there exists $i\ge 2$ such that $p^{i}_1+2\e e_1\le p^{i-1}_1$ (so that, in particular, $\{p^{i}, p^{i}+\e e_1\}\subset\partial^-I_1$). Then (at least) one of the following holds:
	\begin{itemize}
		\item[(i)] $\dseg{p^i-\e e_2}{p^{i-1}-\e e_1}\subset\{u_0\neq 1\}$,
		\item[(ii)] $\dseg{p^i-\e e_2}{p^{i-1}-\e e_1}\subset\Z_1$ and $\mathcal{N}(p^{i-1}-\e e_1-\e e_2)\cap\Z_1 = \{p^{i-1}-\e e_1\}$,
		\item[(iii)] $\I_1\cup\Z_1$ is a rectangle and $\partial^+I_1\subset Z_1$.
	\end{itemize}
	Finally suppose that case (ii) occurs and that $\dseg{p^i-\e e_2}{p^{i-1}-\e e_1}\not\subset\{u_0\neq 1\}$. Then it holds $i=2$.
\end{lemma}
\begin{proof}
    See \cite[Lemma 3.4]{CFS}.
\end{proof}

\begin{lemma}\label{lemma:surfactant_placement}
	Let $u_0,\,u_1\in\A_\e$ be such that $u_0$ satisfies assumption \eqref{H}, and that $u_1$ is a minimizer of $\enF(\cdot, u_0)$. Then there exists $\bar{\e}$ such that for every $\e\le\bar{\e}$, for every pair of points $p,\,q$ such that $u_1(p) = -1$ and $u_1(q) = 0,$ it holds $\#\bigl(\nn(p)\cap\I_1\bigl)\le\#\bigl(\nn(q)\cap\I_1\bigl).$
\end{lemma}
\begin{proof}
	This lemma is a direct consequence of \cite[Lemma 3.6]{CFS}, observing that in view of Proposition~\ref{prop:connectedness} for every $\mu<1/4$ there exists $\bar{\e}$ such that for every $\e\le\bar{\e}$ it holds that $I_1\subset I_0$ and $d^\e_1(a, \partial I_0)\le\e^\mu$ for every $a\in\partial^-I_1$. 
\end{proof}

\begin{remark}\label{rem:blu_on_the_boundary}
	A first consequence of Lemma~\ref{lemma:surfactant_placement} is that if $u_1$ is a minimizer of $\enF(\cdot, u_0)$ it holds either $\partial^+I_1\subset Z_1$ or $Z_1\subset\partial^+I_1$. In particular if $\#\Z_1\le \#\partial^+\I_1$ then $\Z_1\subset\partial^+\I_1.$ 
	
	A second consequence is that if $\# Z_1 \le\#\zero_1$, then $Z_1\subset \zero_1$. This second property can be deduced from Lemma~\ref{lemma:surfactant_placement} together with the fact that for every $p\in\partial^+I_1$ it holds $\#(\nn(p)\cap I_1)\le 2$ since $I_1$ is a staircase set.
\end{remark}

\section{The partial wetting in the case \texorpdfstring{$\gamma <2$}{gamma<2}}\label{sec:gamma<2}
In \cite[Section~5]{CFS} we analyzed the minimizing-movement scheme under the assumption $\gamma<2$, in the case where at the initial time $j=0$ the set $I_0$ is completely surrounded by surfactant, that is, when the surfactant fully wets $I_0$. In the present section we again assume $\gamma<2$, but we start from an initial configuration at time $j=0$ in which the amount of surfactant is not sufficient to surround $I_0$ (i.e. $\# Z_0 < \# \partial^+ I_0$). We begin by recalling \cite[Lemma~5.1]{CFS}, which shows that for $\gamma<2$ the amount of surfactant (that is, $\# Z^\e_j$) remains constant at each time step $j$.

\begin{lemma}
\label{lemma: surfactant costante}
\textnormal{\cite[Lemma 5.1]{CFS}}
    Let $\gamma<2$. Let $u_0, u_1\in\A_\e$ be such that $u_1$ is a minimizer of $\enF(\cdot, u_0)$. Set $C_\e:=\#\Z_0^\e$. Then, there exists $\overline{\e}>0$ such that for every $\e\leq \overline{\e}$ it holds $\#\Z^\e_1=C_\e$.
\end{lemma}
We are going to show that if at time step $j$ the surfactant has not yet surrounded $I_j$, then the next minimizer $I_{j+1}$ is a quasi-octagon. 

Given $I_j$, in order to describe the shape of the set $I_{j+1}$ we use the following definition.
\begin{definition}\label{def:sides_movement}
	Let $I_j$ and $I_{j+1}$ be quasi-octagons such that $\I_{j+1}\subset\I_j$, and let $A_j$ and $A_{j+1}$ be the smallest octagons (in the sense of Definition~\ref{def:wulff_shape}) containing $A_{\I_j}$ and $A_{\I_{j+1}}$ respectively. For $s=j,\,j+1$, denote by $\{\PP_{s, i}: i=1,\dots,4\}$ and $\{\DD_{s, i}:\:i=1, \dots, 4\}$ the parallel and the diagonal sides of $A_s$.
	We denote by $a_{j,i}$ the Euclidean distance between the (parallel) straight lines containing the sides $\PP_{j,i}$ and $\PP_{j+1,i}$, and by $b_{j,i}$ the Euclidean distance between the (parallel) straight lines containing the sides $\DD_{j,i}$ and $\DD_{j+1,i}$. We define $\alpha_{j,i}, \beta_{j,i} \in \NN\cup\{0\}$ as 
    	\begin{equation}
    		\begin{split}
    			\label{eq: spostamenti}
    			&\alpha_{j,i}=\frac{a_{j,i}}{\e},\ \quad i=1,\dots,4\\
    			&\beta_{j,i}=\frac{\sqrt{2}b_{j,i}}{\e},\ \quad i=1,\dots,4.
    		\end{split}
    	\end{equation}
    	For $i=1, \dots, 4$ and $s= j, j+1$, we denote by $P_{s, i}$ and $D_{s, i}$ the lengths of $\PP_{s,i}$ and $\DD_{s,i}$ respectively, so that it holds
    	\begin{equation}
    		\label{eq: Lung. lati}
    		\begin{split}
    			&P_{j+1,i}=P_{j,i}+2a_{j,i}-\sqrt{2}(b_{j,i}+b_{j,i-1}) = P_{j, i}+2\e\alpha_{j, i}-\e(\beta_{j, i}+\beta_{j, i-1}),\\
    			&D_{j+1,i}=D_{j,i}+2 b_{j,i}-\sqrt{2}(a_{j,i}+a_{j,i+1}) = D_{j,i}+\sqrt{2}\e \beta_{j,i}-\sqrt{2}\e(\alpha_{j,i}+\alpha_{j,i+1}),
    		\end{split}
    	\end{equation}
    	where index $i\in\{1, 2, 3, 4\}$ is understood modulo $4$.
\end{definition}

In what follows, starting from discrete approximations of a Wulff-type octagon, setting $C_\e:=\#Z^\e_0 = \#Z^\e_j$, and assuming a partial wetting regime with \begin{equation}\label{eq:lambda}
    \e C_\e\to\lambda\ge0,
\end{equation} we show that the minimizing-movement scheme preserves a quasi-octagonal geometry as long as complete wetting does not occur. Depending on the relative amount of surfactant and on the geometry of the diagonals, the evolution falls into one of the following three regimes described in the introduction: pinning of the diagonals, surfactant-dependent side velocities, and an intermediate regime characterized by nonlocal averaging effects.
We complement the section with examples showing that, depending on the initial data, the minimizing movements may evolve from stage two to the situation in which the surfactant completely surrounds the octagon (see Example~\ref{ex:degenerate_octagon_evolution}), or may evolve ``back'' to the intermediate stage in which surfactant only covers the diagonal sides (see Example~\ref{ex:surfactant_on_diagonals}).

We start by proving a lemma in which we show that, assuming $\gamma<2$ and $I_j$ not surrounded by surfactant, the minimising set $I_{j+1}$ is a quasi-octagon. We provide an upper bound $c_\alpha$ on the sides displacements (see \eqref{eq:3.5}) which depends only on the length of the sides of the octagon $\tilde{I}_j$. In \eqref{eq:3.6} we show that, as long as $I_j$ is sufficiently large (so that also $c_\alpha$ is bounded), the length of the parallel sides of the quasi-octagon is bounded from below. This last estimate together with Proposition~\ref{prop:connectedness} ensures that, as long as the quasi-octagon $I_j$ is large enough (and $\e$ is sufficiently small) the set $I_{j+1}$ is connected.
\begin{lemma}\label{lemma:5_optimal_shape_not_surrounded}
        Let $\gamma<2$. Let $u_j,\ u_{j+1} \in \mathcal{A}_\e$ such that $u_j$ satisfies assumption \eqref{H} of Subsection \ref{subsub_H} and $u_{j+1}$ is a minimizer for $\enF(\cdot, u_j)$.
         Let $\mu<1/4$ and suppose that $\I_{j} = \tilde{\I}_j\cup(\cup_{i=1}^4\PP_{j, i})$ is a quasi-octagon such that \begin{equation}\label{eq:boundary_not_surrounded}
         \#\partial^+\I_j-8\e^{\mu-1}>C_\e. \end{equation} Then, there exists $\bar{\e}$    such that for every $\e\le\bar{\e}$ the following properties hold true.
         \begin{itemize}
         	\item[(i)] The set $\I_{j+1}$ is a quasi-octagon such that \begin{equation}\label{eq:zero_estimate}
         	    \#\zero_{j+1}\le\max\{\#\zero_j, C_\e\},
         	\end{equation}
            and if $\#\zero_{j+1}>C_\e$ then it holds $\zero_{j+1}\subset\zero_j.$
         	\item[(ii)] Setting $\tilde{P}_{min}:=\min\{\tilde{P}_{j, i}:\:i=1, \dots, 4\}$, it holds
         	\begin{equation}\label{eq:3.5}
         		\max\{\alpha_{j, i}:\:i=1, \dots, 4\}\le \Bigl\lfloor \frac{4\tempo}{\tilde{P}_{min}}\Bigl\rfloor+2 =:c_\alpha;
         	\end{equation}
         	\item[(iii)] It holds		       \begin{equation}\label{eq:3.6}
         		\min\{P_{j+1, i}:\:i=1, \dots, 4\}\ge \frac{(1-k)\tempo}{c_\alpha+1}.
            \end{equation}
            \item[(iv)] Suppose that $\min\{\tilde{D}_{j, i}:\:i=1, \dots, 4\}\ge\bar{c}$ (the same constant of \eqref{H}). Then there exists a constant $c_\beta$ depending only on $\bar{c}$ such that 
            \begin{equation*}
                \max\{\beta_{j, i}:\:i=1, \dots, 4\}\le c_\beta.
            \end{equation*}
            In particular we deduce that under these assumptions there exists a constant $c$ depending only on $\bar{c}$ such that $d_\hh(\partial A_j, \partial A_{j+1})\le c\e.$
         \end{itemize}
\end{lemma}
\begin{proof}
We divide the proof in several steps.\\

    {\it First step.} We prove that as long as \eqref{eq:boundary_not_surrounded} holds, we have $\Z_{j+1}\subset\partial^+\I_{j+1}$. From Proposition~\ref{prop:connectedness}, we have that there exists $\bar{\e}$ such that for every $\e\le\bar{\e}$ it holds
\[ d_{\mathcal{H}}(\partial A_j, \partial A_{j+1}) \le \varepsilon^\mu,\] therefore there exists an octagon $\I'_{j+1}\subset I_{j+1}$ whose distance from the boundary of $\I_j$ is bigger than $\e^\mu$. More precisely, we set $\alpha'_i = \e^{\mu-1}$ and $\beta'_i = 2\e^{\mu-1},$ and we define $\I'_{j+1}$ through the relations \eqref{eq: spostamenti} and \eqref{eq: Lung. lati}. We observe that $\I_{j+1}'\subset \I_{j+1}$ and \begin{equation}
\label{controllo_cardinality_boundary_I_primo}
\#\partial^+\I_{j+1}'\le\#\partial^+\I_{j+1}.
\end{equation}
Moreover a direct computation shows that 
    \begin{equation}
    \label{cadinality_difference}
        \#\partial^+\I_{j+1}'-\#\partial^+\I_j = -\sum_{i=1}^4\beta_i'.
    \end{equation}
    By \eqref{cadinality_difference} and \eqref{controllo_cardinality_boundary_I_primo}, we conclude that $\#\partial^+\I_{j+1} \ge \#\partial^+\I_j-\sum_{i=1}^42\e^{\mu-1}>C_\e,$ where last inequality follows from \eqref{eq:boundary_not_surrounded}. Hence, from Lemma~\ref{lemma:surfactant_placement} and Remark~\ref{rem:blu_on_the_boundary} we deduce that $\Z_{j+1}\subset\partial^+\I_{j+1},$ which completes the claim.\\
    
    {\it Second step.} Now we prove that $I_{j+1}$ is a quasi-octagon.
    Suppose first that $I_{j+1}\cup Z_{j+1}$ is a rectangle. Recalling that $Z_{j+1}\subset\partial^+I_{j+1}$, it follows as a direct geometric consequence that $I_{j+1}$ is a quasi-octagon. 
    
    We may therefore suppose that $I_{j+1}\cup Z_{j+1}$ is not a rectangle.
    With the notation of Remark~\ref{rem:geometry_quasi_octagon}, in order to complete the proof we need to show that for every $p, q\in\partial^-\I_{j+1}$ such that $q=p+\e e_i$ for $i=1, 2$ there exists $s\in\{1, \dots, 4\}$ such that either $\{p, q\}\subset\PP_{j+1,s}$ or $\{p, q\}\subset \tilde{\PP}_{j+1,s}$. For $i=1, \dots, n_h$ we denote by $H^{j+1, i} = \dseg{p^{j+1, i}}{q^{j+1,i}}$
    the horizontal slices of $\I_{j+1}$. We argue by contradiction supposing, without loss of generality, that there exist $p, q$ and $3\le i\le n_h-2$ such that $q = p+\e e_1$, and $\{p, q\}\subset H^{j+1, i}\cap\partial^-\I_{j+1}$ with $q_1<p^{j+1,i-1}_1$. Consider therefore the segment $S:=\dseg{p^{j+1, i}-\e e_2}{p^{j+1, i-1}-\e e_1}\subset\{u_{j+1}\neq 1\}.$ Since $i\ge 3$ and since $\I_j$ is a quasi-octagon, it can not hold that $S\subset\{u_j\neq 1\}$. Since we are now assuming that $\I_{j+1}\cup\Z_{j+1}$ is not a rectangle, we deduce that situation $(ii)$ of Lemma~\ref{lemma:shape_optimality} occurs. In view of the final part of that lemma we have that $i=2$, which completes the claim.
    
In order to complete the proof of $(i)$ we have to prove \eqref{eq:zero_estimate}. Suppose by contradiction that $\#\zero_{j+1}>\#\zero_j$ and $\#\zero_{j+1}>C_\e.$ Since $\#\zero_{j+1}>\#\zero_j$ we deduce that there exists $p\in \zero_{j+1}\setminus\zero_j$, and therefore, in particular, $p\in I_j.$ In view of the first part of Lemma~\ref{lemma:shape_optimality} it holds $u_{j+1}(p) = 0$. Moreover, since $C_\e = \#Z_{j+1}<\#\zero_{j+1}$ there exists $q\in\zero_{j+1}\setminus Z_{j+1}$. Again in view of the first part of Lemma~\ref{lemma:shape_optimality} we deduce that $u_{j}(q) = -1$. We consider the competitor $\tilde{u}_{j+1}$ obtained by replacing $u_{j+1}(p)\mapsto 1$ and $u_{j+1}(q)\mapsto 0$. We observe that $\E_\e(\tilde{u}_{j+1})\le\E_\e(u_{j+1})$ and $\dis^1_\e(\tilde{u}_{j+1}, u_j)<\dis^1_\e(u_{j+1}, u_j)$, therefore $\enF(\tilde{u}_{j+1})<\enF(u_{j+1}),$ which contradicts the minimality of $u_{j+1}$ and completes the claim. The argument above shows also that if $\#\zero_{j+1}>C_\e$, then $\zero_{j+1}\subset\zero_j$, which completes the proof of $(i)$.\\

    {\it Third step.} We prove $(ii)$. 
    We observe that since $d_\hh(\partial A_{j}, A_{j+1})\le\e^\mu$, and since $I_{j+1}$ is a quasi-octagon there is a geometrical constant $c>0$ such that $|\tilde{P}_{j, i} - \tilde{P}_{j+1, i}|\le c\e^\mu$. We prove \eqref{eq:3.5} for $\alpha_{j, 1}$, which, we recall, is the displacement of the parallel side $\PP_{j,1}$. We denote by $\tilde{\PP}_{j+1, 1}:=\dseg{\tilde{p}^{j+1, 1}}{\tilde{q}^{j+1, 1}}$ the first horizontal slice of the octagon $\tilde{I}_{j+1}$ and by $\PP_{j+1, 1}:=\dseg{p^{j+1, 1}}{q^{j+1, 1}},$ the first horizontal slice of the quasi-octagon $I_{j+1}$, and we set 
    \[S:=\dseg{\tilde{p}^{j+1, 1}-\e e_2}{p^{j+1, 1}-\e e_1}\cup \dseg{q^{j+1, 1}+\e e_1}{\tilde{q}^{j+1, 1}-\e e_2}\cup \dseg{p^{j+1, 1}-\e e_2}{q^{j+1, 1}-\e e_2}.\]
    For simplicity we suppose that $S\subset I_j,$ since this does not change the strategy of the proof.
    We consider the competitor $\tilde{u}_{j+1}$ defined by (see Figure \ref{fig:1Figur_Lemma_3_3})
    \begin{equation*}
    	\tilde{u}_{j+1}(p):=\begin{cases}
    		1&\text{ if }p\in S,\\
    		u_{j+1}(p+\e e_2)&\text{ if }p\in S-\e e_2,\\
    		u_{j+1}(p)&\text{ otherwise}.
    	\end{cases}
    \end{equation*}
    
    \begin{figure}[H]
    \centering
    \resizebox{0.70\textwidth}{!}{\input{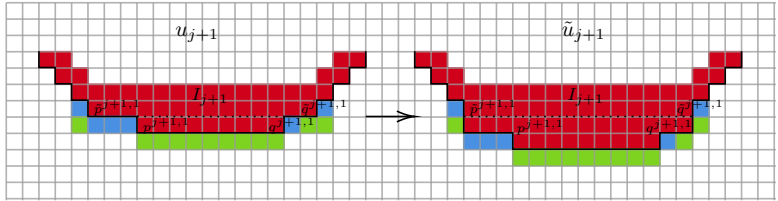}}
    \caption{On the right, the competitor $\tilde{u}_{j+1}$ introduced to determine the upper bound $c_\alpha$ for the side displacement $\alpha_{j,1}$.}
    \label{fig:1Figur_Lemma_3_3}
    \end{figure}
    It holds $\E_\e(\tilde{u}_{j+1})\le \E_\e(u_{j+1})+4\e$, and 
    \[\frac{1}{\tau}\dis^1_\e(\tilde{u}_{j+1},u_j) - \frac{1}{\tau}\dis^1_\e(u_{j+1},u_j)\le -\frac{\tilde{P}_{j+1, 1}(\alpha_{j, 1}-1)\e}{\tempo}\le-\frac{(\tilde{P}_{min}-c\e^\mu)(\alpha_{j, 1}-1)\e}{\tempo},\]
    where $\tilde{P}_{min}:=\min\{\tilde{P}_{j, i}:\:i=1, \dots, 4\}$.
    By minimality of $u_{j+1}$ we deduce that 
    \[4\e\ge \frac{(\tilde{P}_{min}-c\e^\mu)(\alpha_{j, 1}-1)\e}{\tempo}\implies \alpha_{j, 1}\le \frac{4\tempo }{\tilde{P}_{min}-c\e^\mu}+1,\]
    and \eqref{eq:3.5} follows for every $\e$ sufficiently small.\\

    {\it Fourth step.} We prove $(iii)$. More precisely, we prove \eqref{eq:3.6} for the parallel side $\PP_{j+1, 1}$. For simplicity we may suppose that $(\PP_{j+1, 1}-\e e_2)\subset\{u_{j+1} = -1\}$ and that $\dseg{\tilde{p}^{j+1, 1}-\e e_2}{p^{j+1, 1}-\e e_1}\cup \dseg{q^{j+1, 1}+\e e_1}{\tilde{q}^{j+1, 1}-\e e_2}\subset Z_{j+1}$ since the other cases are analogous. Set $n:=\# \dseg{q^{j+1, 1}+\e e_1}{\tilde{q}^{j+1, 1}-\e e_2}$ and $\bar{p}:=p^{j+1, 1}+(n-1)\e e_1$. We define a competitor 
    \begin{equation*}
    	\tilde{u}_{j+1}(p)= \begin{cases}
    		0&\text{ if }p\in \dseg{p^{j+1, 1}}{\bar{p}},\\
    		-1&\text{ if }p\in\dseg{\bar{p}+\e e_1}{\tilde{q}^{j+1, 1}-\e e_2},\\
    		u_{j+1}(p)&\text{ otherwise},
    	\end{cases}
    \end{equation*}
    see Figure \ref{fig:2Figur_Lemma3_3}.
    
    \begin{figure}[H]
    \centering
    \resizebox{0.70\textwidth}{!}{\input{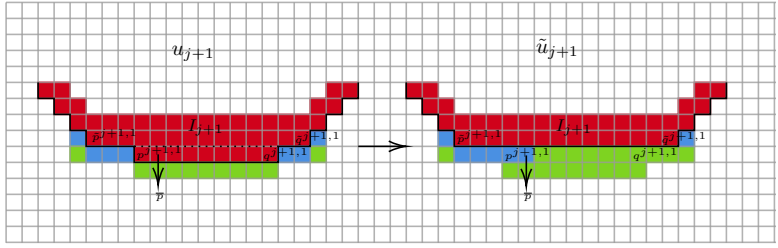}}
    \caption{On the right, the competitor $\tilde{u}_{j+1}$ introduced to determine the lower bound for the lengths of the sides (of $I_{j+1}$) which are parallel to the coordinate axes.}
    \label{fig:2Figur_Lemma3_3}
    \end{figure}
    In this case we have that $\E_\e(\tilde{u}_{j+1})-\E_\e(u_{j+1})\le-\e(1-k)$, and by \eqref{eq:3.5}
    \[\frac{1}{\tau}\dis^1_\e(\tilde{u}_{j+1},u_j) - \frac{1}{\tau}\dis^1_\e(u_{j+1},u_j)\le \frac{P_{j+1, 1}(\alpha_{j, 1}+1)\e}{\tempo}\le  \frac{P_{j+1, 1}(c_\alpha+1)\e}{\tempo}.\]
    By minimality of $u_{j+1}$ we deduce that 
    \[\e(1-k)\le \frac{P_{j+1, 1}(c_\alpha+1)\e}{\tempo}\implies P_{j+1, 1}\ge \frac{\tempo(1-k)}{c_\alpha+1},\]
    which is \eqref{eq:3.6}.

    We omit the proof of $(iv)$, which follows the same strategy as the proof of $(iii)$. The idea is to define a competitor $\tilde{u}_{j+1}$ by adding a layer of phase $1$  in place of $\DD_{j+1, 1}^+$, and by moving properly the surfactant cells which were lying in $\DD_{j+1, 1}^+$, in such a way that $\E_\e(\tilde{u}_{j+1})\le\E_\e(u_{j+1})+c\e$, for a constant $c$ depending only on $k$. After that, we compare $\enF(\tilde{u}_{j+1}, u_j)<\enF(u_{j+1}, u_j)$, and we conclude similarly as before. For more details, we refer to the proof of Proposition~\ref{prop:5.21}, where we show \eqref{eq:sides_movement_quasi_octagon_2} for the displacement $\beta_1.$
\end{proof}

\begin{remark}\label{rem:3.4}
    Let $I$ be a quasi-octagon, and denote by $D_i$ the lengths of the diagonals of the smallest octagon containing $I$. Then it holds that
    \[\#\zero_I = \som\frac{D_{i}}{\sqrt{2}\e},\]
    where $\zero_I$ is the set of points defined according to Definition~\ref{def: corner unit}, relative to $I$. In the following we will use expressions like ``the surfactant covers the diagonal sides'' meaning that $\zero_j\subset Z_j$.
\end{remark}

\subsection{Stage one: ``pinning of diagonals''.} 
From now on, until the beginning of Subsection \ref{subsec:negligible}, we assume $\lambda>0$, where $\lambda$ was defined in \eqref{eq:lambda}. In the next proposition we discuss the sides displacements when ``the surfactant is not enough to cover the diagonals''. More precisely, we show that the side displacements are given by \eqref{eq:movement_parallel_sides_not_covered} provided that we know a priori (i.e. knowing only $u_j$) that 
\( C_\e + 2 \le \#\zero_{j+1} \). 
This last inequality guarantees that if \( I_j \) is an octagon, then \( I_{j+1} \) is also an octagon. 
A sufficient condition ensuring that this inequality holds is given in \eqref{eq:6.23}.

\begin{proposition}\label{prop:5.3_beginning_movement}
Let $\gamma<2,\,\lambda>0$ and suppose that \eqref{eq:boundary_not_surrounded} holds.
    Let $u_j\in\A_\e$ be satisfying assumption \eqref{H} of Subsection \ref{subsub_H} and let $u_{j+1}$ be a minimizer of $\enF(\cdot, u_j)$. 
    Suppose that $\I_{j}$ is an octagon such that 
    \begin{equation}\label{eq:6.23}        \sum_{i=1}^4\Biggl(\frac{D_{j, i}}{\sqrt{2}\e}-2c_\alpha\Biggl)\ge C_\e+2,
    \end{equation}
    where $c_\alpha$ is the constant of Lemma~\ref{lemma:5_optimal_shape_not_surrounded}.
    Then there exists $\bar{\e}$ such that for every $\e\le\bar{\e}$ the set $\I_{j+1}$ is a (possibly degenerate) octagon, and the following properties hold true:
    \begin{itemize}
        \item[(i)] $\#\zero_{j+1}\ge C_\e+2$ and
        $\DD_{j+1, i}\subset\DD_{j, i}$ for every $i$ such that $D_{j+1,i}>0$. 
        \item[(ii)] $D_{j+1, i}=  0$ for every $i$ such that $D_{j, i} = 0$.
        \item[(iii)] There exists a constant $c$ depending only on $\bar{c}$ (the constant of \eqref{H}) such that the displacements of the parallel sides are  \begin{equation}\label{eq:movement_parallel_sides_not_covered}
            \alpha_{j,i} \begin{cases}   =\Bigl\lfloor\frac{4\tempo}{P_{j, i}}\Bigl\rfloor&\text{ if }\text{ dist}\bigl(\frac{4\tempo}{P_{j,i}}, \NN\bigl)\ge c\e\\[0.5em]
                \in\Bigl\{\Bigl[\frac{4\tempo}{P_{j,i}}\Bigl]-1,\Bigl[\frac{4\tempo}{P_{j,i}}\Bigl]\Bigl\}&\text{ otherwise.}
            \end{cases}
        \end{equation}
    \end{itemize}
	
\end{proposition}
\begin{proof}
   For simplicity we rename $\alpha_i = \alpha_{j,i}.$
    By Lemma~\ref{lemma:5_optimal_shape_not_surrounded} there exists $\bar{\e}$ such that for every $\e\le\bar{\e}$ the set $\I_{j+1}$ is a quasi-octagon. We recall that $\DD_{j+1, i}$ denotes the diagonal sides of the smallest discrete octagon which contains $\I_{j+1}$ (as in Definition~\ref{def:quasi_octagon}), and $D_{j+1, i}$ denotes the length of $\DD_{j+1, i}.$ By Remark~\ref{rem:3.4}, \eqref{eq: Lung. lati} and \eqref{eq:6.23} we observe that 
    \[\#\zero_{j+1} = \sum_{i=1}^4\frac{D_{j+1, i}}{\sqrt{2}\e} \ge \sum_{i=1}^4\Biggl(\frac{D_{j, i}}{\sqrt{2}\e}-2\alpha_i\Biggl)\ge\sum_{i=1}^4\Biggl(\frac{D_{j,i}}{\sqrt{2}\e}-2c_\alpha\Biggl)\ge C_\e+2.\]
    From Lemma~\ref{lemma:5_optimal_shape_not_surrounded}-(i) it follows that $\zero_{j+1}\subset\zero_j$. Then (i) follows by a direct geometric argument, and (ii) is a direct consequence of (i).\\
    Now we prove \eqref{eq:movement_parallel_sides_not_covered} for the side displacement $\alpha_1.$ We set $\PP_{j+1, 1} = \dseg{p^1}{q^1}$. For simplicity we may suppose that $\{p^1-\e e_1, q^1+\e e_1\}\subset\zero_{j+1}$. Consider two disjoint sets $S_0$ and $S_{-1}$ such that $\zero_{j+1} = S_0\cup S_{-1}$, and that $\# S_0 = C_\e$. Note that until now both of them can contain points where $u_{j+1}=-1$ or $u_{j+1}=0$. We observe that the competitor $\tilde{u}_{j+1}$ obtained by replacing $u_{j+1}(S_0)\mapsto 0$ and $u_{j+1}(S_{-1})\mapsto -1$ is still a minimizer of $\enF(u_{j+1})$ such that $\{\tilde{u}_{j+1}=1\} = I_{j+1}$.
    In particular since $\#\zero_{j+1}\ge C_\e+2$, we deduce that for every $S_0\subset\zero_{j+1}\setminus \{p^1-\e e_1, q^1+\e e_1\}$ there exists a minimizer $\tilde{u}_{j+1}$ such that $\{\tilde{u}_{j+1} = 0\} = S_0$ and $\{\tilde{u}_{j+1}=1\} = I_{j+1},$ which therefore satisfies $\E_\e(\tilde{u}_{j+1}) = \E_\e(u_{j+1})$, and $\tilde{u}_{j+1}(p^1-\e e_1) = u_{j+1}(q^1+\e e_1)=-1$. Suppose first that $\alpha_1\ge 1$. We build the competitors $\tilde{u}_{j+1}^+,\,\tilde{u}^-_{j+1}$ by setting
    \begin{align*}
        &\tilde{u}_{j+1}^+(p) := \begin{cases}
            1&\text{ if }p\in(\PP_{j+1, 1}-\e e_2)\setminus\{p^1-\e e_2,q^1-\e e_2\},\\
            \tilde{u}_{j+1}(p)&\text{ otherwise},
        \end{cases}\\
        &\tilde{u}_{j+1}^-(p) := \begin{cases}
            -1&\text{ if }p\in\PP_{j+1, 1},\\
            \tilde{u}_{j+1}(p)&\text{ otherwise.}
        \end{cases}
    \end{align*}
    We have $\E_\e(\tilde{u}_{j+1}^+) = \E_\e(\tilde{u}_{j+1})+4\e$ and $\E_\e(\tilde{u}_{j+1}^-) = \E_\e(\tilde{u}_{j+1})-4\e$. We observe that from (i), (ii) and from Lemma~\ref{lemma:5_optimal_shape_not_surrounded}-(ii) it holds $d_\hh(\partial A_0, \partial A_1)\le c_\alpha\e$. Therefore, we can estimate the variation of dissipation as follows.
    \begin{align*}
    	&\frac{1}{\tau}\Bigl(\dis^1_\e(\tilde{u}_{j+1}^+, u_j)-\dis^1_\e(\tilde{u}_{j+1}, u_j)\Bigl)\le -\frac{(P_{j, 1}-2c_\alpha\e)\alpha_{j, 1}\e}{\tempo},\\
    	&\frac{1}{\tau}\Bigl(\dis^1_\e(\tilde{u}_{j+1}^-, u_j)-\dis^1_\e(\tilde{u}_{j+1}, u_j)\Bigl)\le \frac{(P_{j, 1}+2c_\alpha\e)(\alpha_{j, 1}+1)\e}{\tempo}.
    \end{align*}
    In particular, since by minimality of $u_{j+1}$ we have $\enF(\tilde{u}^\pm_{j+1},u_{j})\ge\enF(u_{j+1},u_{j})$, we find that 
    \begin{equation}\label{eq:3.13}
        \frac{(P_{j, 1}-2c_\alpha\e)\alpha_{j, 1}\e}{\tempo}\le4\e\le \frac{(P_{j, 1}+2c_\alpha\e)(\alpha_{j, 1}+1)\e}{\tempo}
    \end{equation}
    and \eqref{eq:movement_parallel_sides_not_covered} follows. In case $\alpha_1 = 0$ then since $\enF(\tilde{u}_{j+1}^-,u_j)\ge\enF(u_{j+1},u_j)$ we deduce only the right inequality of \eqref{eq:3.13}. From this inequality, using $\alpha_{1} = 0$, it follows that 
    \[\frac{4\tempo}{P_{j, 1}-2c_\alpha\e} = \frac{4\tempo}{P_{j, 1}}+\frac{8\tempo c_\alpha\e}{P_{j, 1}(P_{j, 1}-2c_\alpha\e)}\le 1\implies \frac{4\tempo}{P_{j, 1}}<1,\]
    and therefore \eqref{eq:movement_parallel_sides_not_covered} holds even in this case.
\end{proof}

\begin{remark}\label{rem:a_priori_estimate_beginning}
    We point out that \eqref{eq:6.23} was used only to ensure that $\#\zero_{j+1}\ge C_\e+2,$ which is a sufficient condition to have (i)-(ii)-(iii).
\end{remark}

\begin{remark}\label{rem:energy_is_perimeter}
    We observe that if $u_{j+1}$ is as in Proposition~\ref{prop:5.3_beginning_movement}, then 
    \[\E_\e(u_{j+1}) = 2Per(I_{j+1})+4\e(1-k)C_\e-4\e C_\e.\]
    Since for $\e$ fixed the quantity $4\e(1-k)C_\e-4\e C_\e$ is constant, it follows that, as long as the diagonal sides of $u_j$ are sufficiently long, i.e. as long as all the surfactant can be contained in $\cup_{i=1}^4\DD_i^+$, then $u_{j+1}$ is also a minimizer for
 $$\tilde{\mathcal{F}}^{\gamma, \tau}_\e(u_{j+1}) := 2Per(\I_{j+1})+\frac{1}{\tau}\dis^1_\e(u_{j+1}, u_j),$$ which is the same functional considered in \cite{BGN} (up to the multiplicative factor $2$ in front of the perimeter energy).
\end{remark}

\subsection{Stage two: ``surfactant dependent side velocities''.}
In this section we discuss the case in which the surfactant completely covers the sufficiently long diagonal sides of $I_j$, and there are some cells of surfactant adjacent to the parallel sides. Moreover, we suppose that all the diagonal sides are sufficiently long.

The following lemma concerns some technical properties of minimizers.

\begin{lemma}\label{lemma:surfactant placement_around_quasi_octagon} Let $\gamma<2$ and $\lambda>0$.
Let $u_j\in\A_\e$ and let $u_{j+1}$ be a minimizer of $\enF(\cdot, u_j)$. Suppose that $I_{j+1}$ is a quasi-octagon such that $\#\zero_{j+1}+2 \le C_\e < \#\partial^+I_{j+1}.$ Denote by $\PP_{j+1, i} = \dseg{p^{j+1, i}}{q^{j+1, i}}$ and $\tilde{\PP}_{j+1, i} = \dseg{\tilde{p}^{j+1, i}}{\tilde{q}^{j+1, i}}$ the parallel sides of $I_{j+1}$ and $\tilde{I}_{j+1}$ respectively. Set $\PP^\pm_{j+1, i}$ as in Definition~\ref{def:quasi_octagon}. Then the following properties hold true. 
\begin{itemize}
    \item[(i)] If $\PP_{j+1, 1}-\e e_2\cap Z_{j+1}\neq\emptyset$ then there exist two points $p, q$ such that $\PP_{j+1, 1}-\e e_2\cap Z_{j+1}:=\dseg{p}{q}$, i.e. the cells of surfactant below the first horizontal slice form an horizontal segment. Similarly, also $\PP^\pm_{j+1, 1}\cap Z_{j+1}$ are horizontal segments.\\
    Moreover, there exists $p\in Z_{j+1}\setminus\zero_{j+1}$ such that $\#\nn(p)\cap Z_{j+1} = 1.$
    \item[(ii)] There exists $u\in\A_\e$ such that, setting $I_u := \{u=1\}$, $\PP_{i} := \dseg{p^i, q^i}{}$ for $i=1, \dots, 4$ and $\PP_{i}^\pm$ as in Definition~\ref{def:quasi_octagon}, it holds that $I_u$ is a quasi-octagon with $\#\partial^+I_{j+1} = \#\partial^+I_u$ and we have the following:
    \[I_{j+1}\setminus\{p^{j+1, 1}, q^{j+1, 1}\}\subset I_u\subset I_{j+1},\;\;\dseg{p^1-2\e e_1}{p^1-\e e_1}\cup \dseg{q^1+\e e_1}{q^1+2\e e_1}\subset\{u = 0\},\]
    (in particular $\#\PP_{1}^\pm\ge 2$), $\E_\e(u) = \E_\e(u_{j+1}),$ $\#\{u = 0\} = C_\e$ and 
    $|\dis^1_\e(u, u_j)-\dis^1_\e(u_{j+1}, u_j)|\le c\e^3$ for some positive constant, depending only on $\bar{c}$ (the constant of \eqref{H}).
\end{itemize}
\end{lemma}
\begin{proof}
In view of Remark~\ref{rem:blu_on_the_boundary} we have $\Z_{j+1}\subsetneq\partial^+I_{j+1}$. The first part of claim (i) follows directly by the minimality of $u_{j+1}$. Indeed suppose for instance, by contradiction, that $\PP_{j+1, 1}-\e e_2\cap Z_{j+1} = \dseg{p^1}{q^1}\cup \dseg{p^2}{q^2}$ is the disjoint union of two segments. Then we could translate horizontally the first segment until it is adjacent to the second segment, strictly reducing the energy of $u_{j+1}$ and therefore contradicting its minimality. We can argue in the same way to prove that $\PP_{j+1, 1}^\pm\cap Z_{j+1}$ is an horizontal segment. 

We prove the second part of claim (i). Consider any point $p\in Z_{j+1}\setminus\zero_{j+1}.$ If $p\in\PP_{j+1, 1}-\e e_2$ then from the discussion above we know that there exist two points $\bar{p},\bar{q}$ belonging to $Z_{j+1}$ such that $p\in\dseg{\bar{p}}{\bar{q}}
\subset Z_{j+1}$. In particular $\#\nn(\bar{p})\cap Z_{j+1} = 1$, which completes the claim. We argue similarly if $p$ is adjacent to any of the other sides $\PP_{j+1, i}$ for $i=2,\dots, 4$, and also if $p\in\cup_{i=1}^4\PP_{j+1, i}^\pm$. This completes the proof.

Now we prove claim (ii). In view of Remark~\ref{rem:blu_on_the_boundary} we have $\zero_{j+1}\subset Z_{j+1}.$ By the minimality of $u_{j+1}$ one can show that $\tilde{p}^{j+1, 1}_1<p^{j+1, 1}_1<q^{j+1, 1}_1<\tilde{q}^{j+1, 1}_1$ so that, in particular, $p^{j+1, 1}-\e e_1$ and $q^{j+1, 1}+\e e_1$ belong to $\zero_{j+1}$. There are now two cases to consider.\\

{\it Case 1)} Suppose that $p^{j+1, 1}-2\e e_1\in\partial^+I_{j+1}\setminus Z_{j+1}$ and $q^{j+1, 1}+2\e e_1\in\partial^+I_{j+1}\setminus Z_{j+1}$. From (i) there exists $p\in Z_{j+1}\setminus \zero_{j+1}$ such that $\#\nn(p)\cap Z_{j+1} = 1$. We consider the competitor $u'$ obtained by replacing $u_{j+1}(p)\mapsto -1$ and $u_{j+1}(p^{j+1, 1}-2\e e_1)\mapsto 0$. We have that $\E_\e(u') = \E_\e(u_{j+1})$ (and clearly the dissipation has not changed). In particular $u'$ is a minimizer of $\enF(\cdot, u_{j}).$ Since $C_\e-2\ge \#\zero_{j+1}$, we deduce that there exists $q\in Z_{j+1}\setminus (\zero_{j+1}\cup\{p^{j+1, 1}-2\e e_1\})$ such that $\#\nn(q)\cap Z_{j+1} = 1$. We define a competitor $u$ by replacing $u'(q)\mapsto -1$ and $u'(q^{j+1, 1}+2\e e_1)\mapsto 0.$ As before we have that $u$ is a minimizer, and it satisfies all the properties of (ii). Therefore, we deduce that if $\#\PP_{j+1, 1}^-\ge 2$ (resp. if $\#\PP_{j+1, 1}^+\ge 2$) then we can find a minimizer $u$ such that $I_u = I_{j+1}$ and $p^{j+1, 1}-2\e e_1 \in\{u=0\}$ (resp. $q^{j+1, 1}+2\e e_1 \in\{u=0\}$), see Figure \ref{fig:Surfactant_placement_quasi_octagon}.

\begin{figure}[H]
    \centering
    \resizebox{0.70\textwidth}{!}{\input{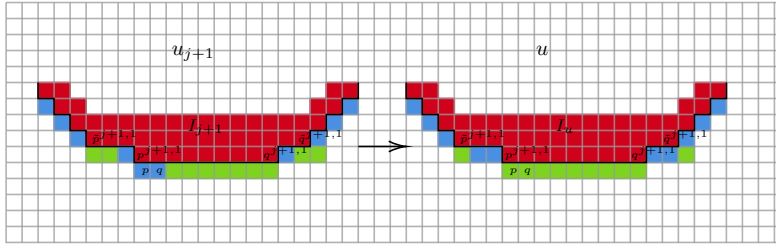}}
    \caption{Case 1). On the right, the minimizer $u$ which preserves energy and dissipation.}
    \label{fig:Surfactant_placement_quasi_octagon}
    \end{figure}

{\it Case 2)}. Suppose that $\#\PP_{j+1, 1}^\pm = 1.$ From (i) there exists $p\in Z_{j+1}\setminus\zero_{j+1}$ such that $\#\nn(p)\cap Z_{j+1} = 1$. If $p^{j+1, 1}-\e e_2\in Z_{j+1}$ we choose $p = p^{j+1, 1}-\e e_2.$ We define a competitor $u'$ by replacing $u_{j+1}(p)\mapsto -1$ and $u_{j+1}(p^{j+1, 1})\mapsto 0$. It holds $\E_\e(u') = \E_\e(u_{j+1})$, and $|\dis^1_\e(u', u_j)-\dis^1_\e(u_{j+1}, u_j)|\le c\e^3$, which is negligible. Moreover, since $C_\e-2\ge\#\zero_{j+1},$ we deduce that there exists $q\in Z_{j+1}\setminus(\zero_{j+1}\cup\{p^{j+1, 1}-\e e_1\})$ such that $\#\nn(p)\cap Z_{j+1} = 1$. If $q^{j+1, 1}-\e e_2\in Z_{j+1}$ we choose $q = q^{j+1, 1}-\e e_2$. We define $u\in\A_\e$ by replacing $u'(q)\mapsto -1$ and $u({q^{j+1, 1}})\mapsto 0.$ The function $u$ verifies all the requirements of (ii), and the proof is complete.
\end{proof}

In the following proposition we show that the side displacements are given by \eqref{eq:sides_movement_quasi_octagon} and  \eqref{eq:sides_movement_quasi_octagon_2} provided that we know a priori that 
\( C_\e \ge \#\zero_{j+1} +2\).
A sufficient condition ensuring that this inequality holds is given by \eqref{eq:5.32}. The values $\amin_{j, i}$ and $\bmax_{j, i}$ in \eqref{eq:5.32} are defined by
\begin{equation}\label{eq:alpha_bar}
\begin{aligned} \amin_{j,i}&:=\begin{cases}
            \Bigl\lfloor\frac{2\tempo(1-k)}{\tilde{P}_{j,i}}\Bigl\rfloor&\text{ if }\text{ dist}\Bigl(\frac{2\tempo(1-k)}{\tilde{P}_{j,i}}, \NN\Bigl)\ge \e^{1/2},\\[.5em]  \Bigl[\frac{2\tempo(1-k)}{\tilde{P}_{j,i}}\Bigl]-1&\text{ otherwise},
        \end{cases}\\[.5em]
        \bmax_{j,i} & :=  \begin{cases}     \Bigl\lfloor\frac{\sqrt{2}\tempo(1+k)}{\tilde{D}_{j,i}}\Bigl\rfloor&\text{ if }\text{dist}\Bigl(\frac{\sqrt{2}\tempo(1+k)}{\tilde{D}_{j,i}}, \NN\Bigl)\ge\e^{1/2}\\[.5em]
        \Bigl[\frac{\sqrt{2}\tempo(1+k)}{\tilde{D}_{j,i}}\Bigl]&\text{otherwise}.
        \end{cases}
    \end{aligned}
\end{equation}
where we recall that $\tilde{D}_{j,i}$ and $\tilde{P}_{j,i}$ are the lengths of the diagonal sides and of the sides parallel to the coordinate axes respectively of
the octagon 
$\tilde{I}_j$.
We will show that, under the assumptions of Proposition~\ref{prop:5.21}, it holds $\alpha_{j, i}\ge\amin_{j,i}$ and $\beta_{j, i}\le\bmax_{j,i}$. We point out that we could have replaced the value $\e^{1/2}$ in \eqref{eq:alpha_bar} with $\e^\theta$ for any $\theta\in (0, 1).$

\begin{proposition}\label{prop:5.21}
Let $\lambda>0$, $\gamma<2$, $\mu<1/4$, and let $u_j\in\A_\e$ be such that $\I_j = \tilde{\I}_j\cup(\cup_{i=1}^4\PP_{j, i})$ is a quasi-octagon verifying assumption \eqref{H} of Subsection \ref{subsub_H}, and \eqref{eq:boundary_not_surrounded}. Denote by $D_{j,i}$ the lengths of the diagonal sides of the smallest octagon containing $\I_j$, and suppose that $D_{j, i}\ge\bar{c}$ for every $i=1, \dots, 4$ (where $\bar{c}$ is the constant of \eqref{H}), and that
    \begin{equation}\label{eq:5.32}
        \sum_{i=1}^4\Biggl(\frac{D_{j, i}}{\sqrt{2}\e}+\bmax_{j, i}-2\amin_{j, i}\Biggl)\le C_\e-2,
    \end{equation}
    where $\beta_{j,i}^{max}$ and $\alpha_{j,i}^{min}$ are defined by \eqref{eq:alpha_bar}.
    Let $u_{j+1}$ be a minimizer of $\enF(\cdot, u_j)$. 
    Then there exists $\bar{\e}$ such that for every $\e\le\bar{\e}$ the set $\I_{j+1}$ is a quasi-octagon satisfying $\#\zero_{j+1}\le C_\e-2$. Moreover the side displacements $\alpha_{j, i}$ and $\beta_{j, i}$ satisfy  \begin{align}\label{eq:sides_movement_quasi_octagon}
        &\alpha_{j,i}\begin{cases}
            \in\Bigl\{\Bigl\lfloor\frac{2\tempo(1-k)}{\tilde{P}_{j,i}}\Bigl\rfloor, \Bigl\lceil\frac{2\tempo(1-k)}{\tilde{P}_{j,i}}\Bigl\rceil\Bigl\}&\text{ if }\text{ dist}\Bigl(\frac{2\tempo(1-k)}{\tilde{P}_{j,i}}, \NN\Bigl)\ge c\e,\\[.5em]
            \in\Bigl\{\Bigl[\frac{2\tempo(1-k)}{\tilde{P}_{j,i}}\Bigl]-1,\,\Bigl[\frac{2\tempo(1-k)}{\tilde{P}_{j,i}}\Bigl],\,\Bigl[\frac{2\tempo(1-k)}{\tilde{P}_{j,i}}\Bigl]+1\Bigl\}&\text{ otherwise};
        \end{cases}\\
        \label{eq:sides_movement_quasi_octagon_2}
        &\beta_{j,i}\begin{cases}     =\Bigl\lfloor\frac{\sqrt{2}\tempo(1+k)}{\tilde{D}_{j,i}}\Bigl\rfloor&\text{ if }\text{dist}\Bigl(\frac{\sqrt{2}\tempo(1+k)}{\tilde{D}_{j,i}}, \NN\Bigl)\ge c\e,\\[.5em]
        \in\Bigl\{\Bigl[\frac{\sqrt{2}\tempo(1+k)}{\tilde{D}_{j,i}}\Bigl]-1,\,\Bigl[\frac{\sqrt{2}\tempo(1+k)}{\tilde{D}_{j,i}}\Bigl]\Bigl\}&\text{otherwise,}
        \end{cases}
    \end{align}
    where the constant $c$ depends only on $\bar{c}$.
\end{proposition}
\begin{proof}
    In view of Lemma~\ref{lemma:5_optimal_shape_not_surrounded} there exists $\bar{\e}$ such that for every $\e\le\bar{\e}$ the set $I_{j+1}$ is a quasi-octagon with $C_\e<\#\partial^+I_{j+1}.$ We divide the proof into two steps. In the first one, we show that \eqref{eq:sides_movement_quasi_octagon} and \eqref{eq:sides_movement_quasi_octagon_2} hold under the assumption that $ \#\zero_{j+1}\leq C_\e -2$. The latter inequality is then proved in the second step. \\
    {\it First step.} 
    We assume that
    \begin{equation}\label{eq:6.36} \#\zero_{j+1} = \som\frac{D_{j+1, i}}{\sqrt{2}\e}\le C_\e-2
    \end{equation}
    holds true, and we prove \eqref{eq:sides_movement_quasi_octagon}. For simplicity we write $\alpha_i$ and $\beta_i$ in place of $\alpha_{j,i}$ and $\beta_{j,i}$. We prove \eqref{eq:sides_movement_quasi_octagon} for $\alpha_1$. Consider a function $u$ with the properties stated in Lemma~\ref{lemma:surfactant placement_around_quasi_octagon}-(ii), and we write $I_u = \tilde{I}_u\cup(\cup_{i=1}^4\PP_i)$, with $\PP_i = \dseg{p^i}{q^i}$ and $\tilde{\PP}_i = \dseg{\tilde{p}^i}{\tilde{q}^i}$ for $i=1, \dots, 4$. We observe that since $\PP_1\subset\PP_{j+1, 1}$ we have that $\alpha_{1}$ is also the side displacement of $\PP_1$ with respect to $\PP_{j, 1}.$    
    We observe that, since $u_{j+1}$ is a minimizer, we can assume that either 
    \begin{equation*}
        p^{j+1,1}_1\le p^{j,1}_1\le q^{j,1}_1\le q^{j+1,1}_1,\;\text{ or }\;p^{j,1}_1\le p^{j+1,1}_1\le q^{j+1,1}_1\le q^{j,1}_1,
    \end{equation*}
    since otherwise we could strictly reduce the dissipation with a horizontal translation of the slice $\PP_{j+1, 1}$. Recalling that from Lemma~\ref{lemma:surfactant placement_around_quasi_octagon} it holds $I_{j+1}\setminus\{p^{j+1, 1}, q^{j+1, 1}\}\subset I_u\subset I_{j+1}$, then for simplicity we complete the proof under the assumption that 
    \begin{equation*}
        p^{j,1}_1\le p^{1}_1\le q^{1}_1\le q^{j,1}_1,
    \end{equation*}
    since the other case is analogous. Moreover, we work under the assumption that $\alpha_{1}\ge 1$. In order to prove that \eqref{eq:sides_movement_quasi_octagon} holds in case $\alpha_{1}=0$ we can argue as we did in the proof of Proposition~\ref{prop:5.3_beginning_movement}.
    
    We build two competitors $u^\pm$ as follows. We suppose for simplicity that $\PP^+_1\cup\PP^-_1\subset \{u=0\}$, and define
    \[S:=(\PP_{1}^-\setminus\{\tilde{p}^{1}-\e e_2\})\cup(\PP_{1}-\e e_2)\cup(\PP_{1}^+\setminus\{\tilde{q}^{1}-\e e_2\}).\]
    We suppose for simplicity that $S\subset I_j$. We define
    \begin{equation}\label{eq:competitors_1}
    \begin{aligned}
        &u^+(p):=\begin{cases}u(p+\e e_2)&\text{ if }p\in S-\e e_2,\\
        1&\text{ if }p\in S,\\
        u(p)&\text{ otherwise}.
        \end{cases}\\
        &u^-(p):=\begin{cases}
            u(p-\e e_2)&\text{ if }p\in(\PP_{1}^-+\e e_2)\cup\PP_{1}\cup(\PP_{1}^++\e e_2),\\
            -1&\text{ if }p\in \PP_{1}^-\cup(\PP_{1}-\e e_2)\cup\PP_{1}^+,\\
            u(p)&\text{ otherwise.}
        \end{cases}
    \end{aligned}
    \end{equation}
    We have \begin{equation}\label{eq:energy_variation}
        \E_\e(u^-)+2\e(1-k) = \E_\e(u) = \E_\e(u_{j+1}^+)-2\e(1-k),\;\text{ and }\;\E_\e(u) = \E_\e(u_{j+1}).
    \end{equation}
    In order to estimate the dissipation we observe that from Lemma~\ref{lemma:surfactant placement_around_quasi_octagon}-(ii) we have
    \begin{equation}\label{eq:3.20}
        \frac{1}{\tau}|\dis^1_\e(u_{j+1}, u_j) - \dis^1_\e(u, u_j)|\le c\e^2.
    \end{equation}
    Moreover since from Lemma~\ref{lemma:5_optimal_shape_not_surrounded}-(iv) we have $d_\hh(\partial A_{j+1}, \partial A_j)\le c\e$, it follows
    \begin{equation}\label{eq:3.200}
        \frac{1}{\tau}\Bigl(\dis^1_\e(u^+, u_j)-\dis^1_\e(u, u_j)\Bigl)\le -\frac{(P^-_{j, 1}+P_{1}+P^+_{j, 1})\alpha_{1}\e+(P_{j, 1}-P_{1})(\alpha_{1}+1)\e}{\tempo}+c\e^2,
    \end{equation}
    \begin{equation}\label{eq:3.201}
        \frac{1}{\tau}\Bigl(\dis^1_\e(u^-, u_j)-\dis^1_\e(u, u_j)\Bigl)\le \frac{(P^-_{j, 1}+P_{1}+P^+_{j, 1})(\alpha_{1}+1)\e+(P_{j, 1}-P_{1})(\alpha_{1}+2)\e}{\tempo}+c\e^2,
    \end{equation}
    for a positive constant depending only on $\bar{c}$.
    Therefore by \eqref{eq:3.20}, \eqref{eq:3.200} and \eqref{eq:3.201} we deduce that 
    \begin{align}\label{eq:free_dissipation_estimate_1}
    	&\frac{1}{\tau}\Bigl(\dis^1_\e(u^+, u_j)-\dis^1_\e(u_{j+1}, u_j)\Bigl)\le -\frac{\tilde{P}_{j, 1}\alpha_{1}\e+(P_{j, 1}-P_{1})\e}{\tempo}+c\e^2,\\\label{eq:free_dissipation_estimate_20}
    	&\frac{1}{\tau}\Bigl(\dis^1_\e(u^-, u_j)-\dis^1_\e(u_{j+1}, u_j)\Bigl)\le \frac{\tilde{P}_{j, 1}(\alpha_{1}+1)\e+(P_{j, 1}-P_{1})\e}{\tempo}+c\e^2.
    \end{align}
    Since $u_{j+1}$ is a minimizer of $\enF(\cdot, u_j)$, by comparing \eqref{eq:energy_variation} with \eqref{eq:free_dissipation_estimate_1} and \eqref{eq:free_dissipation_estimate_20}, we deduce that 
    \begin{equation}\label{eq:bound_alpha}
        \frac{2(1-k)\tempo}{\tilde{P}_{j, 1}}-\frac{P_{j, 1}-P_1}{\tilde{P}_{j, 1}}-1+c\e\le\alpha_{1}\le \frac{2(1-k)\tempo}{\tilde{P}_{j, 1}}-\frac{P_{j, 1}-P_1}{\tilde{P}_{j, 1}}+c\e.
    \end{equation}    The claim follows directly from \eqref{eq:bound_alpha}, observing that $\alpha_{1}$ is integer, and that
    \[0\le\frac{P_{j, 1}-P_{1}}{\tilde{P}_{j, 1}} \le 1.\]
    We now prove \eqref{eq:sides_movement_quasi_octagon_2} for $\beta_{1}$. As before we suppose that $\beta_{ 1}\ge 1$, and we observe that since $\tilde{\DD_1} = \tilde{\DD}_{j+1, 1}$ then $\beta_{1}$ is also the side displacement of $\DD_1$ with respect to $\DD_{j, 1}$. As before we compare $u_{j+1}$ with competitors $u^\pm$. For simplicity we suppose that $\PP_{1}^-\cup\PP_2^-\subset \{u=0\}$. We claim that there exists a point $\bar{p}$ such that either
    \begin{equation}\label{eq:ipotesi_p_1}
        \bar{p}\in \partial^-I_u\;\text{ and }\;\#(\nn(\bar{p})\cap\{u=-1\}) = 1 = \#(\nn(\bar{p})\cap\{u=0\}),
    \end{equation}
    or 
    \begin{equation}\label{eq:ipotesi_p_2}
        u(\bar{p}) = -1\;\text{ and }\;\#(\nn(\bar{p})\cap\{u=1\}) = 1 = \#(\nn(\bar{p})\cap\{u=0\}).
    \end{equation}
    We argue by contradiction. Consider the points $p^1$ and $q^1$. We observe that if both of them do not verify \eqref{eq:ipotesi_p_1} then since $u(p^1-\e e_1) = u(q^1+\e e_1) = 0$ it follows that $u(p^1-\e e_2)=0 = u(q^1-\e e_2)$. Therefore, since $\PP_1-\e e_2\cap \{u=0\}$ is a segment, we deduce that $\PP_1-\e e_2\subset\{u=0\}$. Similarly we deduce that 
    \[(\PP_2-\e e_1)\cup(\PP_3+\e e_2)\cup(\PP_4+\e e_1)\subset\{u=0\}.\]
    Since $\#\{u=0\} = C_\e< \#\partial^+I_{j+1} = \#\partial^+I_u$ we may assume, without loss of generality, that $\PP_1^+\not\subset \{u=0\}.$ If we denote by $\dseg{p}{q}:=\PP_1^+\cap \{u=0\}$ then point $\bar{p} := q+\e e_1$ satisfies \eqref{eq:ipotesi_p_2}.
    
    Now we define two competitors as follows (see Figures \eqref{fig:2phase_u_piu} and \eqref{fig:2_1_phase}).
    \begin{equation}\label{eq:competitors_2}
\begin{aligned}
u^+(p) &:= 
\begin{cases}
0 & \text{if } p \in \tilde{\DD}^+_1 - \varepsilon e_2,\\
1 & \text{if } p \in \tilde{\DD}_1^+,\\
u(p) & \text{otherwise},
\end{cases}
\\[1ex]
u^-(p) &:= 
\begin{cases}
0 & \text{if } p \in \tilde{\DD}_1 \cup \{\bar{p}\},\\
-1 & \text{if } p \in (\tilde{\DD}_1 - \varepsilon e_2)\cup\{\tilde{p}^2-\varepsilon e_1\},\\
u(p) & \text{otherwise}.
\end{cases}
\end{aligned}
\end{equation}
    \begin{figure}[H]
    \centering
    \resizebox{0.75\textwidth}{!}{\input{competitor_u_piu}}
    \caption{On the right, the competitor $u^+$.}
    \label{fig:2phase_u_piu}
    \end{figure}

    \begin{figure}[H]
    \centering
    \resizebox{0.75\textwidth}{!}{\input{competitor_u_meno}}
    \caption{On the right the competitor $u^-$.}
    \label{fig:2_1_phase}
    \end{figure}
    We find that $\E_\e(u^-)+\e(1+k) = \E_\e(u) = \E_\e(u_{j+1}) = \E_\e(u^+)-\e(1+k).$ To conclude we can argue similarly as before: we can estimate $\dis^1_\e(u^\pm, u_j)$, as we did in \eqref{eq:free_dissipation_estimate_1} and \eqref{eq:free_dissipation_estimate_20}, and then we compare $\enF(u^\pm, u_j)$ with $\enF(u_{j+1}, u_j),$ as follows. 
    Since $\I_j$ and $\I_{j+1}$ are octagons, and since $d_{\mathcal{H}}(\partial A_j, \partial A_{j+1})\le c\e$ (from Lemma~\ref{lemma:5_optimal_shape_not_surrounded}-(iv)), we deduce that 
    \begin{align*}
        &\frac{1}{\tau}\Bigl(\dis^1_\e(u^+, u_j)-\dis^1_\e(u_{j+1}, u_j)\Bigl)\le -\frac{(\tilde{D}_{j, 1}-c\e)\beta_{1}\e}{\sqrt{2}\zeta}\le -\frac{\tilde{D}_{j, 1}\beta_{1}\e}{\sqrt{2}\zeta}+c\e^2,\\\label{eq:free_dissipation_estimate_beta_2}
    &\frac{1}{\tau}\Bigl(\dis^1_\e(u^-, u_j)-\dis^1_\e(u_{j+1}, u_j)\Bigl)\le \frac{(\tilde{D}_{j, 1}+c\e)(\beta_{1}+1)\e}{\sqrt{2}\zeta}\le \frac{\tilde{D}_{j, 1}(\beta_{ 1}+1)\e}{\sqrt{2}\zeta}+c\e^2.
\end{align*}
In particular by minimality of $u_{j+1}$, we find that 
\[\frac{\sqrt{2}\zeta(1+k)}{\tilde{D}_{j, 1}}-1-c\e\le\beta_{1}\le \frac{\sqrt{2}\zeta(1+k)}{\tilde{D}_{j, 1}}+c\e,\]
and \eqref{eq:sides_movement_quasi_octagon_2} follows since $\beta_{1}$ is integer.
\vspace{5pt}\\
{\it Second step.} We now prove the claim \eqref{eq:6.36}. 
We can compare the minimizer $u_{j+1}$ with competitors $u^-$ of \eqref{eq:competitors_1} and $u^+$ of \eqref{eq:competitors_2}. We get that $\alpha_{i}\ge\amin_i$ and that $\beta_i\le\bmax_i$ for every $i = 1, \dots, 4$. In particular from \eqref{eq: spostamenti} and \eqref{eq:5.32} we deduce 
    \[\#\zero_{j+1} = \som\frac{D_{j+1, i}}{\sqrt{2}\e} = \som\frac{D_{j, i}}{\sqrt{2}\e}+\beta_{j, i}-2\alpha_{j, i}\le\som \frac{D_{j, i}}{\sqrt{2}\e}+\bmax_{j, i}-2\amin_{j, i}\le C_\e-2.\]
    which is \eqref{eq:6.36}.
\end{proof}

\begin{remark}\label{rem:a_priori_estimate_second_phase}
    We point out that \eqref{eq:5.32} was used only to ensure that $\#\zero_{j+1}\le C_\e-2,$ which is by itself a sufficient condition to have \eqref{eq:sides_movement_quasi_octagon} and \eqref{eq:sides_movement_quasi_octagon_2}. 
\end{remark}

\subsection{Stage three: ``nonlocal averaging of velocities by surfactant redistribution''.}
In this section we deal with the case in which the surfactant approximately covers the diagonal sides, i.e. $C_\e\approx\#\zero_{j}.$ 
In view of Proposition~\ref{prop:5.3_beginning_movement} and Proposition~\ref{prop:5.21}, we are left to consider the case in which either $I_j$ is an octagon which has a short diagonal side (see Example~\ref{ex:degenerate_octagon_evolution} and Proposition~\ref{prop:minimizing_movement_null_diagonal}), or $I_j$ is a quasi-octagon such that
\begin{equation}\label{eq:intermediate_situation}
    C_\e-2-\som\bigl(\bmax_{j, i}-2\amin_{j, i}\bigl)<\som\frac{D_{j, i}}{\sqrt{2}\e}<C_\e+2+8c_\alpha
\end{equation}
(see Example~\ref{ex:surfactant_on_diagonals} and Proposition~\ref{prop:sides_movement_when_set_remains_octagon}). We will observe later that there is no need to study the situation in which $I_j$ is a quasi-octagon (which is not an octagon) with a ``short'' diagonal side.

\begin{example}\label{ex:degenerate_octagon_evolution}
    Consider a minimizing movement having as initial datum an octagon such that $D_{0, 1}\ll D_{0, 2} = D_{0, 3} = D_{0, 4},$ and \[\lambda = \sum_{i=2}^4\frac{D_{0, i}}{\sqrt{2}}-\eta,\]
    where $\lambda$ is as in \eqref{eq:lambda} and $\eta>0$ is a small parameter. The evolution at the beginning is described by Proposition~\ref{prop:5.3_beginning_movement}. We observe that we can choose properly the lengths of $P_{0, i}$ and the parameter $\eta,$ so that after a finite number of steps there exists an index $\bar{j}$ such that $D_{\bar{j}, 1} = 0$. In view of Proposition~\ref{prop:5.3_beginning_movement}, as long as \eqref{eq:6.23} holds, for all $j\ge\bar{j}$ we have $D_{j, 1} = 0.$ If the evolution does not interrupt, then after a finite number of steps \eqref{eq:6.23} does not hold anymore, therefore we cannot apply again Proposition~\ref{prop:5.3_beginning_movement} to understand the subsequent evolution. Moreover, since $D_{j, 1} = 0$ we cannot even apply Proposition~\ref{prop:5.21}. In the following proposition we present the evolution in case the set $I_{j}$ has a ``short'' diagonal side.
\end{example}

\begin{proposition}\label{prop:minimizing_movement_null_diagonal}
    Let $\gamma<2$ and $\lambda\ge 0$ as in \eqref{eq:lambda}.
    Let $u_j\in\A_\e$ be satisfying assumption \eqref{H} of Subsection \ref{subsub_H} and let $u_{j+1}$ be a minimizer of $\enF(\cdot, u_j)$. 
    Suppose that $\I_{j}$ is an octagon verifying \eqref{eq:boundary_not_surrounded} of Lemma~\ref{lemma:5_optimal_shape_not_surrounded}, and such that 
    \begin{equation}\label{eq:3.12}
        C_\e\le\som\frac{D_{j, i}}{\sqrt{2}\e}<C_\e+8c_\alpha+2,
    \end{equation}
    so that in particular \eqref{eq:6.23} does not hold.
    Then there exists $\bar{\e}$ such that for every $\e\le\bar{\e}$ the set $\I_{j+1}$ is a (possibly degenerate) octagon. Moreover the following properties hold true.
    \begin{itemize}
        \item[(i)] Foe every $i=1,\dots,4$ it holds
        \begin{equation}
        \label{stima_spostamenti_diagonali}
            \beta_{j, i}\le 2\som\alpha_{j, i}\le8c_\alpha,
        \end{equation}
    where $c_\alpha$ is the constant of Lemma~\ref{lemma:5_optimal_shape_not_surrounded}-(ii). In particular there exists a constant $c$ depending only on $c_\alpha$ (which depends only on $\bar{c}$, the constant of \eqref{H}) such that $d_\hh(\partial A_j, \partial A_{j+1}) \le c\e.$
    \item[(ii)] If there exists an index $\bar{i}\in\{1, \dots, 4\}$ such that $D_{j, \bar{i}}\le \frac{1}{2}\frac{\sqrt{2}\tempo(1+k)}{8c_\alpha+1}$ then $I_{j+1}$ is an octagon with $Z_{j+1}\subset\zero_{j+1}.$
    \item[(iii)] If there exists a sequence $(\ell_\e)_\e$ such that $\ell_\e\to 0$ as $\e\to 0$, and an index $\bar{i}$ such that $D^\e_{j, \bar{i}}\le\ell_\e$ for every $\e$, then for every $\e\le\bar{\e}$ the set $I^\e_{j+1}$ is an octagon with $Z^\e_{j+1}\subset\zero^\e_{j+1}.$ Moreover for every $i = 1, \dots, 4$ such that $D^\e_{j, i}\ge 16 c_\alpha\ell_\e$ it holds $\beta^\varepsilon_{j, i} = 0$. Finally, the side displacements $\alpha^\e_{j, i}$ for $i=1, \dots, 4$ verify the following system:
        \begin{equation}\label{eq:movement_parallel_sides_not_covered3}
            \alpha^\varepsilon_{j,i} \begin{cases}   =\Bigl\lfloor\frac{4\tempo}{P^\e_{j, i}}\Bigl\rfloor&\text{ if }\text{ dist}\bigl(\frac{4\tempo}{P^\e_{j,i}}, \NN\bigl)\ge c\ell_\e\\[0.5em]
                \in\Bigl\{\Bigl[\frac{4\tempo}{P^\e_{j,i}}\Bigl],\Bigl[\frac{4\tempo}{P^\e_{j,i}}\Bigl]-1\Bigl\}&\text{ otherwise.}
            \end{cases}
        \end{equation}
    where $c$ is a positive constant depending only on $\ell_\e$.
    \end{itemize}
\end{proposition}

    \begin{proof}
   By Lemma~\ref{lemma:5_optimal_shape_not_surrounded} there exists $\bar{\e}$ such that $\I_{j+1}$ is a quasi-octagon for every $\e\le\bar{\e}$. We denote by $\DD_{j+1, i},\,i = 1, \dots, 4$ the diagonal sides of the smallest octagon containing $I_{j+1}$, and by $D_{j+1, i}$ their lengths. From Lemma~\ref{lemma:5_optimal_shape_not_surrounded}-(i) we have that 
   \begin{equation}\label{eq:appr}
       \som \frac{D_{j+1, i}-D_{j, i}}{\sqrt{2}\e} = \#\zero_{j+1}-\#\zero_j\le 0.
   \end{equation}
   By \eqref{eq: Lung. lati}, Lemma~\ref{lemma:5_optimal_shape_not_surrounded}-(ii) and \eqref{eq:appr} we find
        \begin{equation}
        \label{eq:controllo spost diagonali}            \sum_{i=1}^4 \frac{D_{j+1,i}-D_{j, i}}{\sqrt{2}\e}=\sum_{i=1}^4 \beta_{j,i}-2\alpha_{j,i}\le 0\implies \som\beta_{j, i}\le 2\som\alpha_{j, i}\le 8c_\alpha,
        \end{equation}
        and \eqref{stima_spostamenti_diagonali} follows.
Now we prove (ii). In order to prove that $\I_{j+1}$ is an octagon, it is enough  to show that any other quasi-octagonal configuration is not a minimizer. Therefore, we suppose by contradiction that $I_{j+1} = \tilde{I}_{j+1}\cup(\cup_{i=1}^4\PP_{j+1, i})$ is a quasi-octagon which is not an octagon, and we set 
$\PP_{j+1, i} = \dseg{p^i}{q^i},\tilde{\PP}_{j+1, i} = \dseg{\tilde{p}^i}{\tilde{q}^i},\,\tilde{\DD}^+_{j+1,i}$ as in Definition~\ref{def:quasi_octagon}. Without loss of generality we suppose that $\bar{i} = 1$ (where $\bar{i}$ is the index of (ii)). Since $I_{j+1}$ is a quasi-octagon but not an octagon, there exists an index $i$ such that $\#\PP^-_{j+1, i}\ge 2$ or $\#\PP_{j+1, i}^+\ge 2$ (see \eqref{eq: quasi_oct sides} in Definition~\ref{def:quasi_octagon}). For simplicity we complete the proof assuming that $\#\PP_{j+1, 1}^-\ge 2$. Since $\PP_{j+1, 1}^-$ is not contained in $\{u_j\neq 1\},$ by the second part of Lemma~\ref{lemma:shape_optimality} we deduce that $\PP_{j+1, 1}^-\subset\Z_{j+1}$ (and similarly we have that $\PP_{j+1, 2}^-\subset\Z_{j+1}$). In particular $\tilde{p}^1-\e e_2\in\Z_{j+1}$, and since $\#\nn(\tilde{p}^1-\e e_2)\cap I_{j+1} = 1,$ we deduce from Lemma~\ref{lemma:surfactant_placement} that $\zero_{j+1}\subset Z_{j+1}.$ It follows that $\tilde{\DD}_{j+1, 1}^+\subset\Z_{j+1}.$ 
As in the proof of Proposition~\ref{prop:5.21} there exists a point $\bar{p}$ which satisfies either \eqref{eq:ipotesi_p_1} or \eqref{eq:ipotesi_p_2}. Let us assume that $\overline{p}$ satisfies \eqref{eq:ipotesi_p_1}; the other case can be treated analogously. We consider the following competitor (see Figure \ref{fig:quasioct not stable})
\begin{equation*}
\tilde{u}_{j+1}(p)=
            \begin{cases}
                0&\text{if}\ p \in \tilde{\DD}_{j+1,1}\cup \{\bar{p}\}\\
                -1 &\text{if}\ p \in \tilde{\DD}^+_{j+1,1}\cup\{\tilde{p}^1-\e e_2\}\cup\{\tilde{p}^2-\e e_1\}\\
                u_{j+1}(p)&\text{otherwise}.
            \end{cases}
\end{equation*}
  \begin{figure}[H]
    \centering
    \resizebox{0.50\textwidth}{!}{\input{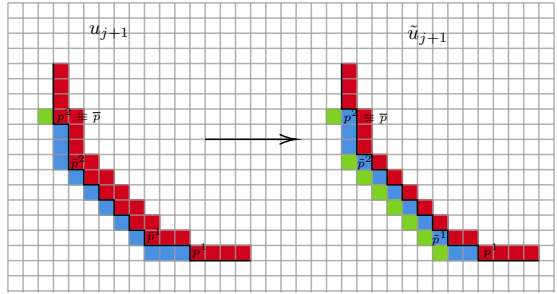}}
    \caption{On the tight, the competitor $\tilde{u}_{j+1}$ in case $\bar{p} = p^2.$}
    \label{fig:quasioct not stable}
\end{figure} 
We have that $\E_\e(\tilde{u}_{j+1})-\E_\e(u_{j+1})=-\e(1+k)<0$, $\dis^0_\e(\tilde{u}_{j+1},u_j)=\dis^0_\e(u_{j+1},u_j),$ and by using \eqref{eq:controllo spost diagonali} it follows

  \begin{equation*}
        \begin{split}
             \frac{\dis^1_\e(\tilde{u}_{j+1},u_j)-\dis^1_\e(u_{j+1},u_j)}{\tau}&\le \frac{(\e\cdot(\#\tilde{\DD}_{j+1, 1}+1))\cdot\max_{p\in\tilde{\DD}_{j+1, 1}\cup\{\bar{p}\}}d^1_\e(p, \partial^+\I_{j})}{\tempo}\\
             &\le \frac{1}{2}\frac{(1+k)}{8c_\alpha+1}(8c_\alpha+1)\e = \frac{(1+k)}{2}\e.
        \end{split}
  \end{equation*}
Hence
\[\E_\e(\tilde{u}_{j+1})-\E_\e(u_{j+1})+\frac{1}{\tau}\left(\dis^1_\e(\tilde{u}_{j+1},u_j)-\dis^1_\e(u_{j+1},u_j)\right)\le-\e(1+k)+\frac{\e(1+k)}{2}<0,\]
and we conclude that any quasi-octagon which is not an octagon cannot be a minimizer.

Now we show that $\Z_{j+1}\subset\zero_{j+1}$. To this end we may assume by contradiction that there exist $p \in (\partial^+I_{j+1} \cap Z_{j+1})\setminus \zero_{j+1}$. We have a contradiction with the minimality of $u_{j+1}$ with a similar argument as above, defining a competitor $\tilde{u}_{j+1}$ as in Figure \ref{fig:quasioct not stable_2}.

 \begin{figure}[H]
    \centering
    \resizebox{0.50\textwidth}{!}{\input{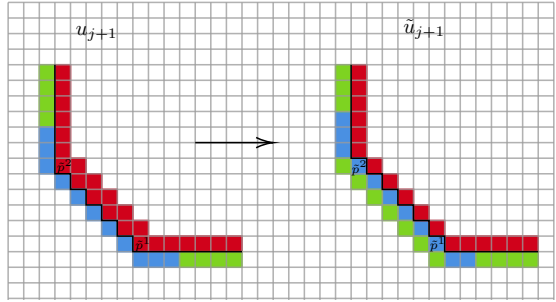}}
    \caption{On the right the configuration that minimize the total energy.}
    \label{fig:quasioct not stable_2}
\end{figure} 

\noindent 

We prove (iii). We already know that $I^\e_{j+1}$ is an octagon with $Z^\e_{j+1}\subset\zero^\e_{j+1}$. We prove that $\beta^\varepsilon_{j, i} = 0$ for every $i$ such that $D^\varepsilon_{j, i}\ge 16\ell_\e c_\alpha.$ Suppose by contradiction that, without loss of generality, $D^\e_{j, 2}\ge 16\ell_\e c_\alpha,$ and that $\beta^\e_{j, 2}\ge1$. We introduce a competitor $u\in\A_\e$ as follows. We define the set $I_u$ through equations \eqref{eq: Lung. lati} by setting $\alpha^\e_i = \alpha^\e_{j, i}$ for every $i$, and 
\begin{equation*}
    \beta^\e_{i} = \begin{cases}
        \beta^\e_{j, 1}+1&\text{ if }i=1,\\
        \beta^\e_{j, 2}-1&\text{ if }i=2,\\
        \beta^\e_{j, i}&\text{ otherwise}.
    \end{cases}
\end{equation*}
Observing that $\#\zero_u = \#\zero_{j+1}\ge C_\e$ (where $\zero_u$ is the set $\zero$ related to $I_u$), we consider any set $S\subset\zero_u$ having $\#S = C_\e$, and we set $Z_u =S$. We have that $\E_\e(u) = \E_\e(u_{j+1}).$ By the definition of dissipation in \eqref{eq:dis_1} we observe that
\begin{equation}\label{eq:explicit_dissipation}
\begin{split}
    \frac{1}{\tempo} \dis^1_\e(u_{j+1},u_j)=  \som \frac{\e}{\tempo}P_{j,i}\frac{\alpha_{j, i}(\alpha_{j, i}+1)}{2}
    +\som \frac{\e}{\tempo}\frac{D_{j,i}}{\sqrt{2}}\frac{\beta_{j, i}(\beta_{j, i}+1)}{2}+e_\e
\end{split}
\end{equation}
where $e_\e$ is an error term depending on $\alpha_{j, i}$ and $\beta_{j, i}$ satisfying $|e_\e(\alpha_{j, i}, \beta_{j, i})|\le c\e^2$ (see also \cite[Subsection 4.2.2]{BS} for more details). From \eqref{eq:explicit_dissipation} we deduce that
\begin{equation*}
\begin{split}
\frac{\dis^1_\e(u, u_j)-\dis^1_\e(u_{j+1}, u_j)}{\tempo}&\le \frac{\e}{2\sqrt{2}\zeta}\Bigl(D_{j, 1}(2\beta_{j, 1}+2)-2D_{j, 2}\beta_{j, 2}\Bigl)+c\e^2\\ &\le \frac{\e}{2\sqrt{2}\zeta}\Bigl(\ell_\e(16c_\alpha+2)-32c_\alpha\ell_\e\Bigl)+c\e^2<0
\end{split}
\end{equation*}
which contradicts the minimality of $u_{j+1}$ and completes the claim. 

To conclude we are left to prove \eqref{eq:movement_parallel_sides_not_covered3}. 
From statement (ii) on $\beta_{j, i}$ of this proposition, and from \eqref{eq:explicit_dissipation} we have 
\begin{equation*}
    \frac{1}{\tau}\dis^1_\e(u_{j+1},u_j) =  \frac{\e}{\tempo}\som P_{j,i}\frac{\alpha_{j,i}(\alpha_{j,i}+1)}{2}+e_\e
\end{equation*}
where $e_\e$ is an error term (which we do not rename) depending on $\alpha_{j, i},\,\beta_{j, i}$ satisfying $|e_\e| \leq c\ell_\e\e$.
Finally, from Remark~\ref{rem:energy_is_perimeter} the energy of $u_{j+1}$ is given by
\begin{align*}
    \E_\e(u_{j+1}) = 2Per(I_{j+1}) + const(\e) = 2\bigl(Per(I_j) -2\e\som\alpha_{j, i}\bigl)+ const(\e).
\end{align*}
Hence the sides displacements $\alpha_{j,i}$ are obtained by solving the following minimum problem
$$\min\left\{-4\e\alpha_{j,i}+\frac{\e}{\tempo}P_{j,i}\frac{\alpha_{j,i}(\alpha_{j,i}+1)}{2}+e_\e: \alpha_{j,i}\in \mathbb{N}\cup\{0\}\right\},$$ and  \eqref{eq:movement_parallel_sides_not_covered3} follows by direct computation.
\end{proof}

\begin{remark}
    We point out that in case (iii) it is no longer true that the diagonal sides are pinned, as it was in Proposition~\ref{prop:5.3_beginning_movement}. Moreover the displacements $\beta_{j, i}$ do not depend only on the length of the corresponding diagonal side $D_{j, i},$ but they depend also on the lengths of the other diagonal sides. In other words, side displacements $\beta_{j, i}$ show a nonlocal behavior.
\end{remark}

We may now turn to the case in which $I_{j}$ is a quasi-octagon satisfying \eqref{eq:intermediate_situation} having ``long'' diagonal sides.

\begin{example}\label{ex:surfactant_on_diagonals}Let $u_j$ be either an octagon such that $\#\zero_j = C_\e$, or a quasi-octagon such that $\#\zero_{j} = C_\e-1.$ We define
\begin{equation}\label{eq:amax}
\begin{aligned} \amax_{j,i}&:=\begin{cases}
            \Bigl\lceil\frac{2\tempo(1-k)}{\tilde{P}_{j,i}}\Bigl\rceil&\text{ if }\text{ dist}\Bigl(\frac{2\tempo(1-k)}{\tilde{P}_{j,i}}, \NN\Bigl)\ge \e^{1/2},\\[.5em]  \Bigl[\frac{2\tempo(1-k)}{\tilde{P}_{j,i}}\Bigl]+1&\text{ otherwise},
        \end{cases}\\[.5em]
        \bmin_{j,i} & :=  \begin{cases}     \Bigl\lfloor\frac{\sqrt{2}\tempo(1+k)}{\tilde{D}_{j,i}}\Bigl\rfloor&\text{ if }\text{dist}\Bigl(\frac{\sqrt{2}\tempo(1+k)}{\tilde{D}_{j,i}}, \NN\Bigl)\ge\e^{1/2}\\[.5em]
        \Bigl[\frac{\sqrt{2}\tempo(1+k)}{\tilde{D}_{j,i}}\Bigl]-1&\text{otherwise}.
        \end{cases}
    \end{aligned}
\end{equation}
We claim that if 
\begin{equation}
\label{eq:suff_condition} \som\frac{D_{j, i}}{\sqrt{2}\e}+\bmin_{j,i}-2\amax_{j,i}>C_\e
\end{equation}
then even $I_{j+1}$ is either an octagon such that $\#\zero_{j+1} = C_\e$, or a quasi-octagon such that $\#\zero_{j+1} = C_\e-1.$ \\
In order to prove the claim, we observe that in view of \eqref{eq:zero_estimate} of Lemma~\ref{lemma:5_optimal_shape_not_surrounded}-(i) it holds 
\begin{equation}
\label{stima_cardinality}
    \#\zero_{j+1}\le C_\e.
\end{equation}
Now we argue by contradiction assuming that  $I_{j+1}$ is a quasi-octagon such that $\#\zero_{j+1}\le C_\e-2$. In view of Remark~\ref{rem:a_priori_estimate_second_phase} and Proposition~\ref{prop:5.21} the side displacements $\alpha_{j, i}$ and $\beta_{j, i}$ verify \eqref{eq:sides_movement_quasi_octagon} and \eqref{eq:sides_movement_quasi_octagon_2}. In particular $\alpha_{j, i}\le\amax_{j, i}$ and $\beta_{j, i}\ge\bmin_{j, i}$. Then, by \eqref{eq:suff_condition}, it follows that
\[\#\zero_{j+1} = \som\frac{D_{j, i}}{\sqrt{2}\e}+\som\beta_{j, i}-2\alpha_{j, i}\ge\som\frac{D_{j, i}}{\sqrt{2}\e}+\som\bmin_{j,i}-2\amax_{j,i}>C_\e,\]
which contradicts \eqref{stima_cardinality} . Therefore, it must be $\#\zero_{j+1}=C_\e$ or $\#\zero_{j+1}=C_\e-1$. 

We observe that, roughly speaking, \eqref{eq:suff_condition} means that ``the diagonal sides are too short with respect to the parallel sides''. 

To conclude we point out that a consequence of this example is that as long as \eqref{eq:suff_condition} holds, the minimizing set $I_j$ remains an octagon or a quasi-octagon with $\#\zero_j = C_\e-1.$ In particular the intermediate stage can last for a long time, and therefore it cannot be neglected when considering the limit flow.
\end{example}

\begin{proposition}\label{prop:sides_movement_when_set_remains_octagon}
Let $\gamma<2$. Let $u_j\in\A_\e$ be such that $\I_j = \tilde{\I}_j\cup(\cup_{i=1}^4\PP_{j, i})$ is a quasi-octagon verifying assumption \eqref{H} of Subsection \ref{subsub_H}. Denoting by $D_{j,i}$ the lengths of the diagonal sides of the smallest octagon containing $\I_j$, we suppose that $D_{j, i}\ge\bar{c}$ (where $\bar{c}$ is the constant of \eqref{H}) for every $i=1, \dots, 4$, and that 
\begin{equation}\label{eq:intermediate_situation_nuova}
    C_\e-2-\som(\bmax_{j,i}-2\amin_{j,i})<\som\frac{D_{j, i}}{\sqrt{2}\e}.
\end{equation}

Let $u_{j+1}$ be a minimizer of $\enF(\cdot, u_j)$, and denote by 
\begin{equation*}
    S_D:=\sum_{i=1}^4\frac{1}{D_{j, i}},\;\;\;S_P:=\sum_{i=1}^4\frac{1}{\tilde{P}_{j, i}},\quad \xi:=\som\frac{D_{j+1, i}-D_{j, i}}{\sqrt{2}\e} = \#\zero_{j+1}-\#\zero_j. 
\end{equation*}
    Then, there exists $\bar{\e}$ such that for every $\e\le\bar{\e}$ the set $\I_{j+1}$ is a quasi-octagon, and the side displacements $\alpha_{j, i}$ and $\beta_{j, i}$ satisfy 
    \begin{align}\label{eq:sides_movement_when_set_remains_octagon}
        &\alpha_{j,i}\begin{cases}
            \in\Bigl\{\Bigl\lfloor\frac{1}{\tilde{P}_{j,i}}\frac{4+4\sqrt{2}\tempo S_D-\xi}{\sqrt{2}S_D+4 S_P}\Bigl\rfloor, \Bigl\lceil\frac{1}{\tilde{P}_{j,i}}\frac{4+4\sqrt{2}\tempo S_D-\xi}{\sqrt{2}S_D+4 S_P}\Bigl\rceil\Bigr\}\text{ if }\text{ dist }\Bigl(\frac{1}{\tilde{P}_{j,i}}\frac{4+4\sqrt{2}\tempo S_D-\xi}{\sqrt{2}S_D+4 S_P}, \NN\Bigl)\ge\e^{1/2},&\\[0.5em]
            \in\Bigl\{\Bigl[\frac{1}{\tilde{P}_{j,i}}\frac{4+4\sqrt{2}\tempo S_D-\xi}{\sqrt{2}S_D+4 S_P}\Bigl]-1, \Bigl[\frac{1}{\tilde{P}_{j,i}}\frac{4+4\sqrt{2}\tempo S_D-\xi}{\sqrt{2}S_D+4 S_P}\Bigl],\Bigl[\frac{1}{\tilde{P}_{j,i}}\frac{4+4\sqrt{2}\tempo S_D-\xi}{\sqrt{2}S_D+4 S_P}\Bigl]+1\Bigl\}\text{ otherwise;} &\end{cases}\\[.5em]      \label{eq:sides_movement_when_set_remains_octagon_2}
            &\beta_{j,i}\begin{cases}
            \in\Bigl\{\Bigl\lfloor\frac{1}{D_{j,i}}\frac{\xi-2+8\tempo S_P}{S_D+2\sqrt{2} S_P}\Bigl\rfloor,\Bigl\lceil\frac{1}{D_{j,i}}\frac{\xi-2+8\tempo S_P}{S_D+2\sqrt{2} S_P}\Bigl\rceil\Bigl\}\text{ if }\text{ dist}\Bigl(\frac{1}{D_{j,i}}\frac{\xi-2+8\tempo S_P}{S_D+2\sqrt{2} S_P}, \NN\Bigl)\ge\e^{1/2},&\\[0.5em]\in\Bigl\{\Bigl[\frac{1}{D_{j,i}}\frac{\xi-2+8\tempo S_P}{S_D+2\sqrt{2} S_P}\Bigl]-1, \Bigl[\frac{1}{D_{j,i}}\frac{\xi-2+8\tempo S_P}{S_D+2\sqrt{2} S_P}\Bigl], \Bigl[\frac{1}{D_{j,i}}\frac{\xi-2+8\tempo S_P}{S_D+2\sqrt{2} S_P}\Bigl]+1\Bigl\}\text{ otherwise.}&
        \end{cases}
    \end{align}
\end{proposition}
\begin{proof}
    From Lemma~\ref{lemma:5_optimal_shape_not_surrounded} there exists $\bar{\e}$ such that for every $\e\le\bar{\e}$ the set $\I_{j+1}$ is a quasi-octagon. We write $\alpha_i,\,\beta_i$ in place of $\alpha_{j, i},\,\beta_{j, i}$. 
    From Lemma~\ref{lemma:5_optimal_shape_not_surrounded}-(ii) we have $\alpha_{j, i}\le c_\alpha$ for every $i$. We deduce that 
    \begin{equation*}
    n:=C_\e-\#\zero_{j+1} = C_\e-\#\zero_j-\som(\beta_{i}-2\alpha_{i})\le C_\e-\#\zero_j+8c_\alpha\le 16c_\alpha+2,
    \end{equation*}
    where last inequality follows combining  \eqref{stima_spostamenti_diagonali} and \eqref{eq:intermediate_situation_nuova}. We claim that $(\PP_{j+1, 1}-\e e_2)\cap\Z_{j+1} = \emptyset$ (and we can deduce similar result after replacing $\PP_{j+1,1}$ with $\PP_{j+1, i}$ for $i=2, 3, 4$). In order to get the claim, we argue by contradiction. We set
    \[S:=\Z_{j+1}\cap(\PP_{j+1, 1}-\e e_2),\] which we assume to be non empty,
    and without loss of generality we can assume that $S = \dseg{p^{j+1, 1}-\e e_2}{p^{j+1, 1}-\e e_2+(n-1)\e e_1}$. We consider the competitor $\tilde{u}_{j+1}$ defined by replacing $u_{j+1}(S)\mapsto -1$ and $u_{j+1}(S+\e e_2)\mapsto 0$; since $\#S = n\le 16 c_\alpha+2$ it follows that for $\e$ sufficiently small $\enF(\tilde{u}_{j+1},u_j)<\enF(u_{j+1},u_j)$, which contradicts the minimality of $u_{j+1}.$ 

    From the discussion above, it follows that $\zero_{j+1}\subset\Z_{j+1}\subset\zero_{j+1}\cup(\cup_{i=1}^4\PP_{j+1, i}^\pm)$, where segments $\PP_{j+1, i}^\pm$ are as in Definition~\ref{def:quasi_octagon}. It follows that 
    \begin{equation}
    \label{eq:energy_intermediate_stage}
    \begin{split}
        \E_\e(u_{j+1}) =& 4\e(1-k)\#\zero_{j+1}+2\e\sum_{i=1}^4\frac{\tilde{P}_{j+1,i}-2\e}{\e}+(C_\e-\#\zero_{j+1})\bigl(-2\e+3\e(1-k)\bigl)\\
        &=2\sum_{i=1}^4\tilde{P}_{j+1,i}+const(\e,\xi, \#\zero_j)\\
        &= \sum_{i=1}^4(4\alpha_i-4\beta_i)+const(\e, \tilde{P}_{j, i}, \xi, \#\zero_j),
    \end{split}
    \end{equation}
    where $const$ is a constant term changing from line to line, depending only on the (constant) quantities within the parentheses.
    
    For simplicity, we now suppose that $\cup_{i=1}^4\PP_{j, i}^\pm\subset \Z_{j}$. The general case is similar. With this choice, we have that $\dis^1_\e(u_{j+1}, u_j)$ can be written as (recalling $\tau = \tempo\e$)
    \begin{equation}\label{eq:dissipation}
        \frac{1}{\tau}\dis^1_\e(u_{j+1}, u_j) = \som\frac{\e}{2\tempo}\tilde{P}_{j, i}\alpha_i(\alpha_i+1)+\som\frac{\e}{2\sqrt{2}\tempo}D_{j, i}\beta_i(\beta_i+1)+e_\e,
    \end{equation}
    where $e_\e$ is an error term depending on displacements $\alpha_{j, i}$ and $\beta_{j, i}$ satisfying $|e_\e|\le c\e^2$.
    Therefore, combining  \eqref{eq:energy_intermediate_stage} and \eqref{eq:dissipation},  we have to minimize 
    \begin{equation}\label{eq:functional_2}
        F(\alpha_i,\dots, \beta_i, \dots) = 4\sum_{i=1}^4\bigl(\alpha_i-\beta_i\bigl)+\sum_{i=1}^4\Bigl(\frac{\e}{\tempo}\frac{D_{j,i}}{\sqrt{2}}\frac{\beta_i(\beta_i+1)}{2}\Bigl)+\frac{\e}{\tempo}\sum_{i=1}^4\tilde{P}_{j,i}\frac{\alpha_i(\alpha_i+1)}{2}+e_\e.
    \end{equation}
    subject to the constraint $\alpha_i,\beta_i\in\NN\cup\{0\}$ and 
    \begin{equation*}
        \xi = \#\zero_{j+1}-\#\zero_j = \som\frac{D_{j+1, i}-D_{j, i}}{\sqrt{2}\e} = \som\beta_i-2\alpha_i.
    \end{equation*}
    We first look for real minimizers using Lagrange multipliers. Neglecting the error term $e_\e$, we find the system
    \[\begin{cases}
        4+\tilde{P}_{j,i}\frac{2\alpha_i+1}{2\tempo} = -2\theta,\\
        -4+D_{j,i}\frac{2\beta_i+1}{2\sqrt{2}} = \theta,\\
        \sum_{i=1}^4\beta_i-2\alpha_i = \xi,
    \end{cases}\]
    for some $\theta\in\R$.
    By direct computation we have that the real solutions of the system are
    \begin{align*}
        \alpha_i = \frac{1}{\tilde{P}_i}\frac{4+4\sqrt{2}\tempo S_D-\xi}{\sqrt{2}S_D+4 S_P}-\frac{1}{2}+e_\e\;\;\text{ and }\;\;\beta_i = \frac{1}{D_i}\frac{\xi-2+8\tempo S_P}{S_D+2\sqrt{2} S_P}-\frac{1}{2}+e_\e,
    \end{align*}
    where we denoted by $e_\e$ two different error terms having the property that $|e_\e|\le c\e$ for a constant $c>0$.
    Finally \eqref{eq:sides_movement_when_set_remains_octagon} and \eqref{eq:sides_movement_when_set_remains_octagon_2} follow since we are looking for solutions in $\NN\cup\{0\}$.
\end{proof}

\begin{remark}\label{rem:distance_from_boundary}
From the systems of inclusions \eqref{eq:sides_movement_when_set_remains_octagon} and \eqref{eq:sides_movement_when_set_remains_octagon_2}, 
the displacement of each side exhibits a nonlocal behavior, i.e. it depends not only on the side itself, 
but also on all the other sides. In addition, the motion involves an averaging of the velocities, 
whereby, in a heuristic sense, longer sides undergo a shorter displacement than shorter ones.
Finally, we point out that from \eqref{eq:sides_movement_when_set_remains_octagon} and \eqref{eq:sides_movement_when_set_remains_octagon_2} it follows that there exists a constant $c$ depending only on $\bar{c}$ such that $\max\{\beta_{j, i}, \alpha_{j, i}:\:i=1, \dots, 4\}\le c$.
\end{remark}

\begin{example}\label{ex:dependence_initial_data}
    [Dependence from initial data] We prove the existence of minimizing movements that, starting from the second stage, evolve toward a configuration in which the surfactant completely surrounds the octagon, as well as the existence of minimizing movements that, again starting from the second stage, evolve ``back'' to the intermediate stage, where the surfactant covers only the diagonal sides. This transition depends on the interplay between the initial amount of surfactant and the side lengths of the quasi-octagon.

We take as initial datum an octagon $I_0$ such that $D_{0,i}=D_{0,j}$ and $P_{0,i}=P_{0,j}$ for every $i,j\in\{1,\dots,4\}$. Let $\eta>0$ be a small parameter, and set
\[
\lambda=\sum\Bigl(\frac{D_{0,i}}{\sqrt{2}}+P_{0,i}\Bigr)-\eta.
\]
By direct computation, we have that
\[
\#\partial^+ I_1=\#\partial^+ I_0-\sum \beta_{0,i}.
\]
In particular, after a finite number of steps we obtain $\#\partial I_j\le C_\varepsilon$, that is, $I_j$ is completely surrounded by surfactant. This inequality follows by choosing $D_{0,i}$ not too large (so that $\beta_{0,i}>0$ in view of~\eqref{eq:sides_movement_quasi_octagon_2}) and $\eta$ sufficiently small.

To show that, for suitable initial data, the evolution proceeds from the second stage to the intermediate stage, we instead choose
\[
\lambda=\sum \frac{D_{0,i}}{\sqrt{2}}+\eta,
\]
for some small $\eta>0$. By taking $P_{0,i}$ very large, it follows from~\eqref{eq:sides_movement_quasi_octagon} that $\alpha_{0,i}\le 1$ for all $i$. Moreover, by~\eqref{eq:sides_movement_quasi_octagon_2} we have
\[
\beta_{0,i}\ge \Bigl[\frac{\sqrt{2}\,\tempo(1+k)}{D_{0,i}}\Bigr]-1=:\Lambda.
\]
Therefore, by~\eqref{eq: Lung. lati},
\[
\frac{D_{1,i}}{\sqrt{2}\varepsilon}
=\frac{D_{0,i}}{\sqrt{2}\varepsilon}+\beta_{0,i}-\alpha_{0,i}-\alpha_{0,i+1}
\ge \frac{D_{0,i}}{\sqrt{2}\varepsilon}+\Lambda-2.
\]
Since the surfactant fills the set $\zero_j$ first, we infer from the previous inequality that, by choosing $D_{0,i}$ sufficiently small (and hence $\Lambda$ sufficiently large), after a finite number of steps condition~\eqref{eq:intermediate_situation} is satisfied, provided that $\eta$ is small enough.
\end{example}

\begin{theorem}
[Limit flow]\label{teo:minimizing_movements_surfactant_covering_more_than_diagonals}
Let $A\subset\rr^2$ be a not degenerate octagon, and for every $\e>0$ let $A_\e$ be an octagon such that $d_{\mathcal{H}}(A, A_\e)\to 0$ as $\e\to 0$. Let $\gamma<2$, $\lambda>0,\,C_\e$ be as in (\ref{eq:lambda}). Let $(u_0^\e)_\e$ be such that $u_0^\e\in\A_\e$ for every $\e$, $\I^\e_0= A_\e\cap\e\ZZ^2$ and $\#\Z^\e_0 = C_\e.$ Let $u^\e_j$ be a minimizing movement with initial datum $u^\e_0$. For every $t\ge 0$ we denote by $A_\e(t)$ the set 
\begin{equation*}
\label{eq:5_flow_bis}
A_\e(t) = \bigcup_{i\in\I^\e_{\lfloor t/\tau\rfloor}}Q_\e(i).
\end{equation*}
Then, it holds that $A_\e(t)$ converges as $\e\to 0$ locally uniformly in time to $A(t).$ The set $A(t)$ is an octagon for every $t\in [0, T)$ with
\begin{equation}\label{eq:T}
    T:=\sup\Biggl\{t\in\rr:\:\lambda<\som \Biggl(\frac{D_i(s)}{\sqrt{2}}+P_i(s)\Biggl)\;\text{ for every }\;s\le t\Biggl\},
\end{equation}
where we denoted by $\PP_i(t),$ and $\DD_i(t)$ its parallel and diagonal sides, and by $P_i(t),\,D_i(t)$ their lengths. Moreover, it holds $A(0) = A,$ and every side of $A(t)$ moves inwards remaining parallel to the side itself. The velocities $v_{\PP_i}(t),\,v_{\DD_i}(t)$ of the $i-$th side depend on the relation between the parameter $\lambda$ and the sum of the lengths of the diagonal sides of $A(t)$. More precisely:
\begin{itemize}
    \item[(i)] If $\lambda<\som \frac{D_i(t)}{\sqrt{2}}$ then $v_{\PP_i}(t)$ and $v_{\DD_i}(t)$ can be deduced from Proposition~\ref{prop:5.3_beginning_movement} by passing to the limit in \eqref{eq:movement_parallel_sides_not_covered}.
    \item[(ii)] If $\som \frac{D_i(t)}{\sqrt{2}}<\lambda<\som \frac{D_i(t)}{\sqrt{2}}+P_i(t)$ then $v_{\PP_i}(t)$ and $v_{\DD_i}(t)$ can be deduced by Proposition~\ref{prop:5.21} by passing to the limit in \eqref{eq:sides_movement_quasi_octagon} and \eqref{eq:sides_movement_quasi_octagon_2}.
\end{itemize}
\end{theorem}
\begin{proof}
    For the proof of this result we refer to \cite[Theorem 4.7]{CFS}. We only give a quick sketch.

    Let $\bar{c}>0$ be a positive constant, and let $\bar{\e}$ be given by Proposition~\ref{prop:connectedness}. 
    Then from Proposition~\ref{prop:connectedness} and Lemma~\ref{lemma:5_optimal_shape_not_surrounded}, we know that for every $\e\le\bar{\e}$ the set $\I^\e_{j+1}$ is a quasi-octagon contained into $\I^\e_j$, as long as \eqref{eq:boundary_not_surrounded} and \eqref{H} are satisfied.    
    Moreover, since $A$ is not degenerate, we can suppose that $D^\e_{0,i}>0$ for every $i=1,\dots,4$.
    
    Recursively repeating the argument above, we find sequences $D^\e_{j, i},\,P^\e_{j, i}$ and $\alpha^\e_{j, i},\,\beta^\e_{j, i}$ which satisfy \eqref{eq: Lung. lati}, i.e.  
 \begin{equation}
     \label{eq: Lung. lati_1}
       \begin{split}
           &P^\e_{j+1,i}=P^\e_{j,i}+2\e \alpha^\e_{j,i}-\e(\beta^\e_{j,i-1}+\beta^\e_{j,i+1})\\           &D^\e_{j+1,i}=D^\e_{j,i}+\sqrt{2}\e\beta^\e_{j,i}-\sqrt{2}\e(\alpha^\e_{j,i}+\alpha^\e_{j,i+1}).
       \end{split}
\end{equation}
For $i=1, \dots, 4$ we define $D_i^\e(t)$ as the linear interpolation in $[j\tau, (j+1)\tau]$ of the values $D^\e_{j,i}$ and $D^\e_{j+1,i}$. We define analogously $P_i^\e(t)$ as the linear interpolation in $[j\tau, (j+1)\tau]$ of the values $P^\e_{j,i}$ and $P^\e_{j+1,i}$. We observe that, in view of \eqref{eq: Lung. lati_1} and of Propositions \ref{prop:5.3_beginning_movement}, \ref{prop:5.21}, \ref{prop:minimizing_movement_null_diagonal} and \ref{prop:sides_movement_when_set_remains_octagon}, the functions $P_i^\e(t)$ and $D_i^\e(t)$ are equibounded and uniformly Lipschitz  on all intervals $[0, \bar{T}]$ such that $P_i^\e(t)\ge\bar{c}$. In particular they converge uniformly as $\e\to 0$ to Lipschitz functions $P_i(t)$ and $D_i(t).$ It follows that the octagons $A^\e(t)$ whose sides are $\PP^\e_i(t)$ and $\DD^\e_i(t)$ converge in the sense of Hausdorff (up to subsequences) as $\e\to 0$ to a limit octagon $A(t)$. 

To conclude, we only point out that in Propositions \ref{prop:5.3_beginning_movement}, \ref{prop:5.21}, \ref{prop:minimizing_movement_null_diagonal} and \ref{prop:sides_movement_when_set_remains_octagon} we did not describe the case in which $I_j$ is a quasi-octagon (which is not an octagon) having a ``short'' diagonal side. In fact, this situation cannot occur. The point is that at the starting step $u^\e_0$ all the diagonal sides $D^\e_{0, i}$ are long (since they are the discretisation of a non degenerate octagon). Then, in case the minimizing movement $u^\e_j$ evolves as in Proposition~\ref{prop:5.21} or Proposition~\ref{prop:sides_movement_when_set_remains_octagon}, the octagon $I_{j}$ is subject to \eqref{eq:sides_movement_quasi_octagon}, \eqref{eq:sides_movement_quasi_octagon_2} or to \eqref{eq:sides_movement_when_set_remains_octagon}, \eqref{eq:sides_movement_when_set_remains_octagon_2}, and in particular under these conditions it holds
\[D^\e_{j, i}\to 0\iff \text{diam}(A_j)\to 0,\]
(where we indicadet by $\text{diam}(A_j)$ the diameter of $A_j$),
so that in particular it cannot happen that a diagonal side is ``too short'' with respect to the parallel sides. If instead at the beginning the minimizing movement $u^\e_j$ evolves subject to Proposition~\ref{prop:5.3_beginning_movement} then as a result it remains an octagon at each step. If one of its diagonal sides becomes too short, then eventually its evolution will be governed by Proposition~\ref{prop:minimizing_movement_null_diagonal}, and in particular $I_j$ will remain an octagon until its diagonal sides become sufficiently long, in view of Proposition~\ref{prop:minimizing_movement_null_diagonal}-(ii).
\end{proof}

\begin{remark}
     For what concerns the case $\lambda = \som \frac{D_i(t)}{\sqrt{2}}$ the existence of the limit flow is still given by Theorem~\ref{teo:minimizing_movements_surfactant_covering_more_than_diagonals}, however we cannot compute the side velocities $v_{\PP_i}(t)$ and $v_{\DD_i}(t)$ (we only have rough estimates of them). We briefly explain the main issue. Let $\bar{t}$ such that $\lambda = \som\frac{D_i(\bar{t})}{\sqrt{2}}$. Then we know that $A_\e(\cdot)\to A(\cdot)$ locally uniformly, therefore in a neighborhood of $\bar{t}$ it holds 
     \begin{equation*}
        \som\frac{D^\e_i(\cdot)}{\sqrt{2}}\approx\lambda.
    \end{equation*}
    However this approximation is not sufficient to know the behavior of the approximating sequence $A_\e$ in a neighborhood of $\bar{t}$, i.e. the minimizing movements at $\e$-level might follow the results of Propositions \ref{prop:5.3_beginning_movement}, \ref{prop:5.21}, \ref{prop:minimizing_movement_null_diagonal} or \ref{prop:sides_movement_when_set_remains_octagon}. Therefore even for the limit flow we can only have rough estimates of the sides displacements based on the formulas provided by these four propositions.
\end{remark}

\begin{remark}
    The description of the minimizing movements in the case where the set $I_j$ is surrounded by surfactant is given in \cite[Section~5]{CFS}. It is straightforward to show that $A_\varepsilon(t)$ converges locally uniformly to a limit flow $A(t)$ also for times $t>T$ (where $T$ is defined in~\eqref{eq:T}), and that $A(t)$ is an octagon for every $t$. If we define
\begin{equation*}
  T_1:=\inf\Biggl\{t\in\rr:\:\lambda>\sum \Biggl(\frac{D_i(t)}{\sqrt{2}}+P_i(t)\Biggr)\Biggr\},
\end{equation*}
then, if $T_1<+\infty$, the limit flow for times $t>T_1$ is described by~\cite[Theorem~5.14]{CFS}, and it corresponds to the evolution of minimizing movements surrounded by surfactant. If instead the parameter $T$ defined in~\eqref{eq:T} satisfies $T<T_1$, we do not have a precise description of the limit flow $A(t)$ for $t\in[T,T_1]$. In this case, we can only derive estimates on the velocities of the sides by using the same strategies employed, for instance, in the proof of Lemma~\ref{lemma:5_optimal_shape_not_surrounded}-(ii)--(iv).
\end{remark}

\subsection{Partial wetting with negligible surfactant}\label{subsec:negligible}
In this subsection we deal with the special case $\lambda = 0$, i.e., we suppose that $\e C_\e\to 0$ as $\e\to0.$ We observe that by Proposition~\ref{prop:5.3_beginning_movement}, as long as  \[\sum_{i=1}^4\Biggl(\frac{D_{j, i}}{\sqrt{2}\e}-2c_\alpha\Biggl)\ge C_\e+2,\]  the minimizing movement is given by a sequence of octagons, and the surfactant set can be arranged along the diagonals of these octagons. After that, the evolution follows by Proposition~\ref{prop:minimizing_movement_null_diagonal}. In the following theorem we describe the limit flow in this special case.

\begin{theorem}\label{teo:negligible_movement}
    Let $A\subset\rr^2$ be a not degenerate octagon, and for every $\e>0$ let $A_\e$ be an octagon such that $d_{\mathcal{H}}(A, A_\e)\to 0$ as $\e\to 0$. Let $(u_0^\e)_\e$ be such that $u_0^\e\in\A_\e$ for every $\e$, and $\I^\e_0= A_\e\cap\e\ZZ^2.$ Suppose that $\gamma<2$, that $u^\e_j$ is a minimizing movement with initial datum $u^\e_0$, and let $C_\e:=\#\Z^\e_{0}$ be such that $\e C_\e\rightarrow 0$ as $\e\to 0$. For every $t\ge 0$ we denote by $A_\e(t)$ the set 
    \begin{equation*}
    A_\e(t) = \bigcup_{i\in\I^\e_{\lfloor t/\tau\rfloor}}Q_\e(i).
    \end{equation*}
    Then,  it holds that $A_\e(t)$ converges as $\e\to 0$ locally uniformly in time to $A(t).$ 
     $A(t)$ is an octagon for every $t\geq 0$
     such that $A(0) = A$; denoting by $\PP_i(t),\,\DD_i(t)$ its parallel and diagonal sides, and by $P_i(t),\,D_i(t)$ their lengths, then every side of $A(t)$ moves inwards remaining parallel to the side itself. The velocity $v_{\PP_i}(t)$ of the $i$-th parallel side satisfies the following differential inclusion \begin{equation}\label{eq:sides_velocity_negligible}
        v_{\PP_i}(t) \begin{cases}
             = \frac{1}{\tempo}\Bigl\lfloor \frac{4\tempo}{P_i(t)}\Bigl\rfloor&\text{ if }\frac{4\tempo}{P_i(t)}\not\in\NN, \\
             \in\Bigl[\frac{1}{\tempo}\Bigl(\frac{4\tempo}{P_i(t)}-1\Bigl),\,\frac{1}{\tempo}\frac{4\tempo}{P_i(t)}\Bigl]&\text{ otherwise}.
        \end{cases}
    \end{equation}
    As long as $D_i(t)>0$ the velocity of the side $\DD_i(t)$ is zero. If $D_i(\hat{t})=0$ for a certain $\hat{t}>0$, then $D_i(t)=0$ for every $t\ge\hat{t}$. 
\end{theorem}

\begin{proof}
   The proof of this theorem is a direct consequence of Proposition~\ref{prop:5.3_beginning_movement} and Proposition~\ref{prop:minimizing_movement_null_diagonal}. Indeed consider a minimizing movement $(u^\e_j)_j$, and suppose that $\bar{j}$ is the first index such that \eqref{eq:6.23} does not hold. Then we deduce that 
   \[\som D_{\bar{j}, i}\le\sqrt{2}\e C_\e+8\sqrt{2}c_\alpha\e+2\sqrt{2}\e.\]
   In particular setting $\ell_\e:=\sqrt{2}\e C_\e+8\sqrt{2}c_\alpha\e+2\sqrt{2}\e$ we deduce that $\ell_\e\to 0$ as $\e\to 0$, and that $D_{\bar{j}, i}\le\ell_\e$ for every $i$. Therefore we deduce from Proposition~\ref{prop:minimizing_movement_null_diagonal}-(iii) that the side displacements $\alpha_{\bar{j}, i}$ verify \eqref{eq:movement_parallel_sides_not_covered3}. Moreover, since from \eqref{eq:3.12} we have $\#\zero_{\bar{j}}\ge C_\e$, and since $Z_{\bar{j}+1}\subset\zero_{\bar{j}+1}$, in view of Lemma~\ref{lemma:5_optimal_shape_not_surrounded}-(i) we deduce that 
   \[C_\e = \# Z_{\bar{j}+1}\le\#\zero_{\bar{j}+1}\le\#\zero_{\bar{j}}\le C_\e+8c_\alpha+2.\]
   In particular \eqref{eq:3.12} is verified even at step $\bar{j}+1$. Inductively we deduce that for every $j\ge\bar{j}$ the side displacements are given by \eqref{eq:movement_parallel_sides_not_covered3}. 

   To prove \eqref{eq:sides_velocity_negligible} we may pass to the limit in \eqref{eq:movement_parallel_sides_not_covered} and in \eqref{eq:movement_parallel_sides_not_covered3}, arguing as in the first part of the proof of Theorem~\ref{teo:minimizing_movements_surfactant_covering_more_than_diagonals}.
\end{proof}

\section{The case \texorpdfstring{$\gamma=2$}{gamma=2}}\label{sec:gamma=2}
In this section we prove that in the case $\gamma=2$ most of the times we can reduce the analysis of the minimizing movements to one of the situations which were analyzed before. Roughly speaking, the evolution depends on the amount of surfactant which is prescribed at time zero, and on the relation between the two parameters $k$ and $\tempo$. 
We introduce three critical cases:
\begin{align}
\tag{1}\label{situation:1} \tempo &> \frac{1}{4k}, \\
\tag{2}\label{situation:2} \tempo &> \frac{1}{3k - 1}, \\
\tag{3}\label{situation:3} \tempo &> \frac{1}{2 - 2k}.
\end{align}
The relevance of these three relations will be clear in what follows.\\

\noindent {\it Case \eqref{situation:1}}.
In this range of parameters if $\#\Z_{j+1}\ge\#\Z_j$, then it is convenient to perform the transition represented in Figure \ref{fig:1.6_1}.
More precisely, suppose that $\#Z_{j+1}\geq \#Z_j$, and that there exists $\overline{p} \in \zero_{j+1}$ such that $u_{j+1}(\overline{p})=-1$, as in Figure \ref{fig:1.6_1}. Then we can define the competitor $\tilde{u}_{j+1}$ replacing $u_{j+1}(\overline{p})\mapsto 0$. We get $\dis^1_\e(\tilde{u}_{j+1},u_j)=\dis^1_\e(u_{j+1},u_j)$ and, if (1) holds, $$\E_\e(\tilde{u}_{j+1})-\E_\e(u_j)+\frac{\e^2}{\tau}\left(\dis^0_\e(\tilde{u}_{j+1},u_j)-\dis^0_\e(u_{j+1},u_j)\right)\le4\e(1-k)-4\e+\frac{\e}{\tempo}<0.$$
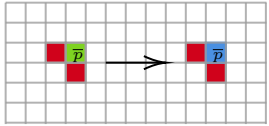
\begin{figure}[H]
    \centering
    \resizebox{0.25\textwidth}{!}{\tikzset{every picture/.style={line width=0.75pt}} 

\begin{tikzpicture}[x=0.75pt,y=0.75pt,yscale=-1,xscale=1]

\draw  [color={rgb, 255:red, 155; green, 155; blue, 155 }  ,draw opacity=0.77 ][fill={rgb, 255:red, 208; green, 2; blue, 27 }  ,fill opacity=1 ] (80.14,79.64) -- (90.14,79.64) -- (90.14,89.64) -- (80.14,89.64) -- cycle ;
\draw  [color={rgb, 255:red, 155; green, 155; blue, 155 }  ,draw opacity=0.77 ][fill={rgb, 255:red, 208; green, 2; blue, 27 }  ,fill opacity=1 ] (90.14,89.64) -- (100.14,89.64) -- (100.14,99.64) -- (90.14,99.64) -- cycle ;
\draw  [color={rgb, 255:red, 155; green, 155; blue, 155 }  ,draw opacity=0.77 ][fill={rgb, 255:red, 208; green, 2; blue, 27 }  ,fill opacity=1 ] (150.14,79.64) -- (160.14,79.64) -- (160.14,89.64) -- (150.14,89.64) -- cycle ;
\draw  [color={rgb, 255:red, 155; green, 155; blue, 155 }  ,draw opacity=0.77 ][fill={rgb, 255:red, 208; green, 2; blue, 27 }  ,fill opacity=1 ] (160.14,89.64) -- (170.14,89.64) -- (170.14,99.64) -- (160.14,99.64) -- cycle ;
\draw  [color={rgb, 255:red, 155; green, 155; blue, 155 }  ,draw opacity=0.77 ][fill={rgb, 255:red, 126; green, 211; blue, 33 }  ,fill opacity=1 ] (90.14,79.64) -- (100.14,79.64) -- (100.14,89.64) -- (90.14,89.64) -- cycle ;
\draw  [color={rgb, 255:red, 74; green, 144; blue, 226 }  ,draw opacity=1 ][fill={rgb, 255:red, 74; green, 144; blue, 226 }  ,fill opacity=1 ] (160.14,79.64) -- (170.14,79.64) -- (170.14,89.64) -- (160.14,89.64) -- cycle ;
\draw  [draw opacity=0] (60.14,59.64) -- (190.44,59.64) -- (190.44,120.33) -- (60.14,120.33) -- cycle ; \draw  [color={rgb, 255:red, 155; green, 155; blue, 155 }  ,draw opacity=0.77 ] (60.14,59.64) -- (60.14,120.33)(70.14,59.64) -- (70.14,120.33)(80.14,59.64) -- (80.14,120.33)(90.14,59.64) -- (90.14,120.33)(100.14,59.64) -- (100.14,120.33)(110.14,59.64) -- (110.14,120.33)(120.14,59.64) -- (120.14,120.33)(130.14,59.64) -- (130.14,120.33)(140.14,59.64) -- (140.14,120.33)(150.14,59.64) -- (150.14,120.33)(160.14,59.64) -- (160.14,120.33)(170.14,59.64) -- (170.14,120.33)(180.14,59.64) -- (180.14,120.33)(190.14,59.64) -- (190.14,120.33) ; \draw  [color={rgb, 255:red, 155; green, 155; blue, 155 }  ,draw opacity=0.77 ] (60.14,59.64) -- (190.44,59.64)(60.14,69.64) -- (190.44,69.64)(60.14,79.64) -- (190.44,79.64)(60.14,89.64) -- (190.44,89.64)(60.14,99.64) -- (190.44,99.64)(60.14,109.64) -- (190.44,109.64)(60.14,119.64) -- (190.44,119.64) ; \draw  [color={rgb, 255:red, 155; green, 155; blue, 155 }  ,draw opacity=0.77 ]  ;
\draw    (110.14,89.64) -- (138.14,89.64) ;
\draw [shift={(140.14,89.64)}, rotate = 180] [color={rgb, 255:red, 0; green, 0; blue, 0 }  ][line width=0.75]    (10.93,-3.29) .. controls (6.95,-1.4) and (3.31,-0.3) .. (0,0) .. controls (3.31,0.3) and (6.95,1.4) .. (10.93,3.29)   ;

\draw (162.14,81.04) node [anchor=north west][inner sep=0.75pt]  [font=\tiny]  {$\overline{p}$};
\draw (92.14,81.04) node [anchor=north west][inner sep=0.75pt]  [font=\tiny]  {$\overline{p}$};

\end{tikzpicture}}
    \caption{ Relation $(1)$ occurs. Assuming that $\#\Z_{j+1}\ge\#\Z_j$, and $\overline{p} \in \zero_{j+1}$ such that $u_{j+1}(\overline{p})=-1$, we perform the transformation that reduces the total energy.}
    \label{fig:1.6_1}
\end{figure}

\noindent {\it Case \eqref{situation:2}.} In this range of parameters if $\#\Z_{j+1}\ge\#\Z_j$, then it is convenient to perform the transition represented in Figure \ref{fig:1.6_2}.
More precisely, whenever there exists $\overline{p} \in \partial^+\I_{j+1}$ such that $u_{j+1}(\overline{p})=-1$, $\#\mathcal{N}(\overline{p})\cap I_{j+1}=1$ and $\#\mathcal{N}(\overline{p})\cap Z_{j+1}\geq 1$, the competitor $\tilde{u}_{j+1}$ obtained by replacing $u_{j+1}(\overline{p})\mapsto 0$ satisfies $\dis^1_\e(\tilde{u}_{j+1},u_j)=\dis^1_\e(u_{j+1},u_j)$ and, if (2) holds, then $$\E_\e(\tilde{u}_{j+1})-\E_\e(u_{j+1})+\frac{\e^2}{\tau}\left(\dis^0_\e(\tilde{u}_{j+1},u_j)-\dis^0_\e(u_{j+1},u_j)\right)\le3\e(1-k)-2\e+\frac{\e}{\tempo}<0.$$
  \begin{figure}[H]
    \centering
    \resizebox{0.25\textwidth}{!}{\tikzset{every picture/.style={line width=0.75pt}} 

\begin{tikzpicture}[x=0.75pt,y=0.75pt,yscale=-1,xscale=1]

\draw  [draw opacity=0][fill={rgb, 255:red, 74; green, 144; blue, 226 }  ,fill opacity=1 ] (350.33,80.03) -- (360.94,80.03) -- (360.94,90.64) -- (350.33,90.64) -- cycle ;
\draw  [draw opacity=0][fill={rgb, 255:red, 74; green, 144; blue, 226 }  ,fill opacity=1 ] (350.53,70.33) -- (360.33,70.33) -- (360.33,80.14) -- (350.53,80.14) -- cycle ;
\draw  [draw opacity=0][fill={rgb, 255:red, 74; green, 144; blue, 226 }  ,fill opacity=1 ] (280.53,70.53) -- (290.33,70.53) -- (290.33,80.33) -- (280.53,80.33) -- cycle ;
\draw  [color={rgb, 255:red, 208; green, 2; blue, 27 }  ,draw opacity=1 ][fill={rgb, 255:red, 208; green, 2; blue, 27 }  ,fill opacity=1 ] (280.33,70.33) -- (280.33,90.33) -- (270.33,90.33) -- (270.33,70.33) -- cycle ;
\draw  [color={rgb, 255:red, 208; green, 2; blue, 27 }  ,draw opacity=1 ][fill={rgb, 255:red, 208; green, 2; blue, 27 }  ,fill opacity=1 ] (350.33,70.33) -- (350.33,90.33) -- (340.33,90.33) -- (340.33,70.33) -- cycle ;
\draw  [draw opacity=0][fill={rgb, 255:red, 126; green, 211; blue, 33 }  ,fill opacity=1 ] (280.33,80.72) -- (290.14,80.72) -- (290.14,90.53) -- (280.33,90.53) -- cycle ;
\draw  [draw opacity=0] (250.33,100.33) -- (380.44,100.33) -- (380.44,60.33) -- (250.33,60.33) -- cycle ; \draw  [color={rgb, 255:red, 155; green, 155; blue, 155 }  ,draw opacity=0.77 ] (250.33,100.33) -- (250.33,60.33)(260.33,100.33) -- (260.33,60.33)(270.33,100.33) -- (270.33,60.33)(280.33,100.33) -- (280.33,60.33)(290.33,100.33) -- (290.33,60.33)(300.33,100.33) -- (300.33,60.33)(310.33,100.33) -- (310.33,60.33)(320.33,100.33) -- (320.33,60.33)(330.33,100.33) -- (330.33,60.33)(340.33,100.33) -- (340.33,60.33)(350.33,100.33) -- (350.33,60.33)(360.33,100.33) -- (360.33,60.33)(370.33,100.33) -- (370.33,60.33)(380.33,100.33) -- (380.33,60.33) ; \draw  [color={rgb, 255:red, 155; green, 155; blue, 155 }  ,draw opacity=0.77 ] (250.33,100.33) -- (380.44,100.33)(250.33,90.33) -- (380.44,90.33)(250.33,80.33) -- (380.44,80.33)(250.33,70.33) -- (380.44,70.33)(250.33,60.33) -- (380.44,60.33) ; \draw  [color={rgb, 255:red, 155; green, 155; blue, 155 }  ,draw opacity=0.77 ]  ;
\draw    (300.33,80.33) -- (328.33,80.33) ;
\draw [shift={(330.33,80.33)}, rotate = 180] [color={rgb, 255:red, 0; green, 0; blue, 0 }  ][line width=0.75]    (10.93,-3.29) .. controls (6.95,-1.4) and (3.31,-0.3) .. (0,0) .. controls (3.31,0.3) and (6.95,1.4) .. (10.93,3.29)   ;

\draw (281.33,80.73) node [anchor=north west][inner sep=0.75pt]  [font=\tiny]  {$\overline{p}$};
\draw (351.33,80.43) node [anchor=north west][inner sep=0.75pt]  [font=\tiny]  {$\overline{p}$};

\end{tikzpicture}}
    \caption{ 
     Relation $(2)$ occurs. Assuming that $\#\Z_{j+1}\ge\#\Z_j$, and $\overline{p} \in \partial^+\I_{j+1}$ such that $u_{j+1}(\overline{p})=-1$, $\#\mathcal{N}(\overline{p})\cap I_{j+1}=1$ and $\#\mathcal{N}(\overline{p})\cap Z_{j+1}\geq 1$ we perform the transformation that reduces the total energy.}
    \label{fig:1.6_2}
\end{figure}
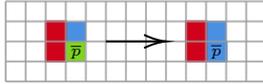

\noindent {\it Case \eqref{situation:3}.}
Finally if $\#Z_{j+1}\leq \#Z_j$, then in this range of parameters it is convenient to perform the two transitions represented in Figure \ref{fig:1.6_3}. 
Indeed, let $\mu<1/4$, and suppose that there exists $\overline{p}$ which is a surfactant in the corner according to Definition~\ref{def: corner unit}. Then, from Proposition~\ref{prop:connectedness} both the competitor $\tilde{u}_{j+1}$ (that we do not rename) obtained by replacing $u_{j+1}(\overline{p})\mapsto 1$, when $\#(\nn(\overline{p})\cap \{u_{j+1}=1\})=2$ or the competitor $\tilde{u}_{j+1}$ obtained by replacing $u_{j+1}\mapsto -1$ when $\#(\nn(\overline{p})\cap \{u_{j+1}=-1\})=2$,  satisfies \[\frac{1}{\tau}\Bigl|\dis^{1}_\e(\tilde{u}_{j+1},u_j)-\dis^1_\e(u_{j+1},u_j)\Bigl|\le c\e^{1+\mu},\] and, if (3) holds, then
\begin{equation*}
    \E_\e(\tilde{u}_1)-\E_\e(u_1)+\frac{\e^2}{\tau}\left(\dis^0_\e(\tilde{u}_1,u_0)-\dis^0_\e(u_1,u_0)\right)=-2\e(1-k)+\frac{\e}{\tempo}<0.
\end{equation*}

 \begin{figure}[H]
    \centering
    \resizebox{0.25\textwidth}{!}{\tikzset{every picture/.style={line width=0.75pt}} 

\begin{tikzpicture}[x=0.75pt,y=0.75pt,yscale=-1,xscale=1]

\draw  [color={rgb, 255:red, 126; green, 211; blue, 33 }  ,draw opacity=1 ][fill={rgb, 255:red, 126; green, 211; blue, 33 }  ,fill opacity=1 ] (470.14,119.64) -- (470.14,129.64) -- (460.14,129.64) -- (460.14,119.64) -- cycle ;
\draw  [color={rgb, 255:red, 126; green, 211; blue, 33 }  ,draw opacity=1 ][fill={rgb, 255:red, 126; green, 211; blue, 33 }  ,fill opacity=1 ] (550.14,119.64) -- (550.14,129.64) -- (540.14,129.64) -- (540.14,119.64) -- cycle ;
\draw  [color={rgb, 255:red, 126; green, 211; blue, 33 }  ,draw opacity=1 ][fill={rgb, 255:red, 126; green, 211; blue, 33 }  ,fill opacity=1 ] (470.14,129.64) -- (480.14,129.64) -- (480.14,139.64) -- (470.14,139.64) -- cycle ;
\draw  [color={rgb, 255:red, 126; green, 211; blue, 33 }  ,draw opacity=1 ][fill={rgb, 255:red, 126; green, 211; blue, 33 }  ,fill opacity=1 ] (550.14,129.64) -- (560.14,129.64) -- (560.14,139.64) -- (550.14,139.64) -- cycle ;
\draw  [color={rgb, 255:red, 74; green, 144; blue, 226 }  ,draw opacity=1 ][fill={rgb, 255:red, 74; green, 144; blue, 226 }  ,fill opacity=1 ] (470.14,109.64) -- (480.14,109.64) -- (480.14,119.64) -- (470.14,119.64) -- cycle ;
\draw  [color={rgb, 255:red, 74; green, 144; blue, 226 }  ,draw opacity=1 ][fill={rgb, 255:red, 74; green, 144; blue, 226 }  ,fill opacity=1 ] (470.14,119.64) -- (480.14,119.64) -- (480.14,129.64) -- (470.14,129.64) -- cycle ;
\draw  [color={rgb, 255:red, 74; green, 144; blue, 226 }  ,draw opacity=1 ][fill={rgb, 255:red, 74; green, 144; blue, 226 }  ,fill opacity=1 ] (480.14,119.64) -- (490.14,119.64) -- (490.14,129.64) -- (480.14,129.64) -- cycle ;
\draw  [color={rgb, 255:red, 155; green, 155; blue, 155 }  ,draw opacity=0.77 ][fill={rgb, 255:red, 74; green, 144; blue, 226 }  ,fill opacity=1 ] (550.14,109.64) -- (560.14,109.64) -- (560.14,119.64) -- (550.14,119.64) -- cycle ;
\draw  [color={rgb, 255:red, 155; green, 155; blue, 155 }  ,draw opacity=0.77 ][fill={rgb, 255:red, 74; green, 144; blue, 226 }  ,fill opacity=1 ] (560.14,119.64) -- (570.14,119.64) -- (570.14,129.64) -- (560.14,129.64) -- cycle ;
\draw  [color={rgb, 255:red, 126; green, 211; blue, 33 }  ,draw opacity=1 ][fill={rgb, 255:red, 126; green, 211; blue, 33 }  ,fill opacity=1 ] (550.14,119.64) -- (560.14,119.64) -- (560.14,129.64) -- (550.14,129.64) -- cycle ;
\draw  [color={rgb, 255:red, 208; green, 2; blue, 27 }  ,draw opacity=1 ][fill={rgb, 255:red, 208; green, 2; blue, 27 }  ,fill opacity=1 ] (470.14,79.64) -- (470.14,89.64) -- (460.14,89.64) -- (460.14,79.64) -- cycle ;
\draw  [color={rgb, 255:red, 208; green, 2; blue, 27 }  ,draw opacity=1 ][fill={rgb, 255:red, 208; green, 2; blue, 27 }  ,fill opacity=1 ] (550.14,79.64) -- (550.14,89.64) -- (540.14,89.64) -- (540.14,79.64) -- cycle ;
\draw  [color={rgb, 255:red, 208; green, 2; blue, 27 }  ,draw opacity=1 ][fill={rgb, 255:red, 208; green, 2; blue, 27 }  ,fill opacity=1 ] (470.14,89.64) -- (480.14,89.64) -- (480.14,99.64) -- (470.14,99.64) -- cycle ;
\draw  [color={rgb, 255:red, 208; green, 2; blue, 27 }  ,draw opacity=1 ][fill={rgb, 255:red, 208; green, 2; blue, 27 }  ,fill opacity=1 ] (550.14,89.64) -- (560.14,89.64) -- (560.14,99.64) -- (550.14,99.64) -- cycle ;
\draw  [color={rgb, 255:red, 74; green, 144; blue, 226 }  ,draw opacity=1 ][fill={rgb, 255:red, 74; green, 144; blue, 226 }  ,fill opacity=1 ] (470.14,69.64) -- (480.14,69.64) -- (480.14,79.64) -- (470.14,79.64) -- cycle ;
\draw  [color={rgb, 255:red, 74; green, 144; blue, 226 }  ,draw opacity=1 ][fill={rgb, 255:red, 74; green, 144; blue, 226 }  ,fill opacity=1 ] (470.14,79.64) -- (480.14,79.64) -- (480.14,89.64) -- (470.14,89.64) -- cycle ;
\draw  [color={rgb, 255:red, 74; green, 144; blue, 226 }  ,draw opacity=1 ][fill={rgb, 255:red, 74; green, 144; blue, 226 }  ,fill opacity=1 ] (480.14,79.64) -- (490.14,79.64) -- (490.14,89.64) -- (480.14,89.64) -- cycle ;
\draw  [color={rgb, 255:red, 155; green, 155; blue, 155 }  ,draw opacity=0.77 ][fill={rgb, 255:red, 74; green, 144; blue, 226 }  ,fill opacity=1 ] (550.14,69.64) -- (560.14,69.64) -- (560.14,79.64) -- (550.14,79.64) -- cycle ;
\draw  [color={rgb, 255:red, 155; green, 155; blue, 155 }  ,draw opacity=0.77 ][fill={rgb, 255:red, 74; green, 144; blue, 226 }  ,fill opacity=1 ] (560.14,79.64) -- (570.14,79.64) -- (570.14,89.64) -- (560.14,89.64) -- cycle ;
\draw  [color={rgb, 255:red, 155; green, 155; blue, 155 }  ,draw opacity=0.77 ][fill={rgb, 255:red, 208; green, 2; blue, 27 }  ,fill opacity=1 ] (550.14,79.64) -- (560.14,79.64) -- (560.14,89.64) -- (550.14,89.64) -- cycle ;
\draw  [draw opacity=0] (440.14,59.64) -- (590.33,59.64) -- (590.33,150) -- (440.14,150) -- cycle ; \draw  [color={rgb, 255:red, 155; green, 155; blue, 155 }  ,draw opacity=0.77 ] (440.14,59.64) -- (440.14,150)(450.14,59.64) -- (450.14,150)(460.14,59.64) -- (460.14,150)(470.14,59.64) -- (470.14,150)(480.14,59.64) -- (480.14,150)(490.14,59.64) -- (490.14,150)(500.14,59.64) -- (500.14,150)(510.14,59.64) -- (510.14,150)(520.14,59.64) -- (520.14,150)(530.14,59.64) -- (530.14,150)(540.14,59.64) -- (540.14,150)(550.14,59.64) -- (550.14,150)(560.14,59.64) -- (560.14,150)(570.14,59.64) -- (570.14,150)(580.14,59.64) -- (580.14,150)(590.14,59.64) -- (590.14,150) ; \draw  [color={rgb, 255:red, 155; green, 155; blue, 155 }  ,draw opacity=0.77 ] (440.14,59.64) -- (590.33,59.64)(440.14,69.64) -- (590.33,69.64)(440.14,79.64) -- (590.33,79.64)(440.14,89.64) -- (590.33,89.64)(440.14,99.64) -- (590.33,99.64)(440.14,109.64) -- (590.33,109.64)(440.14,119.64) -- (590.33,119.64)(440.14,129.64) -- (590.33,129.64)(440.14,139.64) -- (590.33,139.64)(440.14,149.64) -- (590.33,149.64) ; \draw  [color={rgb, 255:red, 155; green, 155; blue, 155 }  ,draw opacity=0.77 ]  ;
\draw    (500.14,89.64) -- (528.14,89.64) ;
\draw [shift={(530.14,89.64)}, rotate = 180] [color={rgb, 255:red, 0; green, 0; blue, 0 }  ][line width=0.75]    (10.93,-3.29) .. controls (6.95,-1.4) and (3.31,-0.3) .. (0,0) .. controls (3.31,0.3) and (6.95,1.4) .. (10.93,3.29)   ;
\draw    (500.14,129.64) -- (528.14,129.64) ;
\draw [shift={(530.14,129.64)}, rotate = 180] [color={rgb, 255:red, 0; green, 0; blue, 0 }  ][line width=0.75]    (10.93,-3.29) .. controls (6.95,-1.4) and (3.31,-0.3) .. (0,0) .. controls (3.31,0.3) and (6.95,1.4) .. (10.93,3.29)   ;

\draw (471.14,121.04) node [anchor=north west][inner sep=0.75pt]  [font=\tiny]  {$\overline{p}$};
\draw (551.14,121.04) node [anchor=north west][inner sep=0.75pt]  [font=\tiny]  {$\overline{p}$};
\draw (471.14,81.04) node [anchor=north west][inner sep=0.75pt]  [font=\tiny]  {$\overline{p}$};
\draw (552.14,81.04) node [anchor=north west][inner sep=0.75pt]  [font=\tiny]  {$\overline{p}$};

\end{tikzpicture}}
    \caption{ 
     Relation $(3)$ occurs. We assume that $\#\Z_{j+1}\le\#\Z_j$ and on the left side of the figure we have  $\overline{p}$ surfactant in the corner. On the right side we represent the transformation that reduces the total energy of the system.}
    \label{fig:1.6_3}
\end{figure}
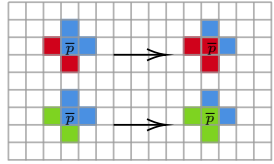

\noindent{\bf The effect of \eqref{situation:3} with initial data surrounded by surfactant.} In this paragraph we discuss the minimizing movements under the assumption that the initial data verify $\#\partial^+I_0\le\# Z_0$. The evolution depends only on whether \eqref{situation:3} holds or not.

\begin{lemma}\label{lemma: il minimizzante circondato da surfactant}
    Let $\gamma = 2$, let $u_j\in\A_\e$ be satisfying assumption \eqref{H} of Subsection \ref{subsub_H}, and let $u_{j+1}$ be a minimizer of $\enF(\cdot, u_j)$. Suppose that $\partial^+I_j\subset Z_j$. Then we have the following.
    \begin{itemize}
        \item[(i)] If $\displaystyle\tempo<\frac{1}{2-2k}$ (i.e. \eqref{situation:3} does not hold ``strictly'') then there exists $\bar{\e}$ such that for every $\e\le\bar{\e}$ it holds $\# Z_{j+1} = \# Z_j.$ In particular \cite[Theorem 5.14]{CFS} holds even for $\gamma=2$.
        \item[(ii)]If $\displaystyle\tempo\ge\frac{1}{2-2k}$ (i.e. \eqref{situation:3} is verified ``weakly'') and if $I_j$ is an octagon, then there exists $\bar{\e}$ such that for every $\e\le\bar{\e}$ the set $I_{j+1}$ is an octagon. Moreover,
        if $\zeta >\frac{1}{2-2k}$ then $Z_{j+1} = \partial^+I_{j+1};$ if instead $\zeta = \frac{1}{2-2k}$ then for every minimizer $u_{j+1}$ there exists an other minimizer $\tilde{u}_{j+1}$ such that 
        \begin{equation}\label{eq:4.2}
            \{\tilde{u}_{j+1}=1\} = I_{j+1},\text{ and }\{\tilde{u}_{j+1}=0\} = \partial^+I_{j+1.}
        \end{equation}
    \end{itemize}
\end{lemma}
\begin{proof}
    From \cite[Lemma 5.4]{CFS} we have that $\# Z_{j+1}\le\# Z_{j}$ and $\partial^+I_{j+1}\subset Z_{j+1}$.\\
    
    We prove (i).
    Suppose by contradiction that $\#\Z_{j+1}<\#\Z_{j}$. We consider any $\bar{p} \in \one_{j+1}$ (see Definition~\ref{def: corner unit}), and we define a competitor $\tilde{u}_{j+1}$ replacing $u_{j+1}(\bar{p})\mapsto0$. In this way we get 
    \begin{equation}\label{eq:energy_diff}
        \E_\e(\tilde{u}_{j+1})-\E_\e(u_{j+1})+\frac{\e}{\tempo}\left(\dis^{0}_\e(\tilde{u}_{j+1},u_j)-\dis^{0}_\e(u_1,u_0)\right)\leq 2\e(1-k)-\frac{\e}{\tempo}<0,
    \end{equation}
    in view of our assumption. Setting $\mu<1/4$, in view of Proposition~\ref{prop:connectedness}, for $\e$ sufficiently small we have $\frac{1}{\e}|\dis^1_\e(\tilde{u}_{j+1},u_j)-\dis^1_\e(u_{j+1},u_j)|\le c\e^{1+\mu}$, i.e. the increment of $\dis^1_\e$ is negligible with respect to \eqref{eq:energy_diff}. This contradicts the minimality of $u_{j+1}$, therefore we get the claim.\\
    
    We now prove (ii). Let us assume that \eqref{situation:3} holds. By contradiction we suppose that $Z_{j+1}\not\subseteq \partial^+I_{j+1}$, then there exists $p\in Z_{j+1}\setminus\partial^+I_{j+1}$ such that $\#\nn(p)\cap\{u_{j+1} = -1\}\ge 2$. We can perform the transformation depicted in Figure \ref{fig:1.6_3}, that reduces the total energy of the system and this is a contradiction in view of the minimality of $u_{j+1}$. Suppose instead that $\tempo = 1/(2-2k)$ and again that $Z_{j+1}\not\subseteq \partial^+I_{j+1}$. As before, there exists $p\in Z_{j+1}\setminus\partial^+I_{j+1}$ such that $\#\nn(p)\cap\{u_{j+1} = -1\}\ge 2$. Then the competitor $u_{j+1, 1}$ obtained by replacing $u_{j+1}(p)\mapsto -1$ is a minimizer since 
    \[\E_\e(u_{j+1, 1})-\E_\e(u_{j+1})+\frac{\e}{\tempo}\Bigl(\dis^0_\e(u_{j+1, 1}, u_j)-\dis^0_\e(u_{j+1}, u_j)\Bigl)\le -2\e(1-k)+\frac{\e}{\tempo} = 0.\]
    Let $n:=\# Z_{j+1}\setminus\partial^+I_{j+1}.$ We can define inductively a sequence of minimizers $u_{j+1, i}:i=1, \dots, n$. At step $i\ge 2$ we find 
    \[p^i\in \{u_{j+1, i-1} = 0\}\setminus\partial^+I_{j+1}\text{  such that }\#\nn(p)\cap\{u_{j+1, i-1} = -1\}\ge 2,\]
    and we define $u_{j+1, i}$ by replacing $u_{j+1}(p^i)\mapsto -1.$ Setting $\tilde{u}_{j+1}:=u_{j+1, n}$ we have by construction that $\tilde{u}_{j+1}$ is a minimizer which satisfies \eqref{eq:4.2}.

    Now we prove that $I_{j+1}$ is an octagon. In view of the discussion above we can suppose without loss of generality that $Z_{j+1} = \partial^+I_{j+1}$. In particular we deduce that $I_{j+1}\cup Z_{j+1}$ is not a rectangle. In view of Remark~\ref{rem:geometry_quasi_octagon} we have to show that for every pair of points $\{p, q\}\subset \partial^-I_{j+1}$ such that $q = p+\e e_i$ there exists $j\in\{1, \dots, 4\}$ such that $\{p, q\}\subset\PP_j$. This follows from the second part of Lemma~\ref{lemma:shape_optimality}. In order to show it we argue by contradiction, assuming without loss of generality that there exists $q = p+\e e_1$ such that $\{p, q\}\subset H^{j+1, i}$ for $i\ge 2$ and that $\{p-\e e_2, q-\e e_2\}\subset\partial^+I_{j+1}$, where we denoted by $H^{j+1, i},i=1, \dots, n_h$ the horizontal slices of $I_{j+1}$. Then we have that case (i) of Lemma~\ref{lemma:shape_optimality} cannot occur since $I_{j}$ is an octagon, and case (ii) of the same lemma cannot occur since $Z_{j+1} = \partial^+I_{j+1}$. Finally, case (iii) of Lemma~\ref{lemma:shape_optimality} cannot occur since $I_{j+1}\cup Z_{j+1}$ is not a rectangle. This is a contradiction, and therefore the proof is complete.
\end{proof}

We now state the theorem which describes the sides displacements under assumption (ii) of Lemma~\ref{lemma: il minimizzante circondato da surfactant}.

\begin{theorem}\label{teo:(1)_(2)_valgono}
    Let us assume that $\gamma=2$ and that $\tempo\ge1/(2-2k)$ (i.e. \eqref{situation:3} holds ``weakly''). Let $A\subset\rr^2$ be a not degenerate octagon, and for every $\e>0$ let $A_\e$ be an octagon such that $d_{\mathcal{H}}(A, A_\e)\to 0$ as $\e\to 0$. Let $(u_0^\e)_\e$ be such that $u_0^\e\in\A_\e$ for every $\e$, $\I^\e_0= A_\e\cap\e\ZZ^2$, and $Z^\e_0 = \partial^+I_0^\e$. For $\e>0$ let $(u^\e_j)_{j\ge 0}$ be a family of minimizing movements with initial data $u^\e_0$. For every $t\ge 0$ we denote by $A_\e(t)$ the set 
    \begin{equation*}
    A_\e(t) = \bigcup_{i\in\I^\e_{\lfloor t/\tau\rfloor}}Q_\e(i).
    \end{equation*}
    Then, it holds that $A_\e(t)$ converges as $\e\to 0$ locally uniformly in time to $A(t).$ The set
     $A(t)$ is an octagon for every $t\geq 0$
     such that $A(0) = A$. Denoting by $\PP_i(t),\,\DD_i(t)$ its parallel and diagonal sides, and by $P_i(t),\,D_i(t)$ their lengths, then every side of $A(t)$ moves inwards remaining parallel to the side itself. The velocities $v_{\PP_i}(t)$ and $v_{\DD_i}(t)$ of the $i$-th parallel and diagonal side satisfy the following differential inclusion 
\begin{align}\label{eq:sides_movement_octagon_gamma=2}
        &v_{\PP_i}(t)\begin{cases}
            =\frac{1}{\tempo}\Bigl\lfloor\frac{2\tempo(1-k)}{P_{i}(t)}\Bigl\rfloor&\text{ if }\frac{2\tempo(1-k)}{{P_{i}(t)}}\not\in \NN,\\[.5em]
            \in\Bigl[\frac{1}{\tempo}\Bigl(\frac{2\tempo(1-k)}{{P_{i}(t)}}-1\Bigl),\,\frac{2(1-k)}{{P_{i}(t)}}\Bigl]&\text{ otherwise},
        \end{cases}\\[1em]
        \label{eq:sides_movement_octagon_gamma=2.2}        &v_{\DD_i}(t)\begin{cases}     =\frac{\sqrt{2}}{2\tempo}\bigg\lfloor\frac{2\sqrt{2}(1-k)\tempo-\sqrt{2}}{D_{i}(t)}\bigg\rfloor&\text{ if }\frac{2\sqrt{2}(1-k)\tempo-\sqrt{2}}{D_{i}(t)}\not\in \NN,\\[.5em]
        \in\Bigl\{\frac{\sqrt{2}}{2\tempo}\Bigl(\frac{2\sqrt{2}(1-k)\tempo-\sqrt{2}}{D_i(t)}-1\Bigl),\,\frac{\sqrt{2}}{2\tempo}\Bigl(\frac{2\sqrt{2}(1-k)\tempo-\sqrt{2}}{D_i(t)}\Bigl)\Bigl\}&\text{otherwise}.
        \end{cases}
    \end{align}
\end{theorem}
\begin{proof}
Let $\bar{c}>0$ be a positive constant, and let $\bar{\e}$ be given by Proposition~\ref{prop:connectedness}. 
Then from Proposition~\ref{prop:connectedness} and Lemma~\ref{lemma: il minimizzante circondato da surfactant}, we know that for every $\e\le\bar{\e}$ the set $\I^\e_{j+1}$ is an octagon contained into $\I^\e_j$, as long as $I_j^\e$ verifies assumption \eqref{H} of Subsection~\ref{subsub_H}. Moreover again from Lemma~\ref{lemma: il minimizzante circondato da surfactant} we know that there exists a minimizer $u_{j+1}$ such that $\Z_{j+1}^\e = \partial^+\I^\e_{j+1}$ for $\e\le\bar{\e}$. As a result the competitors in the minimum problem will be assumed to be octagons. Moreover, since $A$ is not degenerate, we can suppose that $D^\e_{0,i}>0$ for every $i=1,\dots,4$.

From now on we omit the dependence on $\e$, we write $\alpha_{i},\,\beta_i$ in place of $\alpha_{j, i}$ and $\beta_{j, i}$, and we denote by $P_{j, i},\,D_{j, i}$ the lengths of the parallel and diagonal sides of $I_j$. 
To compute the sides displacements we observe that
\begin{equation}
    \E_\e(u_{j+1})=\sum_{i=1}^{4}\frac{D_{j+1,i}}{\sqrt{2}\e}4\e(1-k)+3\e(1-k)\sum_{i=1}^4\frac{P_{j+1,i}}{\e}+4\e(1-k).
    \label{energy_gamma=2}
\end{equation}
The dissipation is given by
\begin{equation*}
\begin{split}
    \frac{1}{\tau}\dis_\e^1(u_{j+1}, u_j) &=\frac{\e}{\tempo}\sum_{i=1}^4\frac{D_{j,i}}{\sqrt{2}}\frac{\beta_i(\beta_i+1)}{2}+\frac{\e}{\tempo}\sum_{i=1}^4P_{j,i}\frac{\alpha_i(\alpha_i+1)}{2}+e(\alpha_i, \beta_i)\\
    \frac{\e^\gamma}{\tau}\dis^0_\e(u_{j+1}, u_j) &=\frac{\e}{\tempo}\left[\som \frac{P_{j,i}-P_{j+1,i}}{\e}+\som \frac{{D_{j,i}}-{D_{j+1,i}}}{\sqrt{2}\e}\right],
\end{split}
\end{equation*}
where $e(\alpha_i, \beta_i)$ is an error term which satisfies $|e(\alpha_i, \beta_i)|\le c\e^2$. 
From \eqref{eq: spostamenti} and \eqref{eq: Lung. lati} we get that
\begin{equation*}
    \left[\som \frac{P_{j,i}-P_{j+1,i}}{\e}+\som \frac{D_{j,i}-D_{j+1,i}}{\sqrt{2}\e}\right]=\som \beta_i.
\end{equation*}
Then, \eqref{energy_gamma=2} can be rewritten as
\begin{equation*}
    \E_\e(u_{j+1})=-2\e(1-k)\som\alpha_i-2\e(1-k)\som\beta_i+\frac{4(1-k)}{\sqrt{2}}\som D_{j,i}+3(1-k)\som P_{j,i} +4\e(1-k).
\end{equation*}
Hence, neglecting the error term $e(\alpha_i, \beta_i)$ which is of order $\e^2$, the function that we have to minimize is 
\begin{equation*}
    -2\e(1-k)\som\alpha_i+\left(\frac{\e}{\tempo}-2\e(1-k)\right)\som\beta_i+\frac{\e}{\tempo}\left[\som\frac{D_{j,i}}{\sqrt{2}}\frac{\beta_i(\beta_i+1)}{2}+\sum_{i=1}^4P_{j,i}\frac{\alpha_i(\alpha_i+1)}{2}\right].
\end{equation*}
The lengths of the sides $\PP_{j+1,i}$ and $\DD_{j+1,i}$ are determined by computing respectively
\begin{equation*}
    \alpha_{j,i}=\argmin\left\{-2(1-k)\alpha_i+\frac{1}{\tempo}P_{j,i}\frac{\alpha_i(\alpha_i+1)}{2}+\frac{e(\alpha_i, \beta_i)}{\e}: \alpha_i\in \NN\right\},
\end{equation*}
\begin{equation*}
    \beta_{j,i}=\argmin\left\{\left(\frac{1}{\tempo}-2(1-k)\right)\beta_i+\frac{1}{\tempo}\frac{D_{j ,i}}{\sqrt{2}}\frac{\beta_i(\beta_i+1)}{2}+\frac{e(\alpha_i, \beta_i)}{\e}: \beta_i\in \NN\right\}.
\end{equation*}
By direct computation we find 
\begin{align}\label{eq:sides_movement_octagon_gamma=20}
        &\alpha_{j,i}\begin{cases}
            =\Bigl\lfloor\frac{2\tempo(1-k)}{P_{j,i}}\Bigl\rfloor&\text{ if }\text{ dist}\Bigl(\frac{2\tempo(1-k)}{P_{j,i}}, \NN\Bigl)\ge c\e,\\[.5em]
            \in\Bigl\{\Bigl[\frac{2\tempo(1-k)}{P_{j,i}}\Bigl]-1,\,\Bigl[\frac{2\tempo(1-k)}{P_{j,i}}\Bigl]\Bigl\}&\text{ otherwise},
        \end{cases}\\[1em]
        \label{eq:sides_movement_octagon_gamma=2.20}        &\beta_{j,i}\begin{cases}     =\Big\lfloor\frac{2\sqrt{2}(1-k)\tempo-\sqrt{2}}{D_{j,i}}\Big\rfloor&\text{ if }\text{dist}\Bigl(\frac{2\sqrt{2}(1-k)\tempo-\sqrt{2}}{D_{j,i}}, \NN\Bigl)\ge c\e\\[.5em]
        \in\Bigl\{\Bigl[\frac{2\sqrt{2}(1-k)\tempo-\sqrt{2}}{D_{j,i}}\Bigl]-1,\,\Bigl[\frac{2\sqrt{2}(1-k)\tempo-\sqrt{2}}{D_{j,i}}\Bigl]\Bigl\}&\text{otherwise}.
        \end{cases}
    \end{align}
Recursively repeating the argument above, we find sequences $D^\e_{j, i},\,P^\e_{j, i}$ and $\alpha^\e_{j, i},\,\beta^\e_{j, i}$ which satisfy (\ref{eq: Lung. lati}), i.e.  
 \begin{equation*}
       \begin{split}
           &P^\e_{j+1,i}=P^\e_{j,i}+2\e \alpha^\e_{j,i}-\e(\beta^\e_{j,i-1}+\beta^\e_{j,i})\\           &D^\e_{j+1,i}=D^\e_{j,i}+\sqrt{2}\e\beta^\e_{j,i}-\sqrt{2}\e(\alpha^\e_{j,i}+\alpha^\e_{j,i+1}).
       \end{split}
\end{equation*}
To conclude, we can argue as in the first part of the proof of Theorem~ \ref{teo:minimizing_movements_surfactant_covering_more_than_diagonals}, observing that \eqref{eq:sides_movement_octagon_gamma=2} and \eqref{eq:sides_movement_octagon_gamma=2.2} follow by passing to the limit in \eqref{eq:sides_movement_octagon_gamma=20} and in \eqref{eq:sides_movement_octagon_gamma=2.20} respectively. We only point out that in \eqref{eq:sides_movement_octagon_gamma=2.2} the multiplicative factor $\frac{\sqrt{2}}{2}$ is needed to take into account that $v_{\DD_i}$ is the velocity in the direction which is orthogonal to the side.
\end{proof}

\noindent{\bf The case \eqref{situation:1} does not hold.}
In view of previous paragraph, from now on we may consider only the situation in which the minimizing set $I_j$ is not surrounded by surfactant. In this paragraph we suppose that $\tempo\le\frac{1}{4k}$, i.e. \eqref{situation:1} does not hold.

\begin{proposition}\label{prop:situation_1_not_occurs}
        Let $\gamma=2$, let $u_j\in \mathcal{A}_\e$ which satisfies assumption \eqref{H} of Subsection \ref{subsub_H}, and let $u_{j+1}$ be a minimizer of $\enF(\cdot, u_j)$. Then we have the following.
    \begin{itemize}
        \item[(i)] If $\displaystyle\tempo<\frac{1}{4k}$ (i.e. \eqref{situation:1} does not hold ``strictly'')     then there exists $\bar{\e}$ such that for every $\e\le\bar
        \e$ it holds $\#\Z_{j+1} = \#\Z_{j}$. In particular all the analysis of case $\gamma<2$ holds with no changes, and we refer to Section~\ref{sec:gamma<2} and \cite[Section 5]{CFS} for a complete study.
        \item[(ii)] Let $\mu<1/4$. Suppose that $\displaystyle \tempo = \frac{1}{4k}$ and that $I_j$ is a quasi-octagon such that  \begin{equation*}
         \#\partial^+\I_j> \# Z_j + 8\e^{\mu-1}.
         \end{equation*}
          Then there exists $\bar{\e}$ such that for every $\e\le\bar{\e}$ any minimizer $u_{j+1}$ is a quasi-octagon with $\#Z_{j+1}\ge\#Z_j$. Moreover, for every minimizer $u_{j+1}$ there exists an other minimizer $\tilde{u}_{j+1}$ such that $\#\{\tilde{u}_{j+1}=0\} = \# Z_j$, and $\{\tilde{u}_{j+1} = 1\} = \{u_{j+1}=1\}.$ Moreover, if $\zero_{j}\subset Z_j$ then for every minimizer $u_{j+1}$ it holds $\# Z_{j+1} = \#Z_j.$
    \end{itemize}
\end{proposition}
\begin{proof}
Let us prove (i). In view of Proposition~\ref{prop:connectedness} there exists $\bar{\e}$ such that for every $\e\le\bar
    {\e}$ the set $I_{j+1}$ is a staircase set. We want to prove that $\#Z_{j+1}=\#Z_j$. Let us assume first that $\#\Z_{j+1}>\#\Z_j$. Without loss of generality we may suppose that there exists $\overline{p} \in \partial^-\Z_{j+1}$ such that $u_{j+1}(\bar{p}+\e e_2) = u_{j+1}(\bar{p}+\e e_1) = -1$, and we define $\tilde{u}_{j+1}$ replacing $u_{j+1}(\overline{p})\mapsto -1$. Since $\dis^1_\e(\tilde{u}_{j+1},u_j)= \dis^1_\e(u_{j+1},u_j)$ and $\zeta<\frac{1}{4k}$ we get $$\E_\e(\tilde{u}_{j+1})-\E_\e(u_{j+1})+\frac{\e}{\tempo}\left(\dis^{0}_\e(\tilde{u}_{j+1},u_j)-\dis^{0}_\e(u_{j+1},u_j)\right)\leq 4\e-4\e(1-k)-\frac{\e}{\tempo}<0$$
     which is in contradiction with the minimality of $u_{j+1}$.
    \\
    Suppose instead that $\#\Z_{j+1}<\#\Z_{j}$. We observe that since $k\in (1/3, 1)$, we have $$\tempo<\frac{1}{4k}<\frac{1}{2-2k},$$ i.e. \eqref{situation:3} does not hold ``strictly'', as well. We can conclude as in the proof of Lemma~\ref{lemma: il minimizzante circondato da surfactant}-(i).\vspace{5pt}\\

    Now we prove case (ii). The fact that for every minimizer $u_{j+1}$ it holds $\# Z_{j+1}\ge \#Z_{j}$ can be proved as in (i).

    Now let $u_{j+1}$ be a minimizer of $\enF(\cdot, u_j).$ Suppose that $\#Z_{j+1}>\#Z_j.$ We observe that $Z_{j+1}\subset \partial^+I_{j+1}$, otherwise we have a contradiction by comparing $u_{j+1}$ with the competitor obtained replacing any element $p\in Z_{j+1}\setminus\partial^+I_{j+1}$ with $u_{j+1}(p)\mapsto -1.$ Under this assumption we prove that
    $\zero_{j}\not\subset Z_j$, which is the last statement of (ii). To do so, first we observe that $Z_{j+1}\subset \zero_{j+1}.$ Indeed otherwise, if there exists $p\in Z_{j+1}\setminus\zero_{j+1}$, then we consider the competitor $\tilde{u}_{j+1}$ obtained by replacing $u_{j+1}(p)\mapsto -1$ and we observe that, since   $\tempo = 1/4k$, it holds 
    \begin{equation}\label{eq:4.12}
        \E_\e(\tilde{u}_{j+1})-\E_\e(u_{j+1})+\frac{\e^\gamma}{\tau}\Bigl(\dis^0_\e(\tilde{u}_{j+1}, u_{j})-\dis^0_\e(u_{j+1}, u_{j})\Bigl)\le 2\e-2\e(1-k)-\frac{\e}{\tempo} < 0,
    \end{equation}
    which contradicts the minimality of $u_{j+1}$ (the first inequality of \eqref{eq:4.12} is an equality in case $\#\nn(p)\cap Z_{j+1} = 2$). Now, suppose by contradiction that $\zero_j\subset Z_j$. Since $\#Z_{j+1}>\#Z_j$, and since $Z_{j+1}\subset \zero_{j+1},$ we deduce that there exists $p\in\zero_{j+1}\setminus\zero_j,$ and therefore, in particular, $p\in\zero_{j+1}\cap I_j$. We define the competitor $\tilde{u}_{j+1}$ by replacing $u_{j+1}(p)\mapsto 1$, so that $\dis^1_\e(\tilde{u}_{j+1}, u_j)<\dis^1_\e(u_{j+1}, u_j).$ Moreover it holds 
    \begin{equation*}
        \E_\e(\tilde{u}_{j+1})-\E_\e(u_{j+1})+\frac{\e^\gamma}{\tau}\Bigl(\dis^0_\e(\tilde{u}_{j+1}, u_{j})-\dis^0_\e(u_{j+1}, u_{j})\Bigl)\le 4\e-4\e(1-k)-\frac{\e}{\tempo} = 0,
    \end{equation*}
    therefore $\enF(\tilde{u}_{j+1},u_j)<\enF(u_{j+1},u_j),$ which contradicts the minimality of $u_{j+1}$. 

    Now we show that for every minimizer $u_{j+1}$ there exists an other minimizer $\tilde{u}_{j+1}$ such that $\#\{\tilde{u}_{j+1} = 0\} = \#Z_{j}$ and $\{\tilde{u}_{j+1} = 1\} = \{u_{j+1} = 1\}.$
    Consider a minimizer $u_{j+1}$, and let $n:=\#Z_{j+1}-\#Z_j$.    
    We may suppose $n>0$, since otherwise there is nothing to prove. We already observed that in this case it holds $Z_{j+1}\subset\zero_{j+1}.$ Therefore we consider points $\{p^1, \dots, p^n\}\subset\zero_{j+1}\cap Z_{j+1},$ and we define the competitor $\tilde{u}_{j+1}$ by replacing $u_{j+1}(p^i)\mapsto -1$ for every $i=1, \dots, n$. Since $\tempo = 1/4k$ we have that $\enF(\tilde{u}_{j+1},u_j) = \enF(u_{j+1},u_j)$.
    
    Finally, we have to prove that for every minimizer $u_{j+1}$ the set $I_{j+1}$ is a quasi-octagon. In view of the previous argument, we only need to show the claim for any minimizer $u_{j+1}$ such that $\#Z_{j+1} = \#Z_j$. But in this case, the proof follows arguing exactly as in Lemma~\ref{lemma:5_optimal_shape_not_surrounded}-(i).
\end{proof}

\begin{remark}\label{equality_in_1}
Suppose that $\displaystyle \tempo=\frac{1}{4k}$. In view of Proposition~\ref{prop:situation_1_not_occurs}-(ii), at each minimization step $j$ the shape of $I_{j+1}$ for any minimizer $u_{j+1}$ can be determined by following the analysis of Section~\ref{sec:gamma<2}, in particular by applying Propositions~\ref{prop:5.3_beginning_movement}, \ref{prop:5.21}, \ref{prop:sides_movement_when_set_remains_octagon}, and \ref{prop:minimizing_movement_null_diagonal}. However, the limit flow depends crucially on whether the quantity $\#Z_j^\varepsilon$ remains constant with respect to $j$ along the minimizing movements. Indeed, if $\#Z_j^\varepsilon$ is allowed to vary from step to step, then the convergence $\varepsilon\,\#Z_j^\varepsilon\to\lambda$ as $\varepsilon\to 0$, assumed in Section~\ref{sec:gamma<2}, would no longer hold. 

One possible way to overcome this difficulty is to require that, at each time step, one selects a ``good minimizer'' $u_{j+1}$ satisfying $\#Z_{j+1}=\#Z_j$. With this choice, the analysis of the limit flow coincides with that carried out in Section~\ref{sec:gamma<2}.

Finally, we observe that if the initial data are prescribed as a family of functions $(u_0^\varepsilon)_\varepsilon$ such that $\#Z_0^\varepsilon\ge \#\zero_0^\varepsilon$ for every $\varepsilon$, then Proposition~\ref{prop:situation_1_not_occurs}-(ii) ensures that $\#Z_{j+1}=\#Z_j$ at each time step. As a consequence, the analysis of Section~\ref{sec:gamma<2} applies once again.
\end{remark}

\noindent{\bf The effect of \eqref{situation:2} on the minimizing movements not surrounded by surfactant.}
In the following proposition we show that if case \eqref{situation:2} holds, then starting from an initial datum which is not surrounded by surfactant, it holds that the minimizing sets are surrounded by surfactant for every $j\ge 1$.

\begin{proposition}
    \label{prop:il minimizzante circondato da surfactant}
    Let $\gamma=2$ and suppose that $\zeta >\frac{1}{3k-1}$. Let $u_0\in\A_\e$ be satisfying assumption \eqref{H} in Subsection \ref{subsub_H} and such that $Z_0\subsetneq \partial^+I_0$. Let $u_{1}$ be a minimizer of $\enF(\cdot, u_0)$.
    Then there exists $\bar{\e}$ such that for every $\e\le\bar{\e}$ it holds $\partial^+\I_{1}\subset \Z_{1}$. 
    In particular the subsequent evolution depends only on whether or not the condition $\zeta>\frac{1}{2-2k}$ is verified.
\end{proposition}

\begin{proof}
    Let $\mu<1/4$. By Proposition \ref{prop:connectedness} there exists $\bar{\e}$ such that for every $\e\le\bar{\e}$ the set $\I_{1}$ is a connected staircase set contained in $\I_0$ such that $d_\hh(\partial A_0, \partial A_1)\le\e^\mu$. We argue by contradiction. Suppose that $\partial^+I_{1}\not\subset Z_{1}$. Then, we observe that either there exists $p\in\one_{1}$ such that $\#\nn(p)\cap\{u_{1} = -1\}\ge 1$, or there exists $p\in\partial^+I_{1}\cap\{u_{1} = -1\}$ such that $\#\nn(p)\cap Z_{1}\ge 1.$ In both cases we build a competitor $\tilde{u}_{1}$ by replacing $u_{1}(p)\mapsto 0$. Since  \eqref{situation:2} holds, we have 
    \[\Bigl(\E_\e(\tilde{u}_{1})-\E_\e(u_{1})\Bigl)+\frac{\e^\gamma}{\tau}\bigl(\dis^0_\e(\tilde{u}_{1}, u_0)+\dis^0_\e(u_{1}, u_0)\bigl)\le \Bigl(-2\e+3\e(1-k)\Bigl)+\frac{\e}{\tempo}<0.\]
     Moreover since in the first case the variation of the dissipation is of order $\e^{1+\mu}\ll\e$, we deduce $\enF(\tilde{u}_{1}, u_0)<\enF(u_{1}, u_0),$ which contradicts the minimality of $u_{j+1}.$ The last part of the Proposition is a direct consequence of \ref{lemma: il minimizzante circondato da surfactant}
    \end{proof}

In view of the previous proposition and the last two paragraphs, the only remaining case to analyse is that of the minimizing movements under the assumptions that \eqref{situation:2} does not hold, \eqref{situation:1} does hold, and the set $I_j$ is not surrounded by surfactant. Since in this case \eqref{situation:1} holds, in view of the discussion at the beginning of this section, for simplicity we may suppose that the initial datum $u_0$ verifies $\zero_0\subset\Z_0$.

\begin{proposition}\label{lemma:ultimo}
    Let $\gamma=2$, suppose that $\zeta >\frac{1}{4k}$ and that $\tempo\le\frac{1}{3k-1}$. Let $u_j\in\A_\e$ be satisfying assumption \eqref{H} of Subsection \ref{subsub_H} and such that $I_j = \tilde{I}_j\cup(\cup_{i=1}^4\PP_{j, i})$ is a quasi-octagon with $\min\{\tilde{D}_{j, i}:i =1, \dots, 4\}\ge\bar{c}$ (where $\bar{c}$ is the constant of \eqref{H}), and such that \begin{equation}\label{eq:boundary_cardinality_bound}\#\Z_j<\#\partial^+\I_j-8\e^{\mu-1}.
    \end{equation} Then there exists $\bar{\e}$ such that for every $\e\le\bar{\e}$ the set $\I_{j+1}$ is a quasi-octagon such that $\Z_{j+1}\subset\partial^+\I_{j+1}$ and $\#\Z_{j+1}\geq \#\Z_{j}$.

    Suppose moreover that \begin{equation}\label{eq:5.32_uguale}
        \sum_{i=1}^4\Biggl(\frac{D_{j, i}}{\sqrt{2}\e}+\bmax_{j, i}-2\amin_{j, i}\Biggl)\le \#Z_j-2,
    \end{equation}
    (which is the same as \eqref{eq:5.32}), where $\amin_{j, i}$ and $\bmax_{j, i}$ are defined as in \eqref{eq:alpha_bar}. 
    Then the following properties hold true.
    \begin{itemize}
        \item[(i)] If $\tempo<\frac{1}{3k-1}$ then $\#Z_{j+1} = \#Z_j.$ In particular Proposition \ref{prop:5.21} holds for the sides displacements. 
        \item[(ii)] If $\tempo = \frac{1}{3k-1}$ then for every minimizer $u_{j+1}$ there exists an other minimizer $\tilde{u}_{j+1}$ such that $\{\tilde{u}_{j+1} = 1\} = \{u_{j+1} = 1\}$ and $\#\{\tilde{u}_{j+1} = 0\} = \#Z_j$.
    \end{itemize}
\end{proposition}
\begin{proof}
Arguing as in the first step of Lemma~\ref{lemma:5_optimal_shape_not_surrounded} we can prove that $\#\partial^+I_{j+1}>\#Z_j$. It follows that $Z_{j+1}\subset\partial^+I_{j+1}$. Indeed otherwise from Remark~\ref{rem:blu_on_the_boundary} we would have $\partial^+I_{j+1}\subsetneq Z_{j+1}$ and $\#Z_{j+1}>\#\partial^+I_{j+1}>\#Z_j.$ In particular there would exist $p\in Z_{j+1}\setminus \partial^+I_{j+1}$. We can define a competitor $\tilde{u}_{j+1}$ by replacing $u_{j+1}(p)\mapsto -1$: since this transformation strictly reduces the energy $\enF(\cdot)$ we have a contradiction.

Knowing that $Z_{j+1}\subset I_{j+1},$ the proof that $I_{j+1}$ is a quasi-octagon follows as in the second step of the proof of Lemma~\ref{lemma:5_optimal_shape_not_surrounded}. Moreover, again as in Lemma~\ref{lemma:5_optimal_shape_not_surrounded}, we can prove that there exists a constant $c$ depending only on $\bar{c}$ such that $\beta_{j, i}\le c$ for all $i$.  Therefore, from \eqref{eq:boundary_cardinality_bound}, we deduce that 
    \begin{equation*}
    \label{eq:7.1}
              \#\partial^+I_j-c\le\#\partial^+I_{j+1}\le\#\partial^+I_{j},
    \end{equation*}
          for a positive constant $c>0$ which we do not rename. We prove that $\#\Z_{j+1}\geq \#\Z_{j}$. To this end we argue by contradiction. Let us assume that $\#\Z_{j+1}<\#\Z_{j}$. Since $\#Z_{j+1}<\#\partial^+I_{j+1}$, we deduce that either there exists $\overline{p} \in \one_{j+1}$ such that $\#(\mathcal{N}(\overline{p})\cap \{u_{j+1}=-1\})=1$ and $\#(\mathcal{N}(\overline{p})\cap Z_{j+1})=1$, or there exists $\bar{p}\in\partial^+I_{j+1}\cap\{u_{j+1} = -1\}$ such that $\#(\nn(\bar{p})\cap Z_{j+1}) = 1$. In both cases we define a competitor $\tilde{u}_{j+1}$ replacing $u_{j+1}(\overline{p})\mapsto 0$ (see Figure \eqref{fig:ultimo_disegno}). 
        \begin{figure}[H]
            \centering
            \resizebox{0.35
            \textwidth}{!}{\tikzset{every picture/.style={line width=0.75pt}} 

\begin{tikzpicture}[x=0.75pt,y=0.75pt,yscale=-1,xscale=1]

\draw [color={rgb, 255:red, 208; green, 2; blue, 27 }  ,draw opacity=1 ][fill={rgb, 255:red, 208; green, 2; blue, 27 }  ,fill opacity=1 ]   (120.44,100) -- (170.44,100) -- (170.44,110) -- (180.44,110) -- (180.44,120) -- (190.44,120) -- (190.44,130) -- (200.44,130) -- (200.44,140) -- (210.44,140) -- (210.44,180) -- (190.44,180) -- (190.44,150) -- (190.44,140) -- (180.44,140) -- (180.44,130) -- (170.44,130) -- (170.44,120) -- (160.44,120) -- (120.44,120) -- cycle ;
\draw [color={rgb, 255:red, 208; green, 2; blue, 27 }  ,draw opacity=1 ][fill={rgb, 255:red, 208; green, 2; blue, 27 }  ,fill opacity=1 ]   (250.44,100) -- (290.44,100) -- (290.44,110) -- (300.44,110) -- (310.44,110) -- (310.44,120) -- (320.44,120) -- (320.44,130) -- (330.44,130) -- (330.44,140) -- (340.44,140) -- (340.44,180) -- (320.44,180) -- (320.44,150) -- (320.44,140) -- (310.44,140) -- (310.44,130) -- (300.44,130) -- (300.44,120) -- (290.44,120) -- (250.44,120) -- cycle ;
\draw  [color={rgb, 255:red, 126; green, 211; blue, 33 }  ,draw opacity=1 ][fill={rgb, 255:red, 126; green, 211; blue, 33 }  ,fill opacity=1 ] (160.44,90) -- (170.44,90) -- (170.44,100) -- (160.44,100) -- cycle ;
\draw  [color={rgb, 255:red, 74; green, 144; blue, 226 }  ,draw opacity=1 ][fill={rgb, 255:red, 74; green, 144; blue, 226 }  ,fill opacity=1 ] (170.44,100) -- (180.44,100) -- (180.44,110) -- (170.44,110) -- cycle ;
\draw  [color={rgb, 255:red, 74; green, 144; blue, 226 }  ,draw opacity=1 ][fill={rgb, 255:red, 74; green, 144; blue, 226 }  ,fill opacity=1 ] (180.44,110) -- (190.44,110) -- (190.44,120) -- (180.44,120) -- cycle ;
\draw  [color={rgb, 255:red, 74; green, 144; blue, 226 }  ,draw opacity=1 ][fill={rgb, 255:red, 74; green, 144; blue, 226 }  ,fill opacity=1 ] (190.44,120) -- (200.44,120) -- (200.44,130) -- (190.44,130) -- cycle ;
\draw  [color={rgb, 255:red, 74; green, 144; blue, 226 }  ,draw opacity=1 ][fill={rgb, 255:red, 74; green, 144; blue, 226 }  ,fill opacity=1 ] (200.44,130) -- (210.44,130) -- (210.44,140) -- (200.44,140) -- cycle ;
\draw  [color={rgb, 255:red, 126; green, 211; blue, 33 }  ,draw opacity=1 ][fill={rgb, 255:red, 126; green, 211; blue, 33 }  ,fill opacity=1 ] (290.44,90) -- (300.44,90) -- (300.44,100) -- (290.44,100) -- cycle ;
\draw  [color={rgb, 255:red, 74; green, 144; blue, 226 }  ,draw opacity=1 ][fill={rgb, 255:red, 74; green, 144; blue, 226 }  ,fill opacity=1 ] (300.44,100) -- (310.44,100) -- (310.44,110) -- (300.44,110) -- cycle ;
\draw  [color={rgb, 255:red, 74; green, 144; blue, 226 }  ,draw opacity=1 ][fill={rgb, 255:red, 74; green, 144; blue, 226 }  ,fill opacity=1 ] (290.44,100) -- (300.44,100) -- (300.44,110) -- (290.44,110) -- cycle ;
\draw  [color={rgb, 255:red, 74; green, 144; blue, 226 }  ,draw opacity=1 ][fill={rgb, 255:red, 74; green, 144; blue, 226 }  ,fill opacity=1 ] (310.44,110) -- (320.44,110) -- (320.44,120) -- (310.44,120) -- cycle ;
\draw  [color={rgb, 255:red, 74; green, 144; blue, 226 }  ,draw opacity=1 ][fill={rgb, 255:red, 74; green, 144; blue, 226 }  ,fill opacity=1 ] (320.44,120) -- (330.44,120) -- (330.44,130) -- (320.44,130) -- cycle ;
\draw  [color={rgb, 255:red, 74; green, 144; blue, 226 }  ,draw opacity=1 ][fill={rgb, 255:red, 74; green, 144; blue, 226 }  ,fill opacity=1 ] (330.44,130) -- (340.44,130) -- (340.44,140) -- (330.44,140) -- cycle ;
\draw  [draw opacity=0] (100.44,50) -- (361.78,50) -- (361.78,210.33) -- (100.44,210.33) -- cycle ; \draw  [color={rgb, 255:red, 155; green, 155; blue, 155 }  ,draw opacity=0.77 ] (100.44,50) -- (100.44,210.33)(110.44,50) -- (110.44,210.33)(120.44,50) -- (120.44,210.33)(130.44,50) -- (130.44,210.33)(140.44,50) -- (140.44,210.33)(150.44,50) -- (150.44,210.33)(160.44,50) -- (160.44,210.33)(170.44,50) -- (170.44,210.33)(180.44,50) -- (180.44,210.33)(190.44,50) -- (190.44,210.33)(200.44,50) -- (200.44,210.33)(210.44,50) -- (210.44,210.33)(220.44,50) -- (220.44,210.33)(230.44,50) -- (230.44,210.33)(240.44,50) -- (240.44,210.33)(250.44,50) -- (250.44,210.33)(260.44,50) -- (260.44,210.33)(270.44,50) -- (270.44,210.33)(280.44,50) -- (280.44,210.33)(290.44,50) -- (290.44,210.33)(300.44,50) -- (300.44,210.33)(310.44,50) -- (310.44,210.33)(320.44,50) -- (320.44,210.33)(330.44,50) -- (330.44,210.33)(340.44,50) -- (340.44,210.33)(350.44,50) -- (350.44,210.33)(360.44,50) -- (360.44,210.33) ; \draw  [color={rgb, 255:red, 155; green, 155; blue, 155 }  ,draw opacity=0.77 ] (100.44,50) -- (361.78,50)(100.44,60) -- (361.78,60)(100.44,70) -- (361.78,70)(100.44,80) -- (361.78,80)(100.44,90) -- (361.78,90)(100.44,100) -- (361.78,100)(100.44,110) -- (361.78,110)(100.44,120) -- (361.78,120)(100.44,130) -- (361.78,130)(100.44,140) -- (361.78,140)(100.44,150) -- (361.78,150)(100.44,160) -- (361.78,160)(100.44,170) -- (361.78,170)(100.44,180) -- (361.78,180)(100.44,190) -- (361.78,190)(100.44,200) -- (361.78,200)(100.44,210) -- (361.78,210) ; \draw  [color={rgb, 255:red, 155; green, 155; blue, 155 }  ,draw opacity=0.77 ]  ;
\draw    (220.44,140) -- (248.44,140) ;
\draw [shift={(250.44,140)}, rotate = 180] [color={rgb, 255:red, 0; green, 0; blue, 0 }  ][line width=0.75]    (10.93,-3.29) .. controls (6.95,-1.4) and (3.31,-0.3) .. (0,0) .. controls (3.31,0.3) and (6.95,1.4) .. (10.93,3.29)   ;
\draw    (120.44,100) -- (170.44,100) -- (170.44,110) -- (180.44,110) -- (180.44,120) -- (190.44,120) -- (190.44,130) -- (200.44,130) -- (200.44,140) -- (210.44,140) -- (210.44,180) ;
\draw    (250.44,100) -- (290.44,100) -- (290.44,110) -- (310.44,110) -- (310.44,120) -- (320.44,120) -- (320.44,130) -- (330.44,130) -- (330.44,140) -- (340.44,140) -- (340.44,180) ;

\draw (160,99.4) node [anchor=north west][inner sep=0.75pt]  [font=\footnotesize]  {$\overline{p}$};
\draw (291,99.4) node [anchor=north west][inner sep=0.75pt]  [font=\footnotesize]  {$\overline{p}$};
\draw (132.44,59.4) node [anchor=north west][inner sep=0.75pt]    {$u_{j+1}$};
\draw (282.44,56.4) node [anchor=north west][inner sep=0.75pt]    {$\tilde{u}_{j+1}$};

\end{tikzpicture}}
         \caption{We consider $\overline{p}\in \one_{j+1}$ such that $\#(\mathcal{N}(\overline{p})\cap \{u_{j+1}=-1\})=1$ and $\#(\mathcal{N}(\overline{p})\cap Z_{j+1})=1$ and on the right we pictured the configuration that minimize the total energy.}
        \label{fig:ultimo_disegno}
        \end{figure}
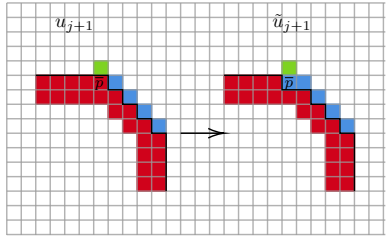
\noindent We get $$\frac{1}{\e}\Bigl|\dis^1_\e(\tilde{u}_{j+1},u_j)-\dis^1_\e(u_{j+1},u_j)\Bigl|\le c\e^2$$ 
and 
$$\E_\e(\tilde{u}_{j+1})-\E_\e(u_{j+1})+\frac{\e}{\tau}\left(\dis^0_\e(\tilde{u}_{j+1},u_j)-\dis^0_\e(u_{j+1},u_j)\right)=3\e(1-k)-2\e-\frac{\e}{\tempo}<0.$$
In particular we deduce $\enF(\tilde{u}_{j+1})<\enF(u_{j+1}),$ which contradicts the minimality of $u_{j+1}$.

We suppose now that \eqref{eq:5.32_uguale} holds and we prove (i). Arguing as in the proof of Proposition \ref{prop:5.21} we have that $\#\zero_{j+1}\le \#Z_j-2.$ Suppose by contradiction that $\#Z_{j+1}>\#Z_j$. Then there exists $p\in Z_{j+1}\setminus \zero_{j+1}.$ As we showed in Lemma~\ref{lemma:surfactant placement_around_quasi_octagon}-(i) there exists $p\in\ Z_{j+1}\setminus\zero_{j+1}$ such that $\#\nn(p)\cap Z_{j+1} = 1$. We define a competitor $\tilde{u}_{j+1}$ by replacing $u_{j+1}(p)\mapsto -1$. Since \eqref{situation:2} does not hold we find $\enF(\tilde{u}_{j+1})<\enF(u_{j+1}),$ which contradicts the minimality of $u_{j+1}.$\vspace{5pt}\\
Finally we prove (ii). Let $u_{j+1}$ be a minimizer such that $\#Z_{j+1}>\#Z_j$. Arguing as before we have that there exists $p\in\ Z_{j+1}\setminus\zero_{j+1}$ such that $\#\nn(p)\cap Z_{j+1} = 1$. We define a competitor $\tilde{u}_{j+1}$ by replacing $u_{j+1}(p)\mapsto -1$, and we observe that in this case, since $\tempo = \frac{1}{3k-1}$ it holds $\enF(\tilde{u}_{j+1}, u_j) = \enF(u_{j+1}, u_j),$ therefore $\tilde{u}_{j+1}$ is a minimizer such that $\#\{\tilde{u}_{j+1} = 0\} = \#\{u_{j+1} = 0\}-1$. Inductively we prove the claim.
\end{proof}

\begin{remark}
    In the case $\gamma=2$ with $\tempo\le \frac{1}{3k-1}$, one can prove a result analogous to Theorem~\ref{teo:minimizing_movements_surfactant_covering_more_than_diagonals}, showing that the minimizing movements $(u_j^\varepsilon)_j$ converge to a limit flow. However, describing the side velocities in the limit flow is more delicate. The main difficulty is that, since the amount of surfactant may increase along a minimizing movement $j\mapsto u_j^\varepsilon$, the convergence $\varepsilon\,\#Z_j^\varepsilon\to\lambda\in\rr$ as $\varepsilon\to 0$ no longer holds. In particular, we cannot identify the qualitatively different stages of the limit flow solely from the lengths of the diagonal sides of the octagon in the limit flow, as was done in Theorem~\ref{teo:minimizing_movements_surfactant_covering_more_than_diagonals}.
\end{remark}

\end{document}